\documentclass[a4paper, reqno]{amsart}
\usepackage{indentfirst}
\usepackage{amsmath}
\usepackage{amsthm}
\usepackage{amssymb}
\usepackage{tikz}
\usepackage{adjustbox}
\usepackage{hyperref}
\usepackage{cleveref}
\usepackage{tikz-cd}
\usepackage{quiver}
\usepackage{xurl}

\usepackage{stackengine} 

\usepackage[T1]{fontenc}
\usepackage[outline]{contour}
\usepackage{xcolor}

\contourlength{0.3pt}
\contournumber{15}

\usepackage{scalerel}
\usepackage{dsfont}
\usepackage{txfonts} 
\usepackage{bbm} 
\DeclareSymbolFont{bbold}{U}{bbold}{m}{n}
\DeclareSymbolFontAlphabet{\mathbbold}{bbold}
\usepackage{mathrsfs} 
\usepackage[scr=boondox]{mathalfa} 

\DeclareFontFamily{U}{mathc}{}
\DeclareFontShape{U}{mathc}{m}{it}%
{<->s*[1.03] mathc10}{}
\DeclareMathAlphabet{\mathccal}{U}{mathc}{m}{it}

\input pdf-trans
\newbox\qbox
\def\usecolor#1{\csname\string\color@#1\endcsname\space}
\newcommand\bordercolor[1]{\colsplit{1}{#1}}
\newcommand\fillcolor[1]{\colsplit{0}{#1}}
\newcommand\colsplit[2]{\colorlet{tmpcolor}{#2}\edef\tmp{\usecolor{tmpcolor}}%
	\def\tmpB{}\expandafter\colsplithelp\tmp\relax%
	\ifnum0=#1\relax\edef\fillcol{\tmpB}\else\edef\bordercol{\tmpC}\fi}
\def\colsplithelp#1#2 #3\relax{%
	\edef\tmpB{\tmpB#1#2 }%
	\ifnum `#1>`9\relax\def\tmpC{#3}\else\colsplithelp#3\relax\fi
}
\newcommand\outline[1]{\leavevmode%
	\def\maltext{#1}%
	\setbox\qbox=\hbox{\maltext}%
	\boxgs{Q q 2 Tr \thickness\space w \fillcol\space \bordercol\space}{}%
	\copy\qbox%
}
\newcommand\mathprecalbb[2][1]{%
	\ThisStyle{%
		\stackengine{0pt}{\def\thickness{.6}\outline{$\SavedStyle\mathcal{#2}$}}{\kern.1pt\outline{$\SavedStyle\mathcal{#2}$}}{O}{l}{F}{F}{L}}}
\bordercolor{black}
\fillcolor{white}
\def\thickness{.1}

\input pdf-trans
\newbox\qbox
\newcommand\mathprebbb[2][1]{%
	\ThisStyle{%
		\stackengine{0pt}{\def\thickness{.5}\outline{$\SavedStyle\mathrm{#2}$}}{\kern.01pt\outline{$\SavedStyle\mathrm{#2}$}}{O}{l}{F}{F}{L}}}
\bordercolor{black}
\fillcolor{white}
\input pdf-trans
\newbox\qbox
\newcommand\mathbbbnobracket[2][1]{%
	\ThisStyle{%
		\stackengine{0pt}{\def\thickness{.2}\outline{$\SavedStyle\mathprebbb{#2}$}}{\kern.6pt{\def\thickness{.2}\outline{$\SavedStyle\mathrm{#2}$}}}{O}{l}{F}{F}{L}}}
\bordercolor{black}
\fillcolor{white}
\newcommand*\mathmybb[1]{{\mathprebbb{#1}}}
\newcommand*\mathbbb[1]{{\mathbbbnobracket{#1}}}

\DeclareMathAlphabet{\matheus}{U}{eus}{m}{n}
\SetMathAlphabet{\matheus}{bold}{U}{eus}{b}{n}

\DeclareFontFamily{U}{dmjhira}{}
\DeclareFontShape{U}{dmjhira}{m}{n}{ <-> dmjhira }{}
\DeclareRobustCommand{\yo}{\text{\usefont{U}{dmjhira}{m}{n}\symbol{"48}}}

\newtheorem{defi}{Definition}[subsection]

\newtheorem{theorem}[defi]{Theorem}

\AfterEndEnvironment{maintheorem}{\noindent\ignorespaces}
\newtheorem{coroll}[defi]{Corollary}
\newtheorem{pro}[defi]{Proposition}
\newtheorem{lemma}[defi]{Lemma}
\theoremstyle{definition}
\newtheorem{construct}[defi]{Construction}
\newtheorem{example}[defi]{Example}
\newtheorem{innercustomgeneric}{\customgenericname}
\providecommand{\customgenericname}{}
\newcommand{\newcustomtheorem}[2]{%
	\newenvironment{#1}[1]
	{%
		\renewcommand\customgenericname{#2}%
		\renewcommand\theinnercustomgeneric{##1}%
		\innercustomgeneric
	}
	{\endinnercustomgeneric}
}
\newcustomtheorem{customdef}{Definition}
\newcustomtheorem{customthm}{Theorem}
\newcustomtheorem{custompro}{Proposition}
\newcustomtheorem{customlem}{Lemma}
\newcustomtheorem{custommainthm}{Main Theorem}
\newcustomtheorem{customeg}{Example}
\newcustomtheorem{customconjecture}{Conjecture}
\newcustomtheorem{customidea}{Idea}
\newcustomtheorem{customcoroll}{Corollary}

\theoremstyle{remark}
\newtheorem{rk}[defi]{Remark}
\newtheorem{nota}[defi]{Notation}
\newcustomtheorem{customrk}{Remark}

\newcommand{\bbA}{\mathmybb{A}}
\newcommand{\scrA}{\mathscr{A}}

\newcommand{\calA}{\mathccal{A}}
\newcommand{\twoA}{\mathccal{A}}
\newcommand{\bbB}{\mathmybb{B}}

\newcommand{\calB}{\mathccal{B}}
\newcommand{\twoB}{\mathccal{B}}

\newcommand{\bbK}{\mathmybb{K}}
\newcommand{\calK}{\mathccal{K}}
\newcommand{\calJ}{\mathccal{J}}
\newcommand{\calR}{\mathccal{R}}

\newcommand{\bbD}{\mathmybb{D}}

\newcommand{\F}{\mathscr{F}}
\newcommand{\bbF}{\mathmybb{F}}

\newcommand{\G}{\mathccal{G}}

\newcommand{\V}{\mathccal{V}}

\newcommand{\rscc}[2][]{\beta_{#2^{#1}}}
\newcommand{\gps}[1]{{#1}\text{-}\mathrm{Space}}
\newcommand{\pgps}[1]{{#1}\text{-}\mathrm{Space}\text{-}\mathrm{Pro}}

\newcommand*{\Set}{\mathbf{Set}}
\newcommand*{\Cat}{\mathccal{Cat}}

\newcommand*{\FCat}{\mathbf{\F}\text{-}\mathccal{Cat}}

\newcommand*{\sSet}{\mathbf{Set}_\Delta}

\newcommand*{\TAlg}{\mathrm{T}\text{-}\mathrm{Alg}}
\newcommand*{\FTAlg}{\mathrm{T}\text{-}\mathmybb{Alg}}

\newcommand{\sigmalim}[2]{\mathrm{m}_{#1}\mathrm{lim\;}{#2}}

\newcommand{\dotlim}[2]{\mathrm{d}_{#1}\mathrm{lim\;}{#2}}

\newcommand{\bbEl}[1]{\mathmybb{El}{#1}}
\newcommand{\Lan}[2][]{\mathrm{Lan}_{#2}{#1}}

\newcommand{\Cone}[1]{\mathrm{Cone}_{#1}}

\newcommand*{\shortbar}{\scalebox{1.6}[1]{-}}

\definecolor{turquoise}{HTML}{00B4CE}

\usepackage{lmodern}
\usepackage{mathtools,enumitem}
\usepackage{microtype}
\usepackage{graphicx}

\usetikzlibrary{arrows,calc,decorations.pathmorphing}
\tikzset{
	marked/.style = {decoration = {markings, mark = at position 0.5 with { 
				\node[transform shape, xscale = .8, yscale=.4] {/}; } }, postaction = 
		{decorate} },
	circled/.style = {decoration = {markings, mark = at position 0.5 with { 
				\node[transform shape, scale = .7] {$\circ$}; } }, postaction = {decorate} 
	},
	loose/.style ={->,
		decorate,
		decoration={snake,amplitude=.4mm,segment length=2mm,post length=1mm}},
}

\tikzcdset{scale cd/.style={every label/.append style={scale=#1},
		cells={nodes={scale=#1}}}}

\tikzcdset{every label/.append style = {font = \footnotesize}} 

\usetikzlibrary{decorations.markings,decorations.pathmorphing,shapes}
\tikzset{double line with arrow/.style 
	args={#1,#2}{decorate,decoration={markings,%
			mark=at position 0 with {\coordinate (ta-base-1) at (0,1pt);
				\coordinate (ta-base-2) at (0,-1pt);},
			mark=at position 1 with {\draw[#1] (ta-base-1) -- (0,1pt);
				\draw[#2] (ta-base-2) -- (0,-1pt);
}}}}

\tikzset{dar/.style={double,double equal sign distance,-implies}}
\tikzset{mid/.style={anchor=mid}} 

\makeatletter
\newbox\dottedarrow@box
\setbox\dottedarrow@box\hbox
{%
	\begin{tikzpicture}
		\draw[shorten <=0.2cm, shorten >= 0.15cm , dotted,->] (-0.1,0) -- 
		(1.8em,0);
	\end{tikzpicture}%
}
\newcommand*\dottedarrow
{\relax\ifmmode\expandafter\dottedarrow@m\else\expandafter\dottedarrow@t\fi}
\newcommand*\dottedarrow@t[1][1.3em]
{\resizebox{#1}{!}{\raisebox{.5ex}{\usebox\dottedarrow@box}}}
\newcommand*\dottedarrow@m[1][]
{%
	\if\relax\detokenize{#1}\relax
	\mathchoice
	{\dottedarrow@t}
	{\dottedarrow@t}
	{\dottedarrow@t[0.5em]}
	{\dottedarrow@t[0.5em]}%
	\else
	\dottedarrow@t[#1]%
	\fi
}

\newcommand*\Neginternal[3]{\mathpalette\Neg@{{#1}{#2}{#3}}}
\newcommand*\Neg@[2]{\Neg@@{#1}#2}
\newcommand*\Neg@@[4]{%
	\mathrel{\ooalign{%
			$\m@th#1#4$\cr
			\hidewidth$\m@th#3{#1}\mkern\muexpr#2*2$\hidewidth\cr
	}}%
}
\newcommand*\mynegslash[1]{\rotatebox[origin=c]{60}{$\m@th#1\shortbar$}}
\newcommand*\amsnegslash[1]{\rotatebox[origin=c]{60}{$\m@th#1-$}}

\makeatother

\newcounter{sarrow}
\newcommand\xleadsto[1]{%
	\stepcounter{sarrow}%
	\mathrel{\begin{tikzpicture}[decoration={snake, amplitude=.4mm,segment length=2mm}]
			\node (\thesarrow) {$\scriptstyle #1$};
			\draw[->,decorate] (\thesarrow.south west) -- (\thesarrow.south 
			east);
		\end{tikzpicture}
	}
}

\renewcommand{\leadsto}{\xleadsto{\mkern9mu}}

\usepackage{graphicx}

\makeatletter
\providecommand*{\twoheadrightarrowfill@}{%
	\arrowfill@\relbar\relbar\twoheadrightarrow
}
\providecommand*{\twoheadleftarrowfill@}{%
	\arrowfill@\twoheadleftarrow\relbar\relbar
}
\providecommand*{\xtwoheadrightarrow}[2][]{%
	\ext@arrow 0579\twoheadrightarrowfill@{#1}{#2}%
}
\providecommand*{\xtwoheadleftarrow}[2][]{%
	\ext@arrow 5097\twoheadleftarrowfill@{#1}{#2}%
}
\makeatother

\usetikzlibrary{graphs,graphs.standard,quotes}

\setlist[enumerate,1]{label=\textup{(\arabic*)}}
\setlist[enumerate,2]{label=\textup{(\alph*)}}

\newcommand{\longref}[2]{\hyperref[#2]{#1~\textup{\ref*{#2}}}}
\newcommand{\eqtref}[2]{\hyperref[#2]{#1~\textup{(\ref*{#2})}}}

\newcommand*{\op}{\mathrm{op}}

\newcommand*{\FMod}[1]{\mathmybb{Mod}_{#1}}

\DeclareMathOperator*{\colim}{colim}

\DeclareMathOperator{\ob}{ob}
\DeclareMathOperator{\mor}{mor}

\numberwithin{equation}{section}

\newcounter{subsubsubsection}[subsubsection]

\makeatletter
\renewcommand\paragraph{\@startsection{paragraph}{5}{\z@}%
	{3.25ex \@plus1ex \@minus.2ex}%
	{-1em}%
	{\normalfont\normalsize\bfseries}}
\renewcommand\subparagraph{\@startsection{subparagraph}{6}{\parindent}%
	{3.25ex \@plus1ex \@minus .2ex}%
	{-1em}%
	{\normalfont\normalsize\bfseries}}
\def\toclevel@subsubsubsection{4}
\def\toclevel@paragraph{5}
\def\toclevel@paragraph{6}
\def\l@subsubsubsection{\@dottedtocline{4}{7em}{4em}}
\def\l@paragraph{\@dottedtocline{5}{10em}{5em}}
\def\l@subparagraph{\@dottedtocline{6}{14em}{6em}}
\makeatother

\makeatletter
\newcommand\makebig[2]{%
	\@xp\newcommand\@xp*\csname#1\endcsname{\bBigg@{#2}}%
	\@xp\newcommand\@xp*\csname#1l\endcsname{\@xp\mathopen\csname#1\endcsname}%
	\@xp\newcommand\@xp*\csname#1r\endcsname{\@xp\mathclose\csname#1\endcsname}%
}
\makeatother
\makebig{Bigggg}{4.5}

\usepackage[margin=22.2mm]{geometry} 

\title[Models of sketches as algebras over monads]{Enhanced $2$-categories of models of sketches \\ as enhanced $2$-categories of algebras over monads}
\author{Joanna Ko}
\email{joanna.ko.maths@gmail.com}
\address{Topos Research UK, Urbanoid Workspace, 
	1 Kings Meadow, Osney Mead, Oxford OX2 0DP, United Kingdom}


\begin{document}
	
	\begin{abstract}
		We establish the equivalence between models of enhanced $2$-sketches and algebras over monads, including the (co)lax morphisms. More precisely, for any enhanced limit $2$-sketch $\mathmybb{T}$ with tight cones, the enhanced $2$-category $\FMod{s, w}(\mathmybb{T}, \bbK)$ of models of $\mathmybb{T}$ in a locally presentable enhanced $2$-category $\bbK$, in which the tight and the loose morphisms are the $\F$-natural transformations and the loose $w$-natural transformations, respectively, is equivalent to the enhanced $2$-category $\FTAlg_{s, w}$ of algebras over an enhanced $2$-monad $T$ on the models $\FMod{}(\mathccal{T}_\tau, \bbK)$ restricted to the tight morphisms  with strict $T$-morphisms and $w$-$T$-morphisms.
		
		As a consequence, we explore the limits in the enhanced $2$-category $\FMod{s, w}(\mathmybb{T}, \bbK)$ of models with loose $w$-natural transformations, and conclude that $\FMod{s, w}(\mathmybb{T}, \bbK)$ inherits all $w$-rigged limits.
		
		Along the way, we establish an enriched analogue of the Orthogonal Sub-category Theorem, and generalise results on the reflectivity and the monadicity of models of enriched limit sketches in the base of enrichment to any arbitrary locally presentable enriched category.
		%
	\end{abstract}
	\maketitle
	\tableofcontents
	
	\section{Introduction}
	\label{sec:intro}
	%
	%
	Many structures of interest in $2$-category theory have aspects that are inherently strict. For instance, a monoidal double category $\bbD$ requires both the source and the target $s, t \colon \bbD_1 \rightrightarrows \bbD_0$ to be \emph{strict} monoidal functors; yet, the unit $i \colon \bbD_0 \to \bbD_1$ and the composition $c \colon \bbD_1 \times_{\bbD_0} \bbD_1 \to \bbD_1$ are only required to be \emph{pseudo} monoidal functors. Unlike double categories, which can be described as pseudo-categories internal to the $2$-category $\Cat$ of categories,  monoidal double categories cannot be simply viewed as pseudo-categories internal to the $2$-category $\mathccal{MonCat}_p$ of monoidal categories and pseudo monoidal functors. Because of this, a mixture of strictness and pseudoness is involved in describing monoidal double categories.
	
	Moreover, many weak morphisms between structures that are strictly $2$-categorical still involve strictness implicitly. Suppose we are given a pair of monoidal categories $A$ and $B$, viewed as $2$-functors from the limit $2$-sketch $\mathccal{M}$ of pseudo-monoids to $\Cat$. Given that a strict monoidal functor $A \to B$ can be described as a $2$-natural transformation, it is natural to guess that a lax monoidal functor can be described as a lax natural transformation $\phi \colon A \to B$. However, a lax monoidal functor only requires lax naturality with respect to the unit and the multiplication. More precisely, the $2$-components corresponding to the product projections $\pi_1$ and $\pi_2$ need to be identities, i.e., the diagrams
	\[
	\begin{tikzcd}
		{A \times A} & {B \times B} \\
		A & B
		\arrow["{\phi \times \phi}", from=1-1, to=1-2]
		\arrow[""{name=0, anchor=center, inner sep=0}, "{(\pi_1)_A}"', from=1-1, to=2-1]
		\arrow[""{name=1, anchor=center, inner sep=0}, "{(\pi_1)_B}", from=1-2, to=2-2]
		\arrow["\phi"', from=2-1, to=2-2]
		\arrow["{=}"{description}, draw=none, from=0, to=1]
	\end{tikzcd}
	\hspace{4em}
	\begin{tikzcd}
		{A \times A} & {B \times B} \\
		A & B
		\arrow["{\phi\times \phi}", from=1-1, to=1-2]
		\arrow[""{name=0, anchor=center, inner sep=0}, "{(\pi_2)_A}"', from=1-1, to=2-1]
		\arrow[""{name=1, anchor=center, inner sep=0}, "{(\pi_2)_B}", from=1-2, to=2-2]
		\arrow["\phi"', from=2-1, to=2-2]
		\arrow["{=}"{description}, draw=none, from=0, to=1]
	\end{tikzcd}
	\]
	need to be strictly commutative. In other words, a lax monoidal functor has to be \emph{$2$-natural} with respect to the projections $\pi_1$ and $\pi_2$. This means that the notion of lax natural transformations fails to capture lax monoidal functors. 
	
	In \cite{ABK:2024}, \emph{enhanced limit $2$-sketche}s and \emph{loose $w$-natural transformations} between models of enhanced limit $2$-sketches are proposed to overcome the above obstacles. Roughly speaking, an enhanced limit $2$-sketch is a limit $2$-sketch with two types of $1$-morphisms: the \emph{tight} ones and the \emph{loose} ones; a loose $w$-natural transformation between models of enhanced limit $2$-sketches is a $w$-natural transformation that \emph{restricts} to a \emph{strict} $2$-natural transformation on tight morphisms.
	
	
	With these new concepts, the vast majority of $2$-dimensional structures and morphisms which cannot be captured through ordinary limit $2$-sketches or enriched transformations between models of enriched limit sketches may successfully be unified and described under the same roof. For instance, pseudo double categories, monoidal double categories, horizontal or vertical intercategories, and double (op)fibrations can all be described as models of enhanced limit $2$-sketches; $w$-monoidal functors, $w$-double functors, and $w$-morphisms of fibrations can all be seen as loose $w$-natural transformations.
	
	Apart from the theoretical aspect, the theory of enhanced limit $2$-sketches has seen great use in \emph{Applied Category Theory}. A central mission in Applied Category Theory is to study the compositionality of modular systems, in which modularity can often be described as certain $2$-categorical structures, including pseudo double categories, monoidal double categories, or symmetric monoidal loose right modules, as shown in \cite{CM:2026}, \cite{Baez:2025}, \cite{KSW:1997}, and \cite{Loregian:2025}. In \cite{BCLM:2025} and also \cite{LM:2025}, enhanced limit $2$-sketches are used to consider these $2$-algebraic structures formally. Understanding models of enhanced limit $2$-sketches then becomes significant to the further development of applications of Category Theory.
	
	This article takes the first step to investigate the features of models of enhanced limit $2$-sketches.
	
	First, we establish that the enhanced $2$-category of models of an enhanced limit $2$-sketch with tight cones is equivalent to the enhanced $2$-category of algebras over an enhanced $2$-monad.
	\begin{customthm}{\hspace{-1mm}}
		Let $\bbK$ be a locally presentable enhanced $2$-category. Let $\mathmybb{T}$ be an enhanced limit $2$-sketch with tight weighted cones, i.e., the shapes of weighted cones in $\mathmybb{T}$ are $2$-categories, viewed as chordate $\F$-categories. There is an enhanced $2$-monad $T$ on the enhanced $2$-category $\FMod{}(\mathccal{T}_\tau, \bbK)$ of models restricted to the tight morphisms such that
		$$\FMod{s, w}(\mathmybb{T}, \bbK) \simeq \FTAlg_{s, w}$$
		is an equivalence of enhanced $2$-categories.
	\end{customthm}
	
	As we will see in the proof of the theorem, the corresponding enhanced $2$-monad actually acts on the complete enhanced $2$-cateogry $\FMod{}(\mathccal{T}_\tau, \bbK)$ of models restricted to the tight part of the enhanced limit $2$-sketch. As a result, we establish an existence result for limits in $\FMod{s, w}(\mathmybb{T}, \bbK)$.
	\begin{customcoroll}{\hspace{-1mm}}
		Let $\mathmybb{T}$ be an enhanced limit $2$-sketch with tight weighted cones, and $\bbK$ be a locally presentable enhanced $2$-category. The restriction
		$$\mathmybb{Mod}_{s, w}(\mathmybb{T}, \mathmybb{K}) \to \mathmybb{Mod}_{s, w}(\mathccal{T}_\tau, \mathmybb{K})$$
		creates all $w$-rigged limits.
	\end{customcoroll}
	
	It is worth noting that in establishing our main theorem, we heavily rely on the notion of \emph{dotted $2$-limits} in \cite{Ko:2023}. More precisely, in showing that the enhanced $2$-category $\FMod{s, w}(\mathmybb{T}, \bbK)$ admits $\overline{w}$-limits of loose morphisms, we have to make use of dotted $2$-limits. The reason that one might not be able to prove via weighted limits directly is that the functor enhanced $2$-category defining weighted limits involves morphisms that are too strict, when comparing to the enhanced $2$-category of models, whose loose morphisms are \emph{loose $w$-natural transformations}, instead of $2$-natural transformations between loose parts. On the other hand, dotted $2$-limits are defined via dotted-$w$-natural transformations, which are closely related to loose $w$-natural transformations.
	
	Aside from our main results on models of enhanced limit $2$-sketches, we establish and generalise some statements in the theory of enriched limit sketches throughout the paper. Assuming that the base $\V$ of enrichment is equipped with a nice weak factorisation system, we show that models of any limit $\V$-sketch in a locally presentable $\V$-category $\matheus{K}$ are reflective. Moreover, we prove that the restriction of models in a locally presentable $\V$-category along any essentially surjective and cone-reflecting morphism of $\V$-sketches is always monadic. These strengthen several well-known results in \cite{book:Kelly:1982}, \cite{LWP:2024}, and \cite{BG:2019}, which only consider models in the base of enrichment $\V$.
	
	The article is outlined as follows. In \longref{Section}{sec:enriched_orthogonality}, we introduce the notion of orthogonal objects with respect to a class of morphisms, and relate it to enriched lifting properties. We then propose a general notion of models of limit $\V$-sketches with respect to a class of morphisms, and show that they are indeed orthogonal objects. In \longref{Section}{sec:reflect}, we first establish an enriched analogue of the celebrated Orthogonal Sub-category Theorem, and prove that models of limit  $\V$-sketches are reflective in the functor $\V$-category. In \longref{Section}{sec:monadicity}, we prove a monadicity result for models of limit $\V$-sketches in an arbitrary locally presentable $\V$-category, and then move on to recall basics in enhanced $2$-category theory and enhanced limit $2$-sketches. After that, we briefly discuss $2$-enhanced $2$-categories. At last, we establish the equivalence between models of an enhanced limit $2$-sketch with tight cones and algebras over an enhanced $2$-monads, and investigate further their limits.

	\section*{Acknowledgement}
	The author would like to acknowledge David Jaz Myers for his help and guidance. The author would also like to thank Timothy Hosgood, Jan Jurka, Giuseppe Leoncini, and Lukáš Vokřínek, for the fruitful discussions. The author is thankful to Nathanael Arkor and John Bourke for their comments on the early version of this article.
	
	The author is supported by Advanced Research + Invention Agency grant MSAI-PR01-P14.

	\section{Enriched orthogonality}
	\label{sec:enriched_orthogonality}	
	
	In this section, we propose several notions that generalise the ordinary orthogonality of objects and lifting properties of morphisms. 
	
	Let $\V$ be a Bénabou cosmos, i.e. $(\V, \otimes, I)$ is a (co)complete symmetric monoidal closed category. Let $\matheus{K}$ be a $\V$-category, i.e. $\matheus{K}$ is enriched in $\V$. Let $\mathccal{R}$ be a class of morphisms in $\V_0$ which is closed under isomorphisms and composition.
	
	\subsection{Orthogonality and lifting properties with respect to a class}
	
	We first generalise the notion of orthogonal objects to the enriched setting.
	
	\begin{defi}
		\label{def:orthogonal}
		Let $m \colon M_1 \to M_2$ be a morphism in $\V_0$.
		An object $K$ of $\matheus{K}$ is \emph{orthogonal to $m$ with respect to $\mathccal{R}$} precisely when the pre-composition
		$$\matheus{K}(M_2, K) \xrightarrow{\matheus{K}(m, K)} \matheus{K}(M_1, K)$$
		belongs to $\mathccal{R}$.
	\end{defi}
	
	\begin{example}
		When $\V = \Set$ and $\calR$ consists of bijections precisely, we recover the ordinary notion of orthogonal objects.
	\end{example}
	
	\begin{example}
		When $\V = \Set$ and $\calR$ consists of surjections precisely, we recover weak orthogonality.
	\end{example}
	
	\begin{example}
		When $\V = \sSet$, the category of simplicial sets, and $\calR$ contains all the weak equivalences, the objects which are orthogonal to the horn inclusions with respect to $\calR$ are precisely the Kan complices.
	\end{example}
	
	\begin{rk}
		Our notion of orthogonality with respect to $\calR$ coincides with the concept of $\calR$-injectivity in \cite{LR:2012}.
	\end{rk}
	
	\begin{nota}
		For a class $\matheus{M}$ of morphisms in $\matheus{K}_0$, we denote by $\matheus{M}^{\perp^\mathccal{R}}$ the full sub-$\V$-category of objects which are orthogonal to all morphisms in $\matheus{M}$ with respect to $\mathccal{R}$.
		
		When $\calR$ consists of precisely the isomorphisms in $\V$, then we omit the superscript '$\calR$'.
	\end{nota}
	
	As in the ordinary case, the notion of orthogonality of objects should be related to certain lifting properties of morphisms. The notion of morphisms with lifting properties in the enriched setting has been proposed in {\cite{Jurka:2025}}.
	
	\begin{construct}[{\cite[Definition 4.7]{Jurka:2025}}]
		Let $f \colon A \to B$ and $g \colon C \to D$  be two morphisms in $\matheus{K}_0$.
		
		Consider the pullback $\mathrm{Sq}(f, g)$
		\[
		\begin{tikzcd}[ampersand replacement=\&]
			{\matheus{K}(B, C)} \&\&\& \\
			\& {\mathrm{Sq}(f, g)} \&\& {\matheus{K}(A, C)} \\
			\\
			\& {\matheus{K}(B, D)} \&\& {\matheus{K}(A, D)}
			\arrow["{\langle f, g \rangle}"{description}, dashed, from=1-1, to=2-2]
			\arrow["{\matheus{K}(f, C)}", from=1-1, to=2-4]
			\arrow["{\matheus{K}(B, g)}"', from=1-1, to=4-2]
			\arrow[from=2-2, to=2-4]
			\arrow[from=2-2, to=4-2]
			\arrow["\lrcorner"{anchor=center, pos=0.125}, draw=none, from=2-2, to=4-4]
			\arrow["{\matheus{K}(A, g)}", from=2-4, to=4-4]
			\arrow["{\matheus{K}(f, D)}"', from=4-2, to=4-4]
		\end{tikzcd}
		\]
		of $\matheus{K}(f, D)$ along $\matheus{K}(A, g)$ in $\V_0$, and the commutative diagram $\matheus{K}(f, D) \cdot \matheus{K}(B, g) = \matheus{K}(A, g) \cdot \matheus{K}(f, C)$. There is a uniquely induced map, which we denote by $\langle f, g \rangle \colon \matheus{K}(B, C) \to \mathrm{Sq}(f, g)$, making the above diagram commute. 
	\end{construct}
	
	\begin{defi}
		\label{def:lifting}
		Let $f \colon A \to B$ and $g \colon C \to D$ be morphisms in $\matheus{K}_0$.
		We say that $f$ \emph{has the left lifting property against $g$ with respect to $\mathccal{R}$} (or. equivalently, $g$ \emph{has the right lifting property against $f$ with respect to $\mathccal{R}$}) just when the induced map $\langle f, g \rangle$ belongs to $\calR$.
	\end{defi}
	
	\begin{nota}
		Using the notation in {\cite[Definition 4.7]{Jurka:2025}}, when $f$ and $g$ have the lifting properties against each other, we may write $f \pitchfork^\calR g$. We then define, for a class $\matheus{M}$ of moprhisms in $\matheus{K}_0$,
		$$
		\begin{aligned}
			&\matheus{M}^{\pitchfork^\calR} := \{g \colon f \pitchfork^\calR g, \forall f \in \matheus{M}\},
			\\
			&^{\pitchfork^\calR}\matheus{M} := \{f \colon f \pitchfork^\calR g, \forall g \in \matheus{M}\}.
		\end{aligned}
		$$
	\end{nota}
	
	Similar to the ordinary situation, the existence of a terminal object is necessary for bridging and unifying the two separate concepts of orthogonal objects and morphisms with lifting properties.
	
	\begin{lemma}
		\label{lem:bridge}
		Suppose $\matheus{K}$ admits a terminal object $1$. Then, for any class $\matheus{M}$ of morphisms in $\matheus{K}_0$, an object $K$ of $\matheus{K}$ belongs to $\matheus{M}^{\perp^\calR}$ if and only if the unique morphism $!_K \colon K \to 1$ to the terminal object is in $\matheus{M}^{\pitchfork^\calR}$.
	\end{lemma}
	
	\begin{proof}
		Consider, for a morphism $m \colon M_1 \to M_2$ in $\matheus{M}$, the pullback
		\begin{equation}
			\label{diag:gap_map}
			\begin{tikzcd}[ampersand replacement=\&]
				{\matheus{K}(M_2, K)} \&\&\& \\
				\& {\mathrm{Sq}(m, !_K)} \&\& {\matheus{K}(M_1, K)} \\
				\\
				\& {1\cong\matheus{K}(M_2, 1)} \&\& {\matheus{K}(M_2, 1)\cong 1}
				\arrow["{\langle m, !_K \rangle}"{description}, dashed, from=1-1, to=2-2]
				\arrow["{\matheus{K}(m, K)}", from=1-1, to=2-4]
				\arrow["{\matheus{K}(M_2, !_K)}"', from=1-1, to=4-2]
				\arrow["\cong"', from=2-2, to=2-4]
				\arrow[from=2-2, to=4-2]
				\arrow["\lrcorner"{anchor=center, pos=0.125}, draw=none, from=2-2, to=4-4]
				\arrow["{\matheus{K}(M_1, !_K)}", from=2-4, to=4-4]
				\arrow["\cong"', from=4-2, to=4-4]
			\end{tikzcd}
		\end{equation}
		of $\matheus{K}(m, 1)$ along $\matheus{K}(M_1, !_K)$, and the induced map $\langle m, !_K \rangle$. Clearly, we have $f \cdot \langle m, !_K \rangle = \matheus{K}(m, K)$, for an isomorphism $f$. In other words, $\langle m, !_K \rangle \in \calR$ if and only if $\matheus{K}(m, K) \in \calR$.
	\end{proof}
	
	\begin{rk}
		This recovers the fact that an object is orthogonal to a class $\matheus{M}$ of morphisms in a category if and only if the map from the object to the terminal object has the right lifting property against $\matheus{M}$.
	\end{rk}
	
	\subsection{The orthogonality of $\V$-models with respect to a class}
	\label{sec:model_orthogonal}
	
	We now turn to the theory of enriched limit sketches, which is first introduced by Kelly in \cite[Chapter 6]{book:Kelly:1982}.
	
	We will see that $\V$-functors which weakly respect cones in an enriched limit sketch are actually orthogonal objects.
	
	\begin{defi}[{\cite[\S6.3]{book:Kelly:1982}}]
		\label{def:V-sketch}
		A \emph{limit $\V$-sketch} consists of a small $\V$-category $\matheus{S}$, equipped with an $I$-indexed collection $\Cone{\matheus{S}}$ of $\V$-natural transformations
		$$\{(W_i \colon \matheus{J}_i \to \V, \; D_i \colon \matheus{J}_i \to \matheus{S}, \; s_i \in \matheus{S}, \; \gamma_i \colon W_i \to \matheus{S}(s_i, D_i -))\}_{i \in I}$$
		which are called \emph{weighted cones}.
		
		A \emph{morphism of limit $\V$-sketches} is a $\V$-functor $i \colon \matheus{S} \to \matheus{T}$ between two limit $\V$-sketches $\matheus{S}$ and $\matheus{T}$, such that the image
		$$(W \colon \matheus{J} \to \V, \; i \cdot D \colon \matheus{J} \to \matheus{T}, \; i(s) \in \matheus{T}, \; i \cdot \gamma_i \colon 	W \xrightarrow{\gamma} \matheus{S}(s, D -) \xrightarrow{i} \matheus{K}(i(s), i \cdot D -))$$
		of a weighted cone $(W, D, s, \gamma)$ in $\matheus{S}$ is a weighted cone in $\matheus{T}$; in other words, $i$ preserves weighted cones.
	\end{defi}
	
	Instead of recalling the definition of $\V$-models, we introduce a weaker notion than the usual $\V$-models of limit $\V$-sketches. For this, we assume further that $\matheus{K}$ is a complete $\V$-category.
	
	\begin{defi}
		\label{def:model_wrt_R}
		Let $\matheus{S}$ be a limit $\V$-sketch with weighted cones $\{W_i, D_i, s_i, \gamma_i\}_{i \in I}$.
		A \emph{$\V$-model with respect to $\calR$ in $\matheus{K}$} is a $\V$-functor $F \colon \matheus{S} \to \matheus{K}$ such that for any $i \in I$, the post-composition
		$$\matheus{K}(K, F(s_i)) \xrightarrow{\matheus{K}(K, \rho_i)} \matheus{K}(K, \{W_i, FD_i\})$$
		for any object $K$ of $\matheus{K}$ belongs to $\calR$, where
		\begin{equation*}
			\label{eqt:universal_map}
			F(s_i) \xrightarrow{\rho_i} \{W_i, FD_i\}
		\end{equation*}
		is induced by the $\V$-natural transformation
		\begin{equation}
			\label{eqt:model}
			W_i \xrightarrow{\gamma_i} \matheus{S}(s_i, D_i -) \xrightarrow{F} \matheus{K}(F(s_i), FD_i -)
		\end{equation}
		under the universal property of the weighted limit $\{W_i, FD_i\}$.
	\end{defi}
	
	In other words, $\V$-models with respect to $\calR$, unlike $\V$-models, do not preserve weighted cones, i.e., they do not send weighted cones in the domain limit $\V$-sketch to limit cones in the codomain; instead, they preserve them up to the class $\calR$.
	
	Our notion of models with respect to a class $\calR$ might seem unfamiliar at first glance, yet, this is because we are trying to capture higher categorical weakenings of equivalences, via a $1$-categorical approach with enrichment. In fact, this is a natural generality which is consistent with enriched orthogonality and enriched lifting. The following examples illustrate how our notion works well in a higher setting, and is able to capture weaker notions of limits.
	
	\begin{example}
		Let $\V = \Cat$, and let $\calR$ be the class of equivalences of categories. Let $\mathccal{S}$ be a limit $2$-sketch equipped with \emph{flexible} weighted cones, introduced in \cite{BKPS:1989}.
		
		A model is then a $2$-functor $F \colon \mathccal{S} \to \mathccal{K}$ such that for any weighted cone $(W, D, s, \gamma)$ in $\mathccal{S}$, the image $F(s)$ shares the universal property of bilimits. For, the composite
		$$
		{\calK(K, F(s))} \xrightarrow[\simeq]{\calK(K, \rho)} \calK(K, \{W, FD\}) {\xrightarrow[\cong]{}} [\mathccal{J}, \Cat](W, \calK(K, FD-))
		$$		
		is clearly an equivalence, for any object $K$ of $\calK$. Now, since $W$ is flexible, according to \cite[Proposition 4.2]{BKPS:1989}, we have
		$$[\mathccal{J}, \Cat](W, \calK(K, FD-)) {\xrightarrow[\simeq]{}} [\mathccal{J}, \Cat](W', \calK(K, FD-)) {\xrightarrow[\cong]{}} [\mathccal{J}, \Cat]_p(W, \calK(K, FD-)),$$
		where $W'$ denotes the image of $W$ under the pseudo morphism classifier, and $[\mathccal{J}, \Cat]_p$ is the $2$-category of $2$-functors from $\mathccal{J}$ to $\Cat$ and their pseudo natural transformations and modifications.
		
		Altogether, we recover \cite[(2.6)]{BKPS:1989}, and thus $F(s)$ has the universal property of bilimits.
	\end{example}
	
	\begin{example}
		Let $\V = \sSet$ and $\calR$ be the class of weak equivalences. Let $\matheus{S} = \mathmybb{\Delta}^\op$, the opposite of the simplex category, with weighted cones given by cones of the form
		\[
		\begin{tikzcd}[ampersand replacement=\&]
			\&\&\& {[n]} \&\&\& \\
			\& {[1]} \&\& \cdots \&\& {[1]} \\
			{[0]} \&\& {[0]} \& \cdots \& {[0]} \&\& {[0]}
			\arrow[""{name=0, anchor=center, inner sep=0}, "{\delta_n}"', from=1-4, to=2-2]
			\arrow[""{name=1, anchor=center, inner sep=0}, "{\delta_0}", from=1-4, to=2-6]
			\arrow["{\delta_1}"', from=2-2, to=3-1]
			\arrow["{\delta_0}", from=2-2, to=3-3]
			\arrow[""{name=2, anchor=center, inner sep=0}, "{\delta_1}"', from=2-4, to=3-3]
			\arrow[""{name=3, anchor=center, inner sep=0}, "{\delta_0}", from=2-4, to=3-5]
			\arrow["{\delta_1}"', from=2-6, to=3-5]
			\arrow["{\delta_0}", from=2-6, to=3-7]
			\arrow["\cdots"{description}, draw=none, from=0, to=1]
			\arrow["\cdots"{description}, draw=none, from=2, to=3]
		\end{tikzcd},
		\]
		where $\delta_i$ denotes the evident face maps. Models of $\matheus{S}$ in $\V = \sSet$ with respect to $\calR$ are simplicial spaces that satisfy the Segal condition defined in \cite[Definition 3.1]{JT:2007}  precisely.
	\end{example}
	
	\begin{example}
		When $\calR$ consists of precisely the isomorphisms in $\V$, it is obvious that \longref{Definition}{def:model_wrt_R} recovers the definition of $\V$-models of a limit $\V$-sketch.
	\end{example}
	
	We are now ready to prove that models with respect to a class are orthogonal objects. 
	
	\begin{construct}
		\label{construct:sigma}
		Let $\yo \colon \matheus{S}^\op \to [\matheus{S}, \V]$ be the Yoneda embedding, and $(W, D, s \gamma)$ be a weighted cone in $\matheus{S}$. Consider the composite
		\begin{align*}
			\matheus{J}^\op &\xrightarrow{D^\op} \matheus{S}^\op \xrightarrow{\yo} [\matheus{S}, \V]
			\\
			j &\mapsto D^\op(j) \mapsto \matheus{S}(-, D^\op(j)).	
		\end{align*}
		Since $\V$ is cocomplete, the weighted colimit $W * \yo D^\op$ of the composite exists. 
		
		By the universal property, there exists a unique map
		\begin{equation*}
			\label{eqt:universal_map_2}
			W * \yo D^\op \xrightarrow{\sigma} \yo(s)
		\end{equation*}
		in $[\matheus{S}, \V]$, which corresponds to the $\V$-natural transformation \longref{}{eqt:model} with $F = \yo$.	
	\end{construct}
	
	\begin{lemma}
		\label{lem:model_orthogonal}
		Let $\matheus{S}$ be a $\V$-sketch with weighted cones $\{(W_i, D_i, s_i, \gamma_i)\}_{i \in I}$, and $\matheus{K}$ be a complete $\V$-category. A $\V$-functor $F \colon \matheus{S} \to \matheus{K}$ is a $\V$-model with respect to $\calR$ if and only if for each object $K$ of $\matheus{K}$, $\matheus{K}(K, F-)$ is orthogonal to the class $\{W_i * \yo D_i^\op \xrightarrow{\sigma_i} \yo(s_i)\}_{i \in I}$ of maps defined in \longref{Construction}{construct:sigma}.
		
		As a consequence, $F \colon \matheus{S} \to \matheus{K}$ is a $\V$-model with respect to $\calR$ if and only if $F$ is orthogonal to the class $\{(W_i * \yo D_i^\op) \otimes K \xrightarrow{\sigma_i \otimes K} \yo(s_i) \otimes K\}_{i \in I, K \in \ob\matheus{K}}$ of maps in $[\matheus{S}, \matheus{K}]$ with respect to $\calR$.
	\end{lemma}
	
	\begin{proof}
		Note that the pre-composition $[\matheus{S}, \V](\sigma_i, \matheus{K}(K, F-))$ by $\sigma_i$ is given by the composite
		\begin{align*}
			[\matheus{S}, \V](\yo(s_i), \matheus{K}(K, F-)) &\cong \matheus{K}(K, F(s_i)) & \text{(by Yoneda's Lemma)}			\\
			& \xrightarrow{\matheus{K}(K, \rho_i)} \matheus{K}(K, \{W_i, FD_i\}) & 
			\\
			&\cong \{W_i, \matheus{K}(K, FD_i)\} & \text{(as representables preserve limits)}
			\\
			&\cong \{W_i, [\matheus{S}, \V](\yo D_i^\op, \matheus{K}(K, F-))\} &\text{(by Yoneda's Lemma)}
			\\
			&\cong [\matheus{S}, \V](W_i * \yo D_i^\op, \matheus{K}(K, F-)), &\text{(by the universal property of weighted colimits)}
		\end{align*}
		which belongs to $\calR$ precisely if $\matheus{K}(K, \rho_i)$ is in $\calR$, i.e., $F$ is a $\V$-model with respect to $\calR$.
		
		
		Since $\matheus{K}$ admits tensors by $\V$, altogether, the pre-composition $[\matheus{S}, \matheus{K}](\sigma_i \otimes K, F)$ 
		$$[\matheus{S}, \matheus{K}](\matheus{S}(-, s_i) \otimes K, F) \to [\matheus{S}, \matheus{K}]((W_i * \yo D_i^\op) \otimes K, F)$$
		by $\sigma_i \otimes K$
		belongs to $\calR$ if and only if $F$ is a $\V$-model with respect to $\calR$.
	\end{proof}

	\section{Reflectivity from orthogonality}
	\label{sec:reflect}
	
	\subsection{Adjunctions with respect to a class}
	\label{sec:adj_wrt_R}
	
	The Orthogonal Sub-category Theorem is about the reflectivity of orthogonal objects. Since we have defined orthogonal objects and morphisms with lifting properties with respect to a class of morphisms, it is natural to introduce a notion of adjunctions with respect to a class.
	
	\begin{defi}
		\label{def:generalised_adj}
		Let $\matheus{C}, \matheus{D}$ be $\V$-categories, and $U \colon \matheus{D} \to \matheus{C}$ be a $\V$-functor. A functor $F \colon \matheus{C}_0 \to \matheus{D}_0$ between the underlying categories is said to be an \emph{unenriched left adjoint of $U$ with respect to $\mathccal{R}$} (or, equivalently, $U$ is a \emph{unenriched right adjoint of $F$ with respect to $\mathccal{R}$}) if and only if there is a $\ob \matheus{C}$-indexed collection of maps
		$$
		\{c \xrightarrow{\eta_c} UFc\}_{c \in \ob \matheus{C}}
		$$
		in $\V_0$ satisfying unenriched naturality, such that for any object $c$ of $\matheus{C}$ and any object $d$ of $\matheus{D}$, the composite
		\begin{equation}
			\label{eqt:adj}
			\matheus{D}(Fc, d) \xrightarrow{U_{Fc, d}} \matheus{C}(UFc, Ud) \xrightarrow{\matheus{C}(\eta_c, Ud)}\matheus{C}(c, Ud)
		\end{equation}
		belongs to $\mathccal{R}$.
		
		We call such a pair of $\V$-functors an \emph{unenriched adjunction with respect to $\mathccal{R}$}.
	\end{defi}
	
	\begin{rk}
		\label{rk:natural}
		According to \cite[1.8(b), (e)]{book:Kelly:1982}, the composite \longref{}{eqt:adj} is natural in $d$.
	\end{rk}
	
	\begin{example}
		When $\V = \Set$, the category of sets, an ordinary adjunction between categories is clearly an unenriched adjunction with respect to the class of bijections.
		
		Indeed, any enriched adjunction is an unenriched adjunction with respect to the class of isomorphisms in $\V$.
	\end{example}
	
	\begin{example}[Projective covers as an  adjunction with respect to surjections]
		Let $\V = \Set$, $\calR$ be the class of surjections, $C$ be a category with enough projectives, and $D = C^{\twoheadrightarrow}$ be the full sub-category of the arrow category of $C$ determined by the epimorphisms. Let
		\begin{align*}
			C^{\twoheadrightarrow} &\xrightarrow{U} C
			\\
			d_1 \twoheadrightarrow d_2 &\mapsto d_2
		\end{align*}
		be the codomain projection, and let
		\begin{align*}
			C &\xrightarrow{F} C^{\twoheadrightarrow}
			\\
			c &\mapsto p \twoheadrightarrow c
		\end{align*}
		be a functor that sends each $c \in \ob C$ to a projective cover of $c$.
		
		Then, for any morphism $c \to d_2$ in $C$, and projective covers $p \twoheadrightarrow c$ and $d_1 \twoheadrightarrow d_2$, there exists a morphism $p \to d_1$ such that
		\[
		\begin{tikzcd}[ampersand replacement=\&]
			p \& c \\
			{d_1} \& {d_2}
			\arrow[two heads, from=1-1, to=1-2]
			\arrow[dashed, from=1-1, to=2-1]
			\arrow[from=1-2, to=2-2]
			\arrow[two heads, from=2-1, to=2-2]
		\end{tikzcd}
		\]
		commutes, by the projective nature of $p$. In other words, by taking $\eta_c$ to be the identity for all $c \in \ob C$, $F$ is a left adjoint of $U$ with respect to $\calR$.
	\end{example}
	
	%
		%
	%
	
	\begin{rk}
		Our notion of unenriched adjunctions with respect to $\calR$ is similar to the notion of weak adjoints introduced in \cite{LR:2012}. We require further that the collection of maps $\{\eta_c\}$ to satisfy at least unenriched naturality, whereas in \cite{LR:2012}, a weak adjunction requires only their existence.
	\end{rk}

	\subsection{The enriched Orthogonal Sub-category Theorems}
	\label{sec:OSP}
	
	We establish an enriched analogue of the Orthogonal Sub-category Theorem. This result will be used in showing the reflectivity of models of enhanced limit $2$-sketches in the functor enhanced $2$-category.
	
	To prove the theorem, we make use of the enriched small object argument developed in \cite{Jurka:2025}.
	
	To begin with, let us assume that $\V_0$ is equipped with a weak factorisation system $(\mathccal{L}, \calR)$ which is cofibrantly generated by a set $\G$ of morphisms, where the left class $\mathccal{L}$ is corner-stable as in \cite[Definition 5.4]{Jurka:2025}. Let $\matheus{K}$ be a $\V$-category which has tensors by $\V$, pushouts, transfinite composites, and a terminal object. Let $\mathccal{I}$ be a set of morphisms in $\matheus{K}_0$, satisfying the smallness condition in \cite[Definition 8.1]{Jurka:2025}. In other words, we assume every condition required in the enriched small object argument.
	
	\begin{pro}[General Enriched Orthogonal Sub-category Theorem]
		\label{pro:general_EOST}
		The inclusion
		$$\mathccal{I}^{\perp^\calR} \xhookrightarrow{I} \matheus{K}$$
		is an unenriched right adjoint with respect to $\calR$.
	\end{pro}
	
	\begin{proof}
		We show that there exists a functor
		$$\matheus{K}_0 \xrightarrow{L} \mathccal{I}^{\perp^\calR}_0,$$
		together with an $\ob \matheus{K}$-indexed collection of $\V$-natural maps
		$$K \xrightarrow{\eta_K} ILK$$
		in $\matheus{K_0}$, such that for any object $K$ of $\matheus{K}$ and $X$ of ${\mathccal{I}^{\perp^\calR}}$, the composite
		\begin{equation*}
			{\mathccal{I}^{\perp^\calR}}(LK, X) \xrightarrow{I_{LK, X}} \matheus{K}(ILK, IX) \xrightarrow{\matheus{K}(\eta_K, IX)}\matheus{K}(K, IX)
		\end{equation*}
		belongs to $\calR$.
		
		Following the proof of \cite[Theorem 8.5]{Jurka:2025}, for any object $K$ of $\matheus{K}$, we can factorise the unique map $!_K \colon K \to 1$ to the terminal object as
		\begin{equation}
			\label{diag:factorise}
			\begin{tikzcd}[ampersand replacement=\&]
				K \&\& 1 \\
				\& {\colim_{\alpha < \lambda} D^K(\alpha)}
				\arrow["{!_K}", from=1-1, to=1-3]
				\arrow["{e_K}"', from=1-1, to=2-2]
				\arrow["{!_{\colim_{\alpha < \lambda} D^K(\alpha)}}"', from=2-2, to=1-3]
			\end{tikzcd}
		\end{equation}
		through the conical colimit $\colim_{\alpha < \lambda} D^K(\alpha)$ of the diagram $D^K \colon \lambda \to \matheus{K}$ with $D^K(0) := K$; the map $e_K$ is the colimit injection $D^K(0) \to \colim_{\alpha < \lambda} D^K(\alpha)$, and that $\lambda$ is a suitable regular cardinal.
		
		We then define a functor
		\begin{align}
			\label{eqt:L}
			L \colon \matheus{K}_0 &\to {\mathccal{I}^{\perp^\calR}_0}
			\\
			K_1 &\mapsto \colim_{\alpha < \lambda} D^{K_1}(\alpha) \nonumber
			\\
			k \downarrow \; &\mapsto \quad \quad \quad \downarrow Lk \nonumber
			\\
			K_2 &\mapsto \colim_{\alpha < \lambda} D^{K_2}(\alpha), \nonumber
		\end{align}
		where $Lk \colon \colim_{\alpha < \lambda} D^{K_1}(\alpha) \to \colim_{\alpha < \lambda} D^{K_2}(\alpha)$ is a morphism in $\matheus{K}_0$ obtained via the universal property.
		
		Besides, for any object $K$ of $\matheus{K}$, we define
		$K \xrightarrow{\eta_K} ILK$ as the map $e_K \colon K \to \colim_{\alpha < \lambda} D^{K}(\alpha)$ in \longref{Diagram}{diag:factorise}; the unenriched naturality of $\{\eta_K\}_{K \in \ob\matheus{K}}$ is guaranteed by the universal property of $\V$-weighted colimits.
		
		Finally, since the inclusion $I$ is fully faithful, it suffices to verify that the pre-composition
		$$\matheus{K}(ILK, IX) \xrightarrow{\matheus{K}(\eta_K, IX)} \matheus{K}(K, IX)$$
		by $\eta_K$ belongs to $\calR$. In fact, as in \longref{Diagram}{diag:gap_map}, we have $\matheus{K}(\eta_K, IX) = \langle \eta_K, !_{IX} \rangle$, hence, it remains to check if $\langle \eta_K, !_{IX} \rangle$ is in $\calR$.
		
		Note, by \cite[Theorem 8.5]{Jurka:2025}, that $\eta_K$ is in $^{\pitchfork^\matheus{K}}{(\mathccal{I}^{\pitchfork^\matheus{K}})}$. Besides, according to \longref{Lemma}{lem:bridge}, $!_{IX}$ belongs to $\mathccal{I}^{\pitchfork^\calK}$. Altogether, we see that $\eta_K$ and $!_{IX}$ have the lifthing properties against each other with respect to $\calR$, therefore, we conclude that $\langle \eta_K, !_{IX} \rangle$ belongs to $\calR$.
	\end{proof}
	
	\begin{example}
		Clearly, the classical Orthogonal Sub-category Theorem is an example.
	\end{example}
	
	\begin{example}[The relative Stone-Čech compactification]
		Consider when $\V = \Set$, where $\calR$ consists of precisely the bijections, and let $\matheus{K}$ be the slice category $\gps{B}$ over a topological space $B$.
		
		We set $\mathccal{I}$ to be the class
		$$\{Y \xrightarrow{\eta_Y} \rscc{B}Y \colon Y \in \gps{B}\}$$
		of continuous maps, where $\rscc{B}$ denotes the relative Stone-Čech compactification in \cite[Definition 4.4]{KM:2024}, and each $\eta_Y$ refers to the inclusion of $Y$ into its relative Stone-Čech compactification.
		
		Then, $\mathccal{I}^{\perp^\calR}$ is exactly the category $\pgps{B}$ of topological spaces where the anchor maps to $B$ are proper. For an an object $Z \in \pgps{B}$, by the universal property of the relative Stone-Čech compactification, for any $Y \in \gps{B}$ and any continuous map $Y \to Z$, there is a unique map $f \colon \rscc{B}Y \to Z$ such that the triangle
		\[
		\begin{tikzcd}[ampersand replacement=\&]
			Y \&\& Z \\
			\& {\rscc{B}Y}
			\arrow[from=1-1, to=1-3]
			\arrow["{\eta_Y}"', from=1-1, to=2-2]
			\arrow["{\exists! f}"', dashed, from=2-2, to=1-3]
		\end{tikzcd}
		\]
		commutes. This means that
		$$\gps{B}(\rscc{B}Y, Z) \xrightarrow{\gps{B}(\eta_Y, Z)} \gps{B}(Y, Z)$$
		is a bijection, and so $Z \in \mathccal{I}^{\perp^\calR}$. Conversely, suppose $Z \in \mathccal{I}^{\perp^\calR}$, which means the pre-composition by $\eta_Y$ above is a bijection, and hence the identity on $Z$ can be factorised along $\eta_Z \colon Z \to \rscc{B}Z$; consequently,  we have $Z \cong \rscc{B}Z$ in $\pgps{B}$.
		
		We recover the fact that $\pgps{B}$ is a reflective sub-category of $\gps{B}$. By choosing $B = 1$, we recover the fact that compact Hausdorff spaces are reflective in topological spaces.
	\end{example}
	
	\begin{rk}
		In general, the functor $L$ may not be enriched. See \cite[\S1]{Vokrinek:2015} for a relevant discussion.
	\end{rk}
	
	Under special conditions, it is possible to upgrade the functor $L$ constructed in the proof of \longref{Proposition}{pro:general_EOST} to an enriched $\V$-functor. This would further require the map
	$$\matheus{K}(\eta_{K_2}, LK_2) \colon \mathccal{I}^{\perp^\calR}(LK_1, LK_2) \to \matheus{K}(K_1, LK_2)$$ to be naturally reversible, i.e., there exists a collection of maps in the opposite direction which satisfy $\V$-naturality.
	
	Indeed, for some of the applications, requiring $\calR$ to be precisely the isomorphisms and $\matheus{K}$ to admit powers is sufficient. 
	
	\begin{coroll}[Special Enriched Orthogonal Sub-category Theorem]
		\label{cor:special_EOST}
		If $\calR$ contains precisely the isomorphisms in $\V$, and $\matheus{K}$ has powers by $\V$, then we obtain a $\V$-functor $L$ such that $L \dashv I$, where I is the inclusion in \longref{Proposition}{pro:general_EOST}.
		
		As a result, $\mathccal{I}^\perp$ is a reflective sub-$\V$-category of $\matheus{K}$.
	\end{coroll}
	
	\begin{proof}
		The map
		\begin{equation*}
			{\mathccal{I}^{\perp^\calR}_0}(LK, X) \xrightarrow{I_{LK, X}} \matheus{K}_0(ILK, IX) \xrightarrow{\matheus{K}_0(\eta_K, IX)}\matheus{K}_0(K, IX)
		\end{equation*}
		is a bijection natural in both $K$ and $X$, according to \longref{Remark}{rk:natural}, so we have an ordinary adjunction between the underlying categories. According to \cite[Theorem 4.85]{book:Kelly:1982}, it suffices to check that the inclusion $I$ preserves powers; in other words, we shall verify that for any object $V$ of $\V$, $Y$ of $\mathccal{I}^\perp$, and any morphism $i \colon I_1 \to I_2$ in $\mathccal{I}$, the canonical map
		$$\matheus{K}(I_2, Y^V) \xrightarrow{\matheus{K}(i, T^V)} \matheus{K}(I_1, Y^V)$$
		is an isomorphism.
		
		Indeed, by the universal property of powers, we know that $\matheus{K}(I_2, Y^V) \cong \V(V, \matheus{K}(I_2, Y))$ and $\matheus{K}(I_1, Y^V) \cong \V(V, \matheus{K}(I_1, Y))$, therefore, it remains to prove that the canonical map
		$$\V(V, \matheus{K}(I_2, Y)) \xrightarrow{\V(V, \matheus{K}(i, Y))} \V(V, \matheus{K}(I_1, Y))$$
		is an isomorphism. Now, since $\matheus{K}(i, Y)$ is an isomorphism because $Y \in \mathccal{I}^\perp$, the result immediately follows.
	\end{proof}
	
	For two such reflective sub-categories, there is a nice interaction between them.
	
	\begin{lemma}
		\label{lem:reflector_interact}
		Let $\matheus{K}^1$ and $\matheus{K}^2$ be $\V$-categories with tensors and powers by $\V$, pushouts, transfinite composites, and a terminal object. Let $\mathccal{I}_i$ be a set of morphisms in $\matheus{K}^i_0$, for $i = 1, 2$, satisfying the smallness condition. Suppose $U_\mathccal{I} \colon \mathccal{I}^\perp_1 \to \mathccal{I}^\perp_2$ and $U_\matheus{K} \colon \matheus{K}^1 \to \matheus{K}^2$ are $\V$-functors such that
		\[
		\begin{tikzcd}[ampersand replacement=\&]
			{\mathccal{I}^\perp_1} \& {\matheus{K}^1} \\
			{\mathccal{I}^\perp_2} \& {\matheus{K}^2}
			\arrow[hook, from=1-1, to=1-2]
			\arrow["{U_\mathccal{I}}"', from=1-1, to=2-1]
			\arrow["{U_\matheus{K}}", from=1-2, to=2-2]
			\arrow[hook, from=2-1, to=2-2]
		\end{tikzcd}
		\]
		commutes, and $U_\matheus{K}$ is a left adjoint. Then, for any object $G \in \ob\matheus{K}^1$, we have
		$$U_\mathccal{I}  L_1(G) \cong L_2 U_\matheus{K}(G),$$
		where $L_i$ denotes the reflector $\matheus{K}^i \to {\mathccal{I}^\perp_i}$, for $i = 1, 2$.
	\end{lemma}
	
	\begin{proof}
		For any object $X \in \ob \matheus{K}^i$, for $i = 1, 2$, we define
		$$D^X \colon \lambda \to \matheus{K}^i_0$$
		as in the proof of \cite[Theorem 8.5]{Jurka:2025}.
		
		We claim that $D^{U_\matheus{K}(G)} \cong U_\matheus{K}(D^G)$. For, each image of $D^G$ is  a colimit, as constructed in the proof of \cite[Theorem 8.5]{Jurka:2025}, and that $U_\matheus{K}$ is a left adjoint, which preserves colimits.
		
		Now, we show that $U_\matheus{K}(\colim D^G) \cong \colim D^{U_\matheus{K}(G)}$. Suppose $D^{U_\matheus{K}(G)} \to \triangle(X)$ is a cocone, where $X \in \ob\matheus{K}^2$. Then, each component
		$$U_\matheus{K}(D^G(\alpha)) \cong D^{U_\matheus{K}(G)}(\alpha) \to X$$
		in $\matheus{K}^2$ at an ordinal $\alpha < \lambda$ corresponds bijectively to a map
		$$D^G(\alpha) \to U(X)$$
		under the adjunction $U_\matheus{K} \dashv R$. So, we have a cocone $D^G \to \triangle(U(x))$ from the naturality of the adjunction, which, by the universal property of $\colim D^G$, induces a unique map
		$$\colim D^G \to U(X)$$
		in $\matheus{K}^1$ such that the evident diagram commutes. This map then corresponds bijectively to a map
		$$U_\matheus{K}(\colim D^G) \to X$$
		in $\matheus{K}^2$ under the adjunction $U_\matheus{K} \dashv R$, such that the evident diagram commutes. Note that such map must be unique. Consequently, $U_\matheus{K}(\colim D^G)$ is a colimit of the diagram $D^{U_\matheus{K}(G)}$.
		
		Lastly, recall from the proof of \longref{Proposition}{pro:general_EOST} that $L_1(G) := \colim D^G$, and $L_2(U_\matheus{K}(G)) := \colim D^{U_\matheus{K}(G)}$, so $U_\mathccal{I}  L_1(G) \cong L_2 U_\matheus{K}(G)$ as desired.
	\end{proof}
	
	We are now ready to apply our results to the case of $\V$-models of limit $\V$-sketches.
	
	\begin{nota}
		Denote by $\matheus{M}\text{\fontshape{ui}\selectfont od}(\matheus{S}, \matheus{K})$ the full sub-$\V$-category of the functor $\V$-category $[\matheus{S}, \matheus{K}]$ determined by the $\V$-models of $\matheus{S}$ (with respect to the isomorphisms in $\V$). 
	\end{nota}
	
	\begin{theorem}
		\label{thm:reflectivity}
		Let $\V$ be a Bénabou cosmos, where $\V_0$ is equipped with a weak factorisation system  which is cofibrantly generated by a set of morphisms whose (co)domains are locally presentable in the unenriched sense, and that  the left class is corner-stable.
		
		Let $\matheus{S}$ be a limit $\V$-sketch, and $\matheus{K}$ be a locally presentable $\V$-category. The (fully faithful) inclusion
		$$\matheus{M}\text{\fontshape{ui}\selectfont od}(\matheus{S}, \matheus{K}) \hookrightarrow [\matheus{S}, \matheus{K}]$$ 
		of $\V$-models of $\matheus{S}$ in $\matheus{K}$ into $\V$-functors is reflective.
	\end{theorem}
	
	\begin{proof}
		According to \cite[Example 3.4]{Kelly:1982}, the functor $\V$-category $[\matheus{S}, \matheus{K}]$ is also locally presentable, thus, the (co)domains of the maps in \longref{Lemma}{lem:model_orthogonal} satisfy the smallness condition required for \longref{Corollary}{cor:special_EOST}.
		
		The result follows immediately from \longref{Lemma}{lem:model_orthogonal} and \longref{Corollary}{cor:special_EOST}.
	\end{proof}
	
	\begin{rk}
		Our \longref{Theorem}{thm:reflectivity} is a strengthening of \cite[Theorem 6.11]{book:Kelly:1982} which considers models only in the base of enrichment $\V$.
	\end{rk}
	
	\begin{rk}
		\label{rk:reflectivity}
		\longref{Theorem}{thm:reflectivity} is also a generalisation of \cite[Proposition 6.3.9]{LWP:2024} in some sense as follows.
		
		Let $\matheus{A} \hookrightarrow \matheus{E}$ be a small full dense sub-$\V$-category of a locally presentable $\V$-category. By \cite[Corollary 7.4]{BQR:1998}, $\matheus{E} \simeq {\matheus{M}\text{\fontshape{ui}\selectfont od}}(\matheus{I}, \V)$ for some limit $\V$-sketch $\matheus{I}$. According to \cite[Proposition 6.3.8]{LWP:2024}, if $\matheus{A} \hookrightarrow {\matheus{M}\text{\fontshape{ui}\selectfont od}}(\matheus{I}, \V)$ is a small full (dense) sub-$\V$-category containing the representables, then any $\matheus{A}$-pre-theory can be seen as a limit $\V$-sketch such that their models in the base of enrichment $\V$ coincide.  Restricting our limit $\V$-sketch $\matheus{S}$ to the limit $\V$-sketch associated to a $\matheus{A}$-pre-theory, and setting $\matheus{K} = \V$, we obtain the reflectivity statement in \cite[Proposition 6.3.9]{LWP:2024}.
	\end{rk}
	
	\section{Monadicity}
	\label{sec:monadicity}
	
	In this section, we investigate the monadicity of models of enhanced limit $2$-sketches, which will finally lead us to the fact that the enhanced $2$-category of models with loose transformations is the enhanced $2$-category of algebras over an enhanced $2$-monad.
	
	\subsection{The monadicity of $\V$-models}
	\label{sec:V_models_monadic}
	
	First, we observe that the adjunction in \longref{Theorem}{thm:reflectivity} is monadic under certain conditions.
	
	\begin{defi}
		Let $\matheus{S}$ and $\matheus{T}$ be two limit $\V$-sketches. A morphism $i \colon \matheus{S} \to \matheus{T}$ of limit $\V$-sketches \emph{reflects weighted cones} precisely if for any quadruple $(W, D, s, \gamma)$ in $\matheus{S}$, if its image $(W, i \cdot D, i(s), i \cdot \gamma)$ under $i$ is a weighted cone in $\matheus{T}$, then it is a  weighted cone in $\matheus{S}$.
	\end{defi}
	
	\begin{lemma}
		\label{lem:reflect_cones}
		Let $\matheus{K}$ be a complete $\V$-category. Let $\matheus{S}$ and $\matheus{T}$ be two limit $\V$-sketches, and $i \colon \matheus{S} \to \matheus{T}$ be a cone-reflecting morphism of limit $\V$-sketches. Then, for a $\V$-functor $H \colon \matheus{T} \to \matheus{K}$, $H$ is a $\V$-model with respect to $\calR$, where $\calR$ is a class of morphisms in $\V_0$ closed under isomorphisms and composition.
	\end{lemma}
	
	\begin{proof}
		This amounts to show that any weighted cone in $\matheus{S}$ is sent to a weighted cone in $\matheus{T}$ by $i$, and any weighted cone in $\matheus{S}$ comes from a weighted cone in $\matheus{T}$ by restricting along $i$. Since $i$ preserves and reflects weighed cones, we are done.
	\end{proof}
	
	\begin{pro}
		\label{pro:V_models_monadic}
		Let $\V$ be a category that satisfies all the required conditions in \longref{Theorem}{thm:reflectivity}.
		Let $i \colon \matheus{S} \to \matheus{T}$ be a morphism of limit $\V$-sketches, and $\matheus{K}$ be a locally presentable $\V$-category. There is an adjunction
		\[
		\begin{tikzcd}[ampersand replacement=\&]
			{\matheus{M}\text{\fontshape{ui}\selectfont od}(\matheus{S}, \matheus{K})} \&\& {\matheus{M}\text{\fontshape{ui}\selectfont od}(\matheus{T}, \matheus{K})}
			\arrow[""{name=0, anchor=center, inner sep=0}, "{L\cdot \Lan{i}}", bend left, from=1-1, to=1-3]
			\arrow[""{name=1, anchor=center, inner sep=0}, "{i^*}", bend left, from=1-3, to=1-1]
			\arrow["\dashv"{anchor=center, rotate=-90}, draw=none, from=0, to=1]
		\end{tikzcd},
		\]
		where $L$ denotes the reflector in \longref{Theorem}{thm:reflectivity}, and $\Lan{i}$ is the left Kan extension along $i$.
		
		Furthermore, if $i$ is essentially surjective and reflects weighted cones, then this adjunction is monadic.
	\end{pro}
	
	\begin{proof}
		We first show that the adjunction exists.
		
		Let $F \colon \matheus{T} \to \matheus{K}$ and $G \colon \matheus{S} \to \matheus{K}$ be two $\V$-models. Then we have
		\begin{align*}
			{\matheus{M}\text{\fontshape{ui}\selectfont od}(\matheus{T}, \matheus{K})}(L\cdot\Lan{i}G, F) &\cong [\matheus{T}, \matheus{K}](\Lan{i}G, F)
			\\
			&\cong [\matheus{S}, \matheus{K}](G, i^*F)
			\\
			&\cong {\matheus{M}\text{\fontshape{ui}\selectfont od}(\matheus{S}, \matheus{K})}(G, i^*F).
		\end{align*}
		
		Next, we show that the adjunction is monadic when $i$ is cone-reflecting.
		
		By \cite[Theorem II.2.1]{book:Dubuc:1970}, our main task is to verify that the $\V$-functor
		$${\matheus{M}\text{\fontshape{ui}\selectfont od}(\matheus{T}, \matheus{K})} \xrightarrow{i^*} {\matheus{M}\text{\fontshape{ui}\selectfont od}(\matheus{S}, \matheus{K})}$$
		creates $i^*$-split coequalisers.
		
		Suppose $F_1, F_2$ are $\V$-models $\matheus{T} \to \matheus{K}$, and that the pair
		\[
		\begin{tikzcd}[ampersand replacement=\&]
			{F_1} \& {F_2}
			\arrow["g"', shift right, from=1-1, to=1-2]
			\arrow["f", shift left, from=1-1, to=1-2]
		\end{tikzcd}
		\]
		consists of two $\V$-natural transformations $f$ and $g$ such that the pair $i^*f$ and $i^*g$ is part of the split coequaliser diagram
		\begin{equation}
			\label{diag:split_coeq}
			\begin{tikzcd}[ampersand replacement=\&]
				{F_1\cdot i} \& {F_2\cdot i} \& H
				\arrow["{g\cdot i}"', shift right, from=1-1, to=1-2]
				\arrow["{f\cdot i}", shift left, from=1-1, to=1-2]
				\arrow[""', curve={height=18pt}, dashed, from=1-2, to=1-1]
				\arrow["h"', dashed, from=1-2, to=1-3]
				\arrow[""', curve={height=12pt}, dashed, from=1-3, to=1-2, in = 30, out = 140]
			\end{tikzcd}
		\end{equation}
		in ${\matheus{M}\text{\fontshape{ui}\selectfont od}(\matheus{S}, \matheus{K})}$. Note that since any split coequaliser is absolute, so \longref{Diagram}{diag:split_coeq} is also a coequaliser diagram in $[\matheus{S}, \matheus{K}]$.
		Since $[\matheus{T}, \matheus{K}]$ is cocomplete, the coequaliser $F_2 \xrightarrow{\overline{h}} \overline{H}$ of $f$ and $g$ exists in $[\matheus{T}, \matheus{K}]$. Since the reflector $L_\matheus{T} \colon [\matheus{T}, \matheus{K}] \to {\matheus{M}\text{\fontshape{ui}\selectfont od}(\matheus{T}, \matheus{K})}$ defined in \longref{Theorem}{thm:reflectivity} is a left adjoint, $L_\matheus{T}(\overline{H})$ is a coequaliser of $f$ and $g$ in ${\matheus{M}\text{\fontshape{ui}\selectfont od}(\matheus{T}, \matheus{K})}$.
		
		Now, by \longref{Lemma}{lem:reflect_cones}, $i^*(L_\matheus{T}(\overline{H}))$ is a $\V$-model in ${\matheus{M}\text{\fontshape{ui}\selectfont od}(\matheus{S}, \matheus{K})}$. We check that $i^*(L_\matheus{T}(\overline{H})) \cong H$.
		
		For any $s \in \ob \matheus{S}$, the evaluation $\V$-functors $\mathrm{ev}_{i(s)} \colon [\matheus{T}, \matheus{K}] \to \matheus{K}$ and $\mathrm{ev}_s \colon [\matheus{S}, \matheus{K}] \to \matheus{K}$ must preserve colimits. So for all $s \in \ob \matheus{S}$, it is true that both $F_2(i(s)) \xrightarrow{\overline{h}_{i(s)}} \overline{H}(i(s))$ and $F_2(i(s)) \xrightarrow{{h_{s}}} {H}(s)$ are the coequaliser of $f_{i(s)}$ and $g_{i(s)}$. Now, as evaluations jointly reflect colimits, we have $\overline{H}\cdot i \cong H$ in $[\matheus{S}, \matheus{K}]$, which implies $L_\matheus{S}(i^*(\overline{H})) \cong L_\matheus{S}(H) \cong H$, where $L_\matheus{S} \colon [\matheus{S}, \matheus{K}] \to {\matheus{M}\text{\fontshape{ui}\selectfont od}(\matheus{S}, \matheus{K})}$ denotes the reflector. Since $i^* \colon [\matheus{T}, \matheus{K}] \to [\matheus{S}, \matheus{K}]$ admits a right adjoint, which is given by the right Kan extension, by \longref{Lemma}{lem:reflector_interact}, we conclude that $i^*(L_\matheus{T}(\overline{H})) \cong L_\matheus{S}(i^*(\overline{H})) \cong H$.
		
		Finally, suppose there is a cocone $F_2 \xrightarrow{} Z$ under $f$ and $g$ in ${\matheus{M}\text{\fontshape{ui}\selectfont od}(\matheus{T}, \matheus{K})}$ whose image under $i^*$ is isomorphic to the underlying cocone in \longref{Diagram}{diag:split_coeq}. Denote by $z \colon L_\matheus{T}(\overline{H}) \to Z$ the unique map in ${\matheus{M}\text{\fontshape{ui}\selectfont od}(\matheus{T}, \matheus{K})} \subseteq [\matheus{T}, \matheus{K}]$  induced by the universal property of $L_\matheus{T}(\overline{H})$. By assumption, $i^*(z) \colon  i^*(L_\matheus{T}(\overline{H})) \xrightarrow{\cong} i^*(Z)$ is an isomorphism in ${\matheus{M}\text{\fontshape{ui}\selectfont od}(\matheus{S}, \matheus{K})} \subseteq[\matheus{S}, \matheus{K}]$. Since $i$ is essentially surjective, so $i^* \colon [\matheus{T}, \matheus{K}] \to [\matheus{S}, \matheus{K}]$ reflects isomorphisms, which means $z \colon L_\matheus{T}(\overline{H}) \xrightarrow{\cong} Z$ is an isomorphism in $[\matheus{T}, \matheus{K}]$; as ${\matheus{M}\text{\fontshape{ui}\selectfont od}(\matheus{T}, \matheus{K})}$ is a full sub-category of $[\matheus{T}, \matheus{K}]$, we conclude that $z$ is an isomorphism in ${\matheus{M}\text{\fontshape{ui}\selectfont od}(\matheus{T}, \matheus{K})}$. So, $Z$ is the coequaliser of $f$ and $g$.
		%
	\end{proof}
	
	\begin{rk}
		Our \longref{Proposition}{pro:V_models_monadic} is a generalisation of \cite[Proposition 27]{BG:2019} for non-concrete models of pre-theories in the base of enrichment $\V$ in some sense as follows.
		
		Following \longref{Remark}{rk:reflectivity}, if $\matheus{A} \hookrightarrow {\matheus{M}\text{\fontshape{ui}\selectfont od}}(\matheus{I}, \V)$ is a small full (dense) sub-$\V$-category containing the representables, then any $\matheus{A}$-pre-theory can be seen as a limit $\V$-sketch such that their models in the base of enrichment $\V$ coincide.  Clearly, morphisms of $\matheus{A}$-pre-theories can then be seen as cone-reflecting morphisms of $\V$-sketches. Restricting our limit $\V$-sketches $\matheus{S}$ and $\matheus{T}$ to limit $\V$-sketches associated to $\matheus{A}$-pre-theories, and setting $\matheus{K} = \V$ in our proposition, we recover the monadicity result in \cite[Proposition 27]{BG:2019} for non-concrete models.
		%
			\end{rk}

			\subsection{Preliminaries on the theory of enhanced limit $2$-sketches}
			\label{sec:prelim}
			
			We move on to apply our results in the above sections to enhanced limit $2$-sketches and their models.
			
			
			We recall some basic notions in enhanced $2$-category theory, which was introduced in \cite{LS:2012} by Lack and Shulman.
			
			\begin{defi}
				Denote by $\F$ the full sub-category of the arrow category 
				$\mathbf{Cat}^\mathbf{2}$ 
				of the $1$-category	$\mathbf{Cat}$ of categories, determined by the fully faithful and injective-on-objects 
				functors, i.e., the \emph{full embeddings}.
			\end{defi}
			
			In other words, an object of $\F$ is a full embedding
			\begin{equation*}
				A_\tau \xhookrightarrow{j_A} A_\lambda,
			\end{equation*}
			a morphism $f$ from $j_A$ to $j_B$ in $\F$ is given by two functors 
			$f_\tau \colon A_\tau \to B_\tau$ and 
			$f_\lambda \colon A_\lambda \to B_\lambda$ 
			making the following square commute
			\begin{center}
				\begin{tikzcd}
					A_\tau \ar[r, hook, "j_A"] \ar[d, "f_\tau"'] & A_\lambda \ar[d, 
					"f_\lambda"]
					\\
					B_\tau \ar[r, hook, "j_B"'] & B_\lambda
				\end{tikzcd}.
			\end{center}
			
			We call $A_\tau$ the \emph{tight} part of $A$, and $A_\lambda$ the 
			\emph{loose} part of 
			$A$; similarly, we apply this terminology to $f$.
			
			\begin{rk}
				$\F$ is (co)complete and Cartesian closed. In particular, it is a Bénabou cosmos.
			\end{rk}
			
			\begin{lemma}
				\label{lem:F}
				The category $\F$ satisfies all the required conditions in \longref{Theorem}{thm:reflectivity}, i.e., $\F_0$ is equipped with a weak factorisation system $(\mathccal{L}, \calR)$ which is cofibrantly generated by a set $\mathccal{G}$ of morphisms whose (co)domains are locally presentable in the unenriched sense, and that the left class $\mathccal{L}$ is corner-stable. Explicitly, the left class $\mathccal{L}$ is given by all morphisms, whereas the right class $\calR$ is given by the isomorphisms.
			\end{lemma}
			
			\begin{proof}
				Clearly, a morphism $f \colon j_A \to j_B$ in $\F$ can be factorised as a map followed by an isomorphism.
				Consider the commutative diagram
				
				in $\F$. A lift $j_N \to j_S$ exists precisely when $r$ is an isomorphism, because we can define the lift by chasing the pre-image of $j_N \to j_T$. So, $\F$ is equipped with a weak factorisation system $(\mathccal{L}, \calR)$, where $\mathccal{L}$ consists of all maps, and that $\calR$ is precisely the isomorphisms.
				
				Note that the right class $\calR$ is cofibrantly generated by
				\begin{align*}
					\mathccal{G} :=
					\Bigggg\{
					&
					\begin{tikzcd}[ampersand replacement=\&]
						\emptyset \& \emptyset \\
						{\{\bullet\}} \& {\{\bullet\}}
						\arrow[equals, from=1-1, to=1-2]
						\arrow[hook, from=1-1, to=2-1]
						\arrow[hook, from=1-2, to=2-2]
						\arrow[equals, from=2-1, to=2-2]
					\end{tikzcd},
					\begin{tikzcd}[ampersand replacement=\&]
						{\{\bullet \quad \bullet\}} \& {\{\bullet \quad \bullet\}} \\
						{\{\bullet \to \bullet\}} \& {\{\bullet \to \bullet\}}
						\arrow[equals, from=1-1, to=1-2]
						\arrow[hook, from=1-1, to=2-1]
						\arrow[hook, from=1-2, to=2-2]
						\arrow[equals, from=2-1, to=2-2]
					\end{tikzcd},
					\begin{tikzcd}[ampersand replacement=\&]
						{\{\bullet \rightrightarrows \bullet\}} \& {\{\bullet \rightrightarrows \bullet\}} \\
						{\{\bullet \to \bullet\}} \& {\{\bullet \to \bullet\}}
						\arrow[equals, from=1-1, to=1-2]
						\arrow[from=1-1, to=2-1]
						\arrow[from=1-2, to=2-2]
						\arrow[equals, from=2-1, to=2-2]
					\end{tikzcd},
					\begin{tikzcd}[ampersand replacement=\&]
						{\{\bullet \quad \bullet\}} \& {\{\bullet \quad \bullet\}} \\
						{\{\bullet\}} \& {\{\bullet\}}
						\arrow[equals, from=1-1, to=1-2]
						\arrow[from=1-1, to=2-1]
						\arrow[from=1-2, to=2-2]
						\arrow[equals, from=2-1, to=2-2]
					\end{tikzcd},\\
					&\quad
					\begin{tikzcd}[ampersand replacement=\&]
						\emptyset \& \emptyset \\
						\emptyset \& {\{\bullet\}}
						\arrow[equals, from=1-1, to=1-2]
						\arrow[equals, from=1-1, to=2-1]
						\arrow[hook, from=1-2, to=2-2]
						\arrow[hook, from=2-1, to=2-2]
					\end{tikzcd}, \quad \quad
					\begin{tikzcd}[ampersand replacement=\&]
						\emptyset \& {\{\bullet \quad \bullet\}} \\
						\emptyset \& {\{\bullet \to \bullet\}}
						\arrow[hook, from=1-1, to=1-2]
						\arrow[equals, from=1-1, to=2-1]
						\arrow[hook, from=1-2, to=2-2]
						\arrow[hook, from=2-1, to=2-2]
					\end{tikzcd}, \quad \quad
					\begin{tikzcd}[ampersand replacement=\&]
						\emptyset \& {\{\bullet \rightrightarrows \bullet\}} \\
						\emptyset \& {\{\bullet \to \bullet\}}
						\arrow[hook, from=1-1, to=1-2]
						\arrow[equals, from=1-1, to=2-1]
						\arrow[from=1-2, to=2-2]
						\arrow[hook, from=2-1, to=2-2]
					\end{tikzcd}, \quad \quad
					\begin{tikzcd}[ampersand replacement=\&]
						\emptyset \& {\{\bullet \quad \bullet\}} \\
						\emptyset \& {\{\bullet\}}
						\arrow[hook, from=1-1, to=1-2]
						\arrow[equals, from=1-1, to=2-1]
						\arrow[from=1-2, to=2-2]
						\arrow[hook, from=2-1, to=2-2]
					\end{tikzcd}
					\Bigggg\},
				\end{align*}
				which ensures that the maps in $\calR$ are precisely the isomorphisms on both the tight parts and the loose parts.
				
				
				Furthermore, $\mathccal{L}$ is clearly corner-stable, as it contains every maps in $\F$.
			\end{proof}
			
			\begin{defi}
				An \emph{enhanced $2$-category} $\bbA$ is a category $\bbA$ enriched in $\F$. 
			\end{defi}
			
			This 
			means $\bbA$ has
			\begin{enumerate}
				\item [$\bullet$] objects $x, y, \cdots$;
				\item [$\bullet$] hom-objects $\bbA(x, y)$ in $\F$, each consists 
				of a full embedding $\bbA(x, y)_\tau \hookrightarrow \bbA(x, 
				y)_\lambda$ of the tight part into the loose part.
			\end{enumerate}
			
			We can form a $2$-category $\calA_\tau$ as follows:
			
			\begin{enumerate}
				\item [$\bullet$] $\calA_\tau$ has all the objects of $\bbA$;
				\item [$\bullet$] the hom-categories $\calA_\tau (x, y)$ for any 
				objects $x, y$ are $\bbA(x, y)_\tau$.
			\end{enumerate}
			
			Similarly, we can form a $2$-category $\calA_\lambda$ by setting the 
			hom-categories $\calA_\lambda (x, y)$  as $\bbA(x, y)_\lambda$ for any 
			objects $x, y$.
			
			Since for each pair of objects $x, y$, $\calA_\tau (x, y) \hookrightarrow 
			\calA_\lambda (x, y)$ is a full-embedding, we obtain a $2$-functor
			
			\begin{align*}
				J_\bbA \colon \calA_\tau &\to \calA_\lambda.
			\end{align*}
			
			By construction, $J_\bbA$ is identity-on-objects, faithful, and locally 
			fully faithful.
			
			\begin{rk}
				We may identify an $\F$-category $\bbA$ with $J_\bbA$. Indeed, any 
				$2$-functor which is identity-on-objects, faithful, and locally 
				fully faithful uniquely determined an $\F$-category.
			\end{rk}
			
			The morphisms in $\calA_\tau$ are called the \emph{tight} morphisms, 
			whereas 
			those in $\calA_\lambda$ are called \emph{loose}.
			
			\begin{nota}
				We write $A \to B$ for a tight morphism from $A$ to $B$, and $A 
				\leadsto B$ for a loose morphism from $A$ to $B$.
			\end{nota}
			
			\begin{nota}
				When we view $\F$ as an enhanced $2$-category, we write $\bbF$.
			\end{nota}
			
			Let $\bbA$ and $\bbB$ be two $\F$-categories. An \emph{$\F$-functor} 
			$F\colon 
			\bbA \to \bbB$ is a functor enriched in $\F$, which precisely means that 
			$F$ consists 
			of $2$-functors $F_\tau \colon \calA_\tau \to 
			\calB_\tau$ and $F_\lambda \colon \calA_\lambda \to \calB_\lambda$ 
			making the following diagram commute
			
			\begin{equation}
				\label{diag:F_functor}
				\begin{tikzcd}
					\calA_\tau \ar[r, hook, "J_\bbA"] \ar[d, "F_\tau"'] & \calA_\lambda 
					\ar[d, 
					"F_\lambda"]
					\\
					\calB_\tau \ar[r, hook, "J_\bbB"'] & \calB_\lambda
				\end{tikzcd}.
			\end{equation}
			
			
			Let $F, G\colon \bbA \rightrightarrows \bbB$ be two $\F$-functors. An 
			\emph{{$\F$}-natural 
				transformation}	$\alpha \colon F \to G$ then consists of a $2$-natural $\alpha_\lambda \colon F_\lambda \to G_\lambda$ where every $1$-component is tight.

			In a joint work \cite{ABK:2024} with Arkor and Bourke, we introduce and develop the theory of enhanced limit $2$-sketches, which demonstrates enhanced $2$-category theory as a powerful tool in capturing numerous $2$-dimensional structures and morphisms that involve a mixture of strictness and laxity.
			
			\begin{defi}
				An \emph{enhanced limit 2-sketch} is a limit $\F$-sketch $\mathmybb{S}$ in the sense of \longref{Definition}{def:V-sketch}.
				
				Let $\mathmybb{K}$ be a complete $\F$-category. An $\F$-model is an $\F$-functor $F \colon \mathmybb{S} \to \mathmybb{K}$ that sends all weighted cones in $\mathmybb{S}$ to limit cones in $\mathmybb{K}$.
			\end{defi}
			
			\begin{nota}
				The full sub-$\F$-category of the functor $\F$-category $[\mathmybb{S}, \mathmybb{K}]$ spanned by the $\F$-models is denoted as $\FMod{}(\mathmybb{S}, \mathmybb{K})$.
			\end{nota}
			
			Most of the time, however, we are interested in some weaker notions of morphisms, which could not be expressed simply via the functor $\V$-category.
			
			\begin{nota}
				We write $s, p, l, c$ for `strict', `pseudo', `lax', and `colax', respectively.
				
				For a pair $w', w$ of weaknesses in $\{s, p, l, c\}$, if $w'$ is stricter than $w$, then we write $w' \leq w$.
				
				We denote by $\overline{w}$ the dual of $w$, i.e., 
				\begin{equation*}
					\begin{aligned}
						\overline{s} = {s}, & &
						\overline{p} = {p}, & &
						\overline{l} = c, & &
						\overline{c} = l.
					\end{aligned}
				\end{equation*}
			\end{nota}

			\begin{defi}[{\cite[\S4.1]{LS:2012}, \cite[Definition 5.9]{ABK:2024}}]
				\label{def:functor-F-category}
				Let $w' \leq w \in \{s, p, l, c\}$ be a pair of weaknesses. Let $\mathmybb{S}$ and $\mathmybb{K}$ be small $\F$-categories. Denote by $\mathmybb{Fun}_{w', w}(\mathmybb{S}, \mathmybb{K})$ the small $\F$-category of $\F$-functors and \emph{loose $(w', w)$-natural transformations}, defined as follows.
				\begin{enumerate}
					\item Objects are $\F$-functors $\mathmybb{S} \to \mathmybb{K}$.
					\item Loose morphisms $\phi \colon M \leadsto N$ are $w$-natural
					transformations $\phi \colon M \to N$, which comprise $w$-natural families
					$\{ \phi_S \colon M(S) \leadsto N(S) \}_{S \in \mathmybb{S}}$, that become
					$w'$-natural on tight morphisms in $\mathmybb{S}$. These are called \emph{loose $(w', w)$-natural transformations}.
					
					\item A morphism $\phi \colon M \leadsto N$ is tight when each component $\phi_S
					\colon M(S) \leadsto N(S)$ is tight and each $\phi_s$ is a $w'$-cell.
					\item 2-cells are modifications.
				\end{enumerate}
			\end{defi}
			
			\begin{nota}
				The full sub-$\F$-category of $\mathmybb{Fun}_{w', w}(\mathmybb{S}, \mathmybb{K})$ spanned by the $\F$-models is denoted as $\FMod{s, w}(\mathmybb{S}, \mathmybb{K})$.
			\end{nota}
			
			\begin{example}[{\cite{ABK:2024}}]
				Examples of models of enhanced limit $2$-sketches include pseudo double categories, monoidal double categories, horizontal or vertical intercategories, and double (op)fibrations.
				
				Examples of loose $w$-natural transformations include $w$-monoidal functors between monoidal categories, $w$-double functors between double categories, and $w$-morphisms of fibrations between double (op)fibrations.
			\end{example}

			\subsection{Preliminaries on $2$-enhanced $2$-categories}
			\label{sec:F2}
			
			To capture the behaviour of $\FMod{s, w}(\mathmybb{T}, \bbK)$; we need to go further into $2$-enhanced $2$-categories, which are $2$-categories with \emph{three} types of $1$-morphisms.
			
			\begin{defi}
				Denote by $\mathbf{Cat}^\mathbf{3}$ the category of composable arrows in $\Cat$.
				Denote by $\F_2$ the full sub-category of $\mathbf{Cat}^\mathbf{3}$ determined by the full embeddings followed by full embeddings.
			\end{defi}
			
			In other words, an object of $\F_2$ is a pair of composable full embeddings
			\begin{equation*}
				A_\tau \xhookrightarrow{j_{A_\tau}} A_\theta \xhookrightarrow{j_{A_\theta}} A_\lambda,
			\end{equation*}
			denoted as $j_A$, and that a morphism $f$ from $j_A$ to $j_B$ in $\F$ is given by three functors 
			$f_\tau \colon A_\tau \to B_\tau$,  $f_\theta \colon A_\theta \to B_\theta$, and
			$f_\lambda \colon A_\lambda \to B_\lambda$ such that the diagram
			\[
			\begin{tikzcd}[ampersand replacement=\&]
				{A_\tau} \& {A_\theta} \& {A_\lambda} \\
				{B_\tau} \& {B_\theta} \& {B_\lambda}
				\arrow["{j_{A_\tau}}", hook, from=1-1, to=1-2]
				\arrow["{f_\tau}"', from=1-1, to=2-1]
				\arrow["{j_{A_\theta}}", hook, from=1-2, to=1-3]
				\arrow["{f_\theta}", from=1-2, to=2-2]
				\arrow["{f_\lambda}", from=1-3, to=2-3]
				\arrow["{j_{B_\tau}}"', hook, from=2-1, to=2-2]
				\arrow["{j_{B_\theta}}"', hook, from=2-2, to=2-3]
			\end{tikzcd}
			\]
			commutes

			We call $A_\tau$ the \emph{tight} part of $A$, $A_\theta$ the \emph{fit} part of $A$, and $A_\lambda$ the 
			\emph{loose} part of 
			$A$; similarly, we apply this terminology to $f$.
			
			\begin{rk}
				$\F_2$ is (co)complete and Cartesian closed.
			\end{rk}
			
			\begin{defi}
				A \emph{$2$-enhanced $2$-category} $\mathbbb{A}$ is a category  enriched in $\F_2$. 
			\end{defi}
			
			This 
			means $\mathbbb{A}$ has
			\begin{enumerate}
				\item [$\bullet$] objects $x, y, \cdots$;
				\item [$\bullet$] hom-objects $\mathbbb{A}(x, y)$ in $\F_2$, each consists 
				of a pair of composable full embeddings $\mathbbb{A}(x, y)_\tau \hookrightarrow \mathbbb{A}(x, y)_\theta \hookrightarrow \mathbbb{A}(x, 
				y)_\lambda$ of the tight part into the fit part into the loose part.
			\end{enumerate}
			
			%
			
			%
			%
			%
			%
			%
			%
			
			An \emph{$\F_2$-functor} 
			$F\colon 
			\mathbbb{A} \to \mathbbb{B}$ is a functor enriched in $\F_2$, which precisely means that $F$ is a $2$-functor between the loose parts that preserves tight and fit morphisms.
			
			%
			
			\subsection{The equivalence between \texorpdfstring{$\FMod{s, w}(\mathmybb{T}, \bbK)$}{Modₛ,ᵥᵥ (𝕊, 𝕂)} and \texorpdfstring{$\FTAlg_{s, w}$}{T-Algₛ,ᵥᵥ}}
			\label{sec:equivalence}
			
			Our result involves the concept of enhanced $2$-monads and algebras over them; we briefly recall the basics.
			%
			
			\begin{defi}
				\label{def:F-monad}
				An \emph{enhanced $2$-monad} is an $\F$-monad $T \colon \bbA \to \bbA$ on an enhanced $2$-category $\bbA$, i.e., a monad in the $2$-category $\FCat$ of enhanced $2$-categories.
			\end{defi}
			
			
			This means that $T$ consists of an $\F$-functor, together with a multiplication and a unit, which are both $\F$-natural transformations, satisfying the usual axioms for a monad in a $2$-category.
			
			Just as in $2$-monad theory, we also have the notion of strict $T$-algebras over enhanced $2$-monads. Since we only consider strict $T$-algebras in this paper, we would simply call them $T$-algebras.
			
			The notions of \emph{strict, pseudo, lax, colax} $T$-morphisms in the enhanced $2$-categorical context is introduced in \cite[\S4.3]{LS:2012}, and also discussed in \cite[\S6.3]{Bourke:2014}. A strict $T$-morphism $f$ from $(A, a \colon TA \to A)$ to $(B, b \colon TB \to B)$ consists of a commutative diagram
			\[
			\begin{tikzcd}[ampersand replacement=\&]
				TA \& TB \\
				A \& B
				\arrow["Tf", from=1-1, to=1-2]
				\arrow["a"', from=1-1, to=2-1]
				\arrow["b", from=1-2, to=2-2]
				\arrow["f"', from=2-1, to=2-2]
			\end{tikzcd},
			\]
			whereas a $w$-$T$-morphism is a $w$-morphism
			\[
			\begin{tikzcd}[ampersand replacement=\&]
				TA \& TB \\
				A \& B
				\arrow["Tf", loose, from=1-1, to=1-2]
				\arrow["a"', from=1-1, to=2-1]
				\arrow["w"{description}, draw=none, from=1-2, to=2-1]
				\arrow["b", from=1-2, to=2-2]
				\arrow["f"', loose, from=2-1, to=2-2]
			\end{tikzcd}.
			\]
			In other words, in the $2$-categorical sense, strict $T$-morphisms are strict $T_\tau$-morphisms, and $w$-$T$-morphisms are $w$-$T_\lambda$-morphisms.

			\begin{nota}
				Let $w \in \{s, p, l, c\}$.
				We denote by $\FTAlg_{s, w}$ the enhanced $2$-category of $T$-algebras, where the tight morphisms are given by strict $T$-morphisms, and the loose morphisms are given by $w$-$T$-morphisms; the $2$-morphisms are given by the $T_\lambda$-transformations.
			\end{nota}
			
			We are now ready to establish one of our main results in this article.
			
			\begin{theorem}
				\label{thm:equivalence}
				Let $\bbK$ be a locally presentable enhanced $2$-category. Let $\mathmybb{T}$ be an enhanced limit $2$-sketch with tight weighted cones, i.e., the shapes of weighted cones in $\mathmybb{T}$ are $2$-categories, viewed as chordate $\F$-categories. There is an enhanced $2$-monad $T$ on the enhanced $2$-category $\FMod{}(\mathccal{T}_\tau, \bbK)$ of models restricted to the tight morphisms such that
				$$\FMod{s, w}(\mathmybb{T}, \bbK) \simeq \FTAlg_{s, w}$$
				is an equivalence of enhanced $2$-categories.
			\end{theorem}
			
			\begin{proof}
				First, we have an $\F_2$-category $\mathbbb{Mod}(\mathmybb{T}, \mathmybb{K})$, where the tight, fit, and loose morphisms are the $\F$-natural transformations, the $2$-natural transformations (between the loose parts), and the loose $w$-natural transformations, respectively.
				
				Note that the inclusion $i \colon \mathccal{T}_\tau \hookrightarrow \mathmybb{T}$ of the tight part of $\mathmybb{T}$, viewed as a chordate $\F$-category, into $\mathmybb{T}$ reflects weighted cones, and is essentially surjective. Consider the pre-composition
				$$\FMod{}(\mathmybb{T}, \mathmybb{K}) \xrightarrow{i^*} \FMod{}(\mathccal{T}_\tau, \mathmybb{K})$$
				by $i$; by \longref{Lemma}{lem:F} and \longref{Proposition}{pro:V_models_monadic}, it is monadic.
				
				Denote by $i^*_{s, w}$ the pre-composition
				$$\FMod{s, w}(\mathmybb{T}, \mathmybb{K}) \xrightarrow{i^*} \FMod{s, w}(\mathccal{T}_\tau, \mathmybb{K})$$
				by $i$ of $\FMod{s, w}(\mathmybb{T}, \mathmybb{K})$. Observe that $\FMod{s, w}(\mathccal{T}_\tau, \mathmybb{K}) = \FMod{}(\mathccal{T}_\tau, \mathmybb{K})$, and so we have
				\[
				\begin{tikzcd}[ampersand replacement=\&]
					{\FMod{}(\mathmybb{T}, \mathmybb{K})} \&\& {\FMod{s, w}(\mathmybb{T}, \mathmybb{K})} \\
					\& {\FMod{}(\mathccal{T}_\tau, \mathmybb{K})}
					\arrow[hook, from=1-1, to=1-3]
					\arrow["{i^*}"', from=1-1, to=2-2]
					\arrow["{i^*_{s, w}}", from=1-3, to=2-2]
				\end{tikzcd},
				\]
				which means our $\F_2$-category $\mathbbb{Mod}(\mathmybb{T}, \mathmybb{K})$ fits into the setting in \cite[Theorem 26]{Bourke:2014}. 
				
				Together with \cite[Proposition 3]{Bourke:2014}, it suffices to check that $\FMod{}(\mathmybb{T}, \mathmybb{K})$, $\FMod{}(\mathccal{T}_\tau, \mathmybb{K})$, and $\FMod{s, w}(\mathmybb{T}, \mathmybb{K})$ all admit $\overline{w}$-limits of loose morphisms, and that $i^*_{s, w}$ is locally faithful, reflects identity $2$-morphisms, and satisfies weak $w$-doctrinal adjunction.
				
				Since $\bbK$ is complete, and also the inclusion $\FMod{}(\mathmybb{S}, \bbK) \hookrightarrow [\mathmybb{S}, \bbK]$ is reflective by \longref{Proposition}{pro:V_models_monadic} for any enhanced limit $2$-sketch $\mathmybb{S}$, it is immediate that both  $\FMod{}(\mathmybb{T}, \mathmybb{K})$ and $\FMod{}(\mathccal{T}_\tau, \mathmybb{K})$ have $\overline{w}$-limits of loose morphisms.
				
				As we assume that all the weighted cones in $\mathmybb{T}$ are tight, following the upcoming \longref{Proposition}{pro:colax_lim_arrow}, we conclude that $\FMod{s, w}(\mathmybb{T}, \mathmybb{K})$ admits $\overline{w}$-limits of loose morphisms.
				
				Moreover, since $i$ is bijective-on-objects, it is clear that $i^*_{s, w}$ is locally faithful and reflects identity $2$-morphisms.
				
				It then remains to check that $i^*_{s, w}$ satisfies weak $w$-doctrinal adjunction. For simplicity, we prove this statement when $w = l$; the other cases could be argued similarly.
				
				Let $F, G \colon \mathmybb{T} \rightrightarrows \bbK$ be $\F$-models, $\alpha \colon F \to G$ be an $\F$-natural transformation, and $\beta \colon G \cdot i \to F \cdot i$ be a $2$-natural transformation (between the loose parts). Suppose that $(\varepsilon, \alpha \cdot i \dashv \beta, \eta)$ is an adjunction in $\FMod{}(\mathccal{T}_\tau, \mathmybb{K})$.
				
				We construct a loose lax natural transformation $\overline{\beta} \colon G \to F$ by defining the $1$-component $\overline{\beta}_T$ at any object $T \in \mathmybb{T}$ as $\beta_T$, and the $2$-component $\overline{\beta}_t$ at a loose morphism $t \colon T_1 \leadsto T_2$ in $\mathmybb{T}$ as the mate
				\[
				\begin{tikzcd}[ampersand replacement=\&]
					\& {G(T_1)} \& {G(T_2)} \& {G(T_2)} \\
					{F(T_1)} \& {F(T_1)} \& {F(T_2)}
					\arrow["{G(t)}", loose, from=1-2, to=1-3]
					\arrow["{\alpha_{T_1}}"{description}, from=1-2, to=2-2]
					\arrow[""{name=0, anchor=center, inner sep=0}, equals, from=1-3, to=1-4]
					\arrow["{\alpha_{T_2}}"{description}, from=1-3, to=2-3]
					\arrow[""{name=1, anchor=center, inner sep=0}, "{\overline{\beta}_{T_1} = {\beta}_{T_1}}", from=2-1, to=1-2]
					\arrow[""{name=2, anchor=center, inner sep=0}, equals, from=2-1, to=2-2]
					\arrow["{F(t)}"', loose, from=2-2, to=2-3]
					\arrow[""{name=3, anchor=center, inner sep=0}, "{\overline{\beta}_{T_2} = {\beta}_{T_2}}"', from=2-3, to=1-4]
					\arrow["{\eta_{T_2}}"', between={0.2}{0.8}, Rightarrow, from=0, to=3]
					\arrow["{\varepsilon_{T_1}}", between={0.2}{0.8}, Rightarrow, from=1, to=2]
				\end{tikzcd}
				\]
				of $\alpha_t$. Note that since $\eta$ is a modification, we have, for a tight morphism $t \colon T_1 \to T_2$,
				\[
				\begin{tikzcd}[ampersand replacement=\&]
					{G(T_1)} \&\& {G(T_2)} \& \\
					\&\&\& {F(T_2)} \\
					{G(T_1)} \&\& {G(T_2)}
					\arrow["{G(t)}", from=1-1, to=1-3]
					\arrow[equals, from=1-1, to=3-1]
					\arrow["{\alpha_{T_2}}", from=1-3, to=2-4]
					\arrow[""{name=0, anchor=center, inner sep=0}, equals, from=1-3, to=3-3]
					\arrow["{\beta_{T_2}}", from=2-4, to=3-3]
					\arrow["{G(t)}"', from=3-1, to=3-3]
					\arrow["{\eta_{T_2}}"', between={0.4}{0.8}, Rightarrow, from=0, to=2-4]
				\end{tikzcd}
				=
				\begin{tikzcd}[ampersand replacement=\&]
					{G(T_1)} \&\& {G(T_2)} \& \\
					\& {F(T_1)} \&\& {F(T_2)} \\
					{G(T_1)} \&\& {G(T_2)}
					\arrow["{G(t)}", from=1-1, to=1-3]
					\arrow["{\alpha_{T_1}}", from=1-1, to=2-2]
					\arrow[""{name=0, anchor=center, inner sep=0}, equals, from=1-1, to=3-1]
					\arrow["{\alpha_t}"{description}, draw=none, from=1-3, to=2-2]
					\arrow["{\alpha_{T_2}}", from=1-3, to=2-4]
					\arrow["{F(t)}"', from=2-2, to=2-4]
					\arrow["{\beta_{T_1}}", from=2-2, to=3-1]
					\arrow["{\beta_t}"{description}, draw=none, from=2-2, to=3-3]
					\arrow["{\beta_{T_2}}", from=2-4, to=3-3]
					\arrow["{G(t)}"', from=3-1, to=3-3]
					\arrow["{\eta_{T_1}}"', between={0.4}{0.8}, Rightarrow, from=0, to=2-2]
				\end{tikzcd},
				\]
				and so if $t$ is tight, we have
				\[
				\begin{tikzcd}[ampersand replacement=\&]
					\& {G(T_1)} \& {G(T_2)} \& {G(T_2)} \\
					{F(T_1)} \& {F(T_1)} \& {F(T_2)}
					\arrow["{G(t)}", from=1-2, to=1-3]
					\arrow["{\alpha_{T_1}}"{description}, from=1-2, to=2-2]
					\arrow[""{name=0, anchor=center, inner sep=0}, equals, from=1-3, to=1-4]
					\arrow["{\alpha_{T_2}}"{description}, from=1-3, to=2-3]
					\arrow[""{name=1, anchor=center, inner sep=0}, "{\overline{\beta}_{T_1} = {\beta}_{T_1}}", from=2-1, to=1-2]
					\arrow[""{name=2, anchor=center, inner sep=0}, equals, from=2-1, to=2-2]
					\arrow["{F(t)}"', from=2-2, to=2-3]
					\arrow[""{name=3, anchor=center, inner sep=0}, "{\overline{\beta}_{T_2} = {\beta}_{T_2}}"', from=2-3, to=1-4]
					\arrow["{\eta_{T_2}}"', between={0.2}{0.8}, Rightarrow, from=0, to=3]
					\arrow["{\varepsilon_{T_1}}", between={0.2}{0.8}, Rightarrow, from=1, to=2]
				\end{tikzcd}
				=
				\begin{tikzcd}[ampersand replacement=\&]
					\& {G(T_1)} \& {G(T_1)} \& {G(T_2)} \\
					{F(T_1)} \& {F(T_1)}
					\arrow[""{name=0, anchor=center, inner sep=0}, equals, from=1-2, to=1-3]
					\arrow["{\alpha_{T_1}}"{description}, from=1-2, to=2-2]
					\arrow["{G(t)}", from=1-3, to=1-4]
					\arrow[""{name=1, anchor=center, inner sep=0}, "{\overline{\beta}_{T_1} = {\beta}_{T_1}}", from=2-1, to=1-2]
					\arrow[""{name=2, anchor=center, inner sep=0}, equals, from=2-1, to=2-2]
					\arrow[""{name=3, anchor=center, inner sep=0}, "{\overline{\beta}_{T_1} = {\beta}_{T_1}}"', from=2-2, to=1-3]
					\arrow["{\eta_{T_1}}"', between={0.2}{0.8}, Rightarrow, from=0, to=3]
					\arrow["{\varepsilon_{T_1}}", between={0.2}{0.8}, Rightarrow, from=1, to=2]
				\end{tikzcd},
				\]
				which gives the identity by the triangle identities. The naturality of $\overline{\beta}$ is inherited from $\alpha$. Altogether, $\overline{\beta}$ is a loose lax natural transformation.
				
				Now, we check that $\eta$ can be extended to a modification $1 \to \overline{\beta} \cdot \alpha$. In fact, for a loose morphisms $t \colon T_1 \leadsto T_2$ in $\mathmybb{T}$, we have
				\begin{align*}
					\begin{tikzcd}[ampersand replacement=\&]
						{G(T_1)} \&\& {G(T_2)} \& \\
						\& {F(T_1)} \&\& {F(T_2)} \\
						{G(T_1)} \&\& {G(T_2)}
						\arrow["{G(t)}", loose, from=1-1, to=1-3]
						\arrow["{\alpha_{T_1}}", from=1-1, to=2-2]
						\arrow[""{name=0, anchor=center, inner sep=0}, equals, from=1-1, to=3-1]
						\arrow["{\alpha_t}"{description}, draw=none, from=1-3, to=2-2]
						\arrow["{\alpha_{T_2}}", from=1-3, to=2-4]
						\arrow["{F(t)}"', loose, from=2-2, to=2-4]
						\arrow["{\overline{\beta}_{T_1}}", from=2-2, to=3-1]
						\arrow["{\overline{\beta}_{T_2}}", from=2-4, to=3-3]
						\arrow["{\overline{\beta}_t}"', between={0.45}{0.55}, Rightarrow, from=3-1, to=2-4]
						\arrow["{G(t)}"', loose, from=3-1, to=3-3]
						\arrow["{\eta_{T_1}}"', between={0.4}{0.8}, Rightarrow, from=0, to=2-2]
					\end{tikzcd}
					&=
					\begin{tikzcd}[ampersand replacement=\&]
						{G(T_1)} \&\& {G(T_2)} \& \\
						\& {F(T_1)} \& {F(T_1)} \& {F(T_2)} \\
						{G(T_1)} \& {G(T_2)} \& {G(T_2)}
						\arrow["{G(t)}", loose, from=1-1, to=1-3]
						\arrow["{\alpha_{T_1}}", from=1-1, to=2-2]
						\arrow[""{name=0, anchor=center, inner sep=0}, equals, from=1-1, to=3-1]
						\arrow["{\alpha_t}"{description}, draw=none, from=1-3, to=2-2]
						\arrow["{\alpha_{T_2}}", from=1-3, to=2-4]
						\arrow[equals, from=2-2, to=2-3]
						\arrow["{\overline{\beta}_{T_1}}"{description}, from=2-2, to=3-1]
						\arrow["{F(t)}", loose, from=2-3, to=2-4]
						\arrow["{\overline{\beta}_{T_2}}", from=2-4, to=3-3]
						\arrow[""{name=1, anchor=center, inner sep=0}, "{\alpha_{T_1}}"', from=3-1, to=2-3]
						\arrow["{G(t)}"', loose, from=3-1, to=3-2]
						\arrow[""{name=2, anchor=center, inner sep=0}, "{\alpha_{T_2}}", from=3-2, to=2-4]
						\arrow[equals, from=3-2, to=3-3]
						\arrow["{\eta_{T_1}}"', between={0.4}{0.8}, Rightarrow, from=0, to=2-2]
						\arrow["{\varepsilon_{T_1}}", between={0.2}{1}, Rightarrow, from=1, to=2-2]
						\arrow["{\eta_{T_2}}", between={0}{0.8}, Rightarrow, from=3-3, to=2]
					\end{tikzcd}
					\\
					&=
					\begin{tikzcd}[ampersand replacement=\&]
						{G(T_1)} \&\& {G(T_2)} \& \\
						\&\&\& {F(T_2)} \\
						{G(T_1)} \&\& {G(T_2)}
						\arrow["{G(t)}", loose, from=1-1, to=1-3]
						\arrow[equals, from=1-1, to=3-1]
						\arrow["{\alpha_{T_2}}", from=1-3, to=2-4]
						\arrow[""{name=0, anchor=center, inner sep=0}, equals, from=1-3, to=3-3]
						\arrow["{\overline{\beta}_{T_2}}", from=2-4, to=3-3]
						\arrow["{G(t)}"', loose, from=3-1, to=3-3]
						\arrow["{\eta_{T_2}}"', between={0.4}{0.8}, Rightarrow, from=0, to=2-4]
					\end{tikzcd},
				\end{align*}
				which gives the modification axiom. Similarly,
				\begin{align*}
					\begin{tikzcd}[ampersand replacement=\&]
						\& {F(T_1)} \&\& {F(T_2)} \\
						{G(T_1)} \&\& {G(T_2)} \\
						\& {F(T_1)} \&\& {F(T_2)}
						\arrow["{F(t)}", loose, from=1-2, to=1-4]
						\arrow["{\alpha_t}"{description}, draw=none, from=1-2, to=2-3]
						\arrow[""{name=0, anchor=center, inner sep=0}, equals, from=1-4, to=3-4]
						\arrow["{\alpha_{T_1}}", from=2-1, to=1-2]
						\arrow["{G(t)}", loose, from=2-1, to=2-3]
						\arrow["{\alpha_{T_2}}", from=2-3, to=1-4]
						\arrow["{\overline{\beta}_{T_1}}", from=3-2, to=2-1]
						\arrow["{\overline{\beta}_t}", between={0.2}{0.7}, Rightarrow, from=3-2, to=2-3]
						\arrow["{F(t)}"', loose, from=3-2, to=3-4]
						\arrow["{\overline{\beta}_{T_2}}", from=3-4, to=2-3]
						\arrow["{\varepsilon_{T_2}}"', between={0.3}{0.7}, Rightarrow, from=2-3, to=0]
					\end{tikzcd}
					&=
					\begin{tikzcd}[ampersand replacement=\&]
						\& {F(T_1)} \&\& {F(T_2)} \\
						{G(T_1)} \& {G(T_2)} \& {G(T_2)} \\
						\& {F(T_1)} \& {F(T_1)} \& {F(T_2)}
						\arrow["{F(t)}", loose, from=1-2, to=1-4]
						\arrow["{\alpha_t}"{description}, draw=none, from=1-2, to=2-3]
						\arrow[""{name=0, anchor=center, inner sep=0}, equals, from=1-4, to=3-4]
						\arrow["{\alpha_{T_1}}", from=2-1, to=1-2]
						\arrow["{G(t)}", loose, from=2-1, to=2-2]
						\arrow[""{name=1, anchor=center, inner sep=0}, "{\alpha_{T_1}}", from=2-1, to=3-3]
						\arrow[equals, from=2-2, to=2-3]
						\arrow[""{name=2, anchor=center, inner sep=0}, "{\alpha_{T_2}}"', from=2-2, to=3-4]
						\arrow["{\alpha_{T_2}}", from=2-3, to=1-4]
						\arrow["{\overline{\beta}_{T_1}}", from=3-2, to=2-1]
						\arrow[equals, from=3-2, to=3-3]
						\arrow["{F(t)}"', loose, from=3-3, to=3-4]
						\arrow["{\overline{\beta}_{T_2}}"{description}, from=3-4, to=2-3]
						\arrow["{\varepsilon_{T_1}}", between={0.2}{1}, Rightarrow, from=1, to=3-2]
						\arrow["{\varepsilon_{T_2}}"', between={0.3}{0.7}, Rightarrow, from=2-3, to=0]
						\arrow["{\eta_{T_2}}"', between={0}{0.8}, Rightarrow, from=2-3, to=2]
					\end{tikzcd}
					\\
					&=
					\begin{tikzcd}[ampersand replacement=\&]
						\& {F(T_1)} \&\& {F(T_2)} \\
						{G(T_1)} \\
						\& {F(T_1)} \&\& {F(T_2)}
						\arrow["{F(t)}", loose, from=1-2, to=1-4]
						\arrow[""{name=0, anchor=center, inner sep=0}, equals, from=1-2, to=3-2]
						\arrow[equals, from=1-4, to=3-4]
						\arrow["{\alpha_{T_1}}", from=2-1, to=1-2]
						\arrow["{\overline{\beta}_{T_1}}", from=3-2, to=2-1]
						\arrow["{F(t)}"', loose, from=3-2, to=3-4]
						\arrow["{\varepsilon_{T_1}}"', between={0.3}{0.7}, Rightarrow, from=2-1, to=0]
					\end{tikzcd},
				\end{align*}
				so $\varepsilon$ also extends to a modification $\alpha \cdot \overline{\beta} \to 1$. Consequently, we conclude that $\alpha \dashv \overline{\beta}$ is an adjunction in $\FMod{s, l}(\mathmybb{T}, \bbK)$, and hence, $i^*_{s, w}$ satisfies weak $l$-adjunction. This argument also applies to the case when $w = p, c$.
				
				As a result, we obtain a commutative diagram
				\begin{equation}
					\label{diag:equiv}
					\begin{tikzcd}[ampersand replacement=\&]
						{\FMod{s, w}(\mathbb{T}, \mathbb{K})} \&\& {\FTAlg_{s, w}} \\
						\& {\FMod{}(\mathcal{T}_\tau, \mathbb{K})}
						\arrow["\simeq", from=1-1, to=1-3]
						\arrow["{i^*_{s, w}}"', from=1-1, to=2-2]
						\arrow["{U_{s, w}}", from=1-3, to=2-2]
					\end{tikzcd}
				\end{equation}
				of enhanced $2$-categories, where the top row is an equivalence, and that $U_{s, w}$ denotes the canonical underlying $\F$-functor.
			\end{proof}
			
			\begin{rk}
				\label{rk:tight_cones_1}
				We specifically assume that the enhanced $2$-sketch $\mathmybb{T}$ has only tight weighted cones. In fact, this assumption is not too restrictive in practice. As explained in \cite[Remark 5.11]{ABK:2024}, the tight morphisms in an enhanced $2$-sketch are typically precisely the maps encoding type dependencies, i.e., the \emph{display maps} in \cite{Taylor:1987}, which are the morphisms that one would like to take limits of. 
				
				In other words, for many interesting and important examples, the weighted cones are always tight.
			\end{rk}
			
				%
			
			\begin{rk}
				The assumption that $\bbK$ is locally presentable ensures that the smallness condition required for \longref{Lemma}{lem:F} and \longref{Proposition}{pro:V_models_monadic} is fulfilled. 
				
				Nonetheless, the full power of locally presentability is not necessary. In fact, what we actually need is that the (co)domains of the class of maps in \longref{Lemma}{lem:model_orthogonal} satisfy the smallness condition.
			\end{rk}
			
			\begin{coroll}
				\label{cor:functors_as_algebras}
				Let $\mathmybb{T}$ and $\bbK$ be enhanced $2$-categories. There exists an enhanced $2$-monad $T$ on  $[\mathcal{T}_\tau, \bbK]$ such that 
				$$\FTAlg_{s, w} \cong \mathmybb{Fun}_{s, w}(\mathmybb{T}, \mathmybb{K})$$
				is an isomorphism between the enhanced $2$-categories of $T$-algebras and $\F$-functors from $\mathmybb{T}$ to $\bbK$.
				
				In particular, if $w = s$, then we have an isomorphism
				$$\TAlg_{s} \cong [\mathmybb{T}, \mathmybb{K}]_\tau$$
				of $2$-categories.
			\end{coroll}
			
			\begin{proof}
				We take $L$ to be the identity $\F$-functor in \longref{Proposition}{pro:V_models_monadic}. Following the proof of \longref{Proposition}{pro:V_models_monadic}, when $L$ is the identity and $i$ is bijective-on-objects, the restriction $i^* \colon [\mathcal{T}_\tau, \bbK] \to \mathmybb{Fun}_{s, w}(\mathmybb{S}, \mathmybb{K})$ strictly creates $i^*$-split coequalisers, which means $i^*(\overline{H}) = H$.
				
				Now, by \cite[Theorem II.2.1]{book:Dubuc:1970}, the adjunction from \longref{Proposition}{pro:V_models_monadic} is then a strictly monadic adjunction. Therefore, according to  \cite[Theorem 26]{Bourke:2014}, we have a commutative diagram
				\[
				\begin{tikzcd}[ampersand replacement=\&]
					{\mathmybb{Fun}_{s, w}(\mathmybb{T}, \mathmybb{K})} \&\& {\FTAlg_{s, w}} \\
					\& {[\mathcal{T}_\tau, \mathbbm{K}]}
					\arrow["\cong", from=1-1, to=1-3]
					\arrow["{i^*_{s, w}}"', from=1-1, to=2-2]
					\arrow["{U_{s, w}}", from=1-3, to=2-2]
				\end{tikzcd}
				\]
				of enhanced $2$-categories, where the top row is an isomorphism.
			\end{proof}

			\subsection{Limits in \texorpdfstring{$\FMod{s, w}(\mathmybb{T}, \bbK)$}{Modₛ,ᵥᵥ (𝕊, 𝕂)}}
			\label{sec:lim}
			
			In this section, we investigate what kinds of limits exist in the enhanced $2$-category $\FMod{s, w}(\mathmybb{T}, \bbK)$ of models of an enhanced limit $2$-sketch $\mathmybb{T}$ with tight cones in a locally presentable enhanced $2$-category $\bbK$. 
			
			First of all, we shall complete the proof of \longref{Theorem}{thm:equivalence}. To achieve this, we make use of the notion of dotted $2$-limits in \cite{Ko:2023}.

			Indeed, proving the statement directly with weighted limits is much more difficult than one would think. The functor $\F$-category defining $\F$-weighted limits involves morphisms that are too strict, when comparing to the enhanced $2$-category of models. In the former, the loose morphisms are $2$-natural transformations, yet, in the latter, the loose morphisms are {loose $w$-natural transformations}, instead. On the other hand, dotted $2$-limits are defined via \emph{dotted-$w$-natural transformations}, which are closely related to loose $w$-natural transformations. In other words, dotted $2$-limits provide a more straightforward and elementary approach to handle this subtle problem carefully. 
			
			We briefly recall the basic definitions.
			
			\begin{defi}[{\cite{Ko:2023}}]
				\label{def:marked_cat}
				Let $\twoA$ be a $2$-category. Let $\Sigma$ be a class of morphisms in 
				$\twoA$, which contains all the 
				identities and is closed under composition. The pair $(\twoA, \Sigma)$ 
				is called a \emph{marked $2$-category}.
			\end{defi}
			
			\begin{rk}
				Clearly, marked $2$-categories are equivalent to enhanced $2$-categories.
			\end{rk}

			\begin{defi}[{\cite[I,2, p.14, (i)]{book:Gray:1974}}]
				\label{def:sigmatrans}
				Let $(\twoA, \Sigma)$ be a marked $2$-category. Let $F, G \colon 
				\twoA \rightrightarrows 
				\twoB$ be $2$-functors. A \emph{marked-lax natural transformation} 
				$\alpha \colon F \to G$ between $F$ and $G$ is a lax natural 
				transformation $\alpha \colon F \to G$ such that for any $f \in 
				\Sigma$, the $2$-component $\alpha_f = 1$. 
			\end{defi}
			
			
			\begin{rk}
				Similarly, we can talk about \emph{marked-pseudo} or 
				\emph{marked-colax} 
				natural transformations.
			\end{rk}
			
			\begin{nota}
				We denote the $2$-category of $2$-functors $\twoA \to \twoB$, 
				marked-lax natural 
				transformations, 
				and modifications by $[\twoA, \twoB]_{l, \Sigma}$. 
			\end{nota}

			\begin{defi}[{\cite[I,7, p.23, (iii)]{book:Gray:1974}, \cite[Definition 
					2.4]{Szyld:2019}}]
				Let $(\twoA, \Sigma)$ be a small marked $2$-category, i.e., $\ob\twoA$ 
				is small, and let $\twoB$ be a $2$-category. The \emph{marked-lax 
					limit} 
				$\sigmalim{l}{F}$ of a $2$-functor $F \colon \twoA \to 
				\twoB$ is 
				characterised by an isomorphism
				\begin{equation}
					\label{def:sigmalim}
					\twoB(B, \sigmalim{l}{F}) \cong [\twoA, \twoB]_{l, 
						\Sigma}(\triangle(B), 
					F)
				\end{equation}
				in $\Cat$, which is natural in $B \in \ob\twoB$.
			\end{defi}
			
			\begin{rk}
				We can talk about marked-colax or marked-pseudo limits, by replacing 
				the marked-lax natural transformations with marked-colax or 
				marked-pseudo natural transformations, respectively.
			\end{rk}
			
			\begin{defi}[{\cite{Ko:2023}}]
				\label{def:marked_F-cat}
				Let $\bbD$ be an $\F$-category. Let $\Sigma$ be a class of morphisms in 
				$\bbD$, which contains all the 
				identities and is closed under composition. The pair $(\bbD, \Sigma)$ 
				is called a \emph{marked $\F$-category}.
			\end{defi}
			
			\begin{defi}[{\cite{Ko:2023}}]
				\label{def:dotted_F-cat}
				Let $(\bbD, \Sigma)$ be a marked $\F$-category. Let $T$ be a 
				collection 
				of objects in $\bbD$ such that if $a \in T$ and there is a 
				tight morphism $a \nrightarrow b$ in $\Sigma$, then $b \in 
				T$. The triple $(\bbD, \Sigma, \Gamma)$ 
				is called a \emph{dotted $\F$-category}.
			\end{defi}
			
			\begin{nota}
				For an object $x$ in $\Gamma$, we highlight it in diagrams with a dot above: 
				$\dot{x}$.
			\end{nota}
			
			\begin{defi}[{\cite{Ko:2023}}]
				\label{def:dottedtrans}
				Let $(\bbD, \Sigma, \Gamma)$ be a dotted $\F$-category. Let $S, R 
				\colon \bbD \rightrightarrows \bbA$ be two $\F$-functors. A 
				\emph{dotted(-marked)-lax natural transformation}  $\alpha \colon S \to 
				R$ 
				between $S$ and $R$ is a marked-lax natural transformation 
				$\alpha_\lambda \colon 
				S_\lambda \to R_\lambda$ such that and for any ${x} 
				\in \Gamma$, 
				the $1$-component $\alpha_{{x}}$ is a tight morphism in $\bbA$.
			\end{defi}

			\begin{rk}
				We also have about \emph{dotted-colax} or 
				\emph{dotted-pseudo} 
				natural transformations.
			\end{rk}
			
			\begin{nota}
				We denote the above $\F$-category of $\F$-functors $\bbD \to \bbA$, lax 
				natural 
				transformations between loose parts and dotted-lax natural 
				transformations, 
				and modifications by $[\bbD, \bbA]_{l, \Sigma, \Gamma}$.
			\end{nota}

			\begin{defi}[{\cite{Ko:2023}}]
				\label{def:dotted_lim}
				Let $(\bbD, \Sigma, T)$ be a small dotted $\F$-category, and 
				$\bbA$ be an $\F$-category. The \emph{dotted(-marked)-lax limit} 
				$\dotlim{l}{S}$ of an $\F$-functor $S 
				\colon \bbD \to 
				\bbA$ is 
				characterised by an isomorphism
				\begin{equation}
					\label{def:dotlim}
					\bbA(A, \dotlim{l}{S}) \cong [\bbD, \bbA]_{l, 
						\Sigma, \Gamma}(\triangle(A), 
					S)
				\end{equation}
				in $\bbF$, which is natural in $A \in \ob\bbA$.
			\end{defi}
			
			\begin{rk}
				By replacing dotted-lax natural transformations with dotted-colax or 
				dotted-pseudo natural transformations, we obtain the notions of 
				dotted-colax or dotted-pseudo limits, respectively.
			\end{rk}
			
			\begin{example}[$w$-limits of loose morphisms]
				\label{eg:w-lim_loose_arr}
				Let
				\[
				\bbD := \{
				\begin{tikzcd}[ampersand replacement=\&]
					{\dot{A}} \& {\dot{B}}
					\arrow["f", loose, from=1-1, to=1-2]
				\end{tikzcd}
				\}
				\]
				be the dotted $\F$-category with two dotted objects and only one non-identity (loose) morphism between them.
				
				Then, the dotted $w$-limit of an $\F$-functor $S \colon \bbD \to \bbA$ is the $w$-limit of the loose morphism $SA \leadsto SB$ in $\bbA$.
			\end{example}
			
			%
			%
				%
			
			We now present the main technical proposition in this article.
			
			\begin{pro}
				\label{pro:colax_lim_arrow}
				Let $\mathmybb{T}$ be an enhanced limit $2$-sketch with tight weighted cones, and $\bbK$ be a locally presentable enhanced $2$-category.  The enhanced $2$-category $\mathmybb{Mod}_{s, w}(\mathmybb{T}, \mathmybb{K})$ admits $\overline{w}$-limits of loose morphisms.
				%
			\end{pro}
			
			\begin{proof}
				For simplicity, we prove the statement for the case when $w = c$. 
				
				Let $S \colon \bbD \to \mathmybb{Fun}_{s, c}(\mathmybb{T}, \mathmybb{K})$ be an $\F$-functor, where $\bbD$ is the enhanced $2$-category in \longref{Example}{eg:w-lim_loose_arr}. We first show that the dotted-lax limit $\dotlim{l}{S}$ exists in $\mathmybb{Fun}_{s, c}(\mathmybb{T}, \mathmybb{K})$, and is preserved by $\mathrm{ev}_T$ at $T$, for any object $T$ of $\mathmybb{T}$.
				
				\noindent\textbf{Construction of $\boldsymbol{L}$ as the limit:}
				
				For any object $T$ of $\mathmybb{T}$, write $L_T := \dotlim{l}{\mathrm{ev}_T \cdot S}$ as the dotted-lax limit of the composite
				$$\bbD \xrightarrow{S} \mathmybb{Fun}_{s, c}(\mathmybb{T}, \mathmybb{K}) \xrightarrow{\mathrm{ev}_T} \bbK$$
				in $\bbK$. Let $\eta^T \colon \triangle(L_T) \to {\mathrm{ev}_T \cdot S}$ be the unit for this dotted limit, which is clearly a dotted-lax natural transformation. Then, the $2$-component $\eta^T_f$ at $f \colon \dot{A} \leadsto \dot{B}$ in $\bbD$ is given by
				\[
				\begin{tikzcd}[ampersand replacement=\&]
					\& {L_T} \& \\
					{(S(\dot{A}))(T)} \&\& {(S(\dot{B}))(T)}
					\arrow["{\eta^T_{\dot{A}}}"', from=1-2, to=2-1]
					\arrow[""{name=0, anchor=center, inner sep=0}, "{\eta^T_{\dot{B}}}", from=1-2, to=2-3]
					\arrow["{S(f)_T}"', loose, from=2-1, to=2-3]
					\arrow["{\eta^T_f}", between={0.4}{0.6}, Rightarrow, from=2-1, to=0]
				\end{tikzcd}
				\]
				in $\bbK$.
				
				Suppose $T_1$ and $T_2$ are objects of $\mathmybb{T}$, and $t \colon T_1 \leadsto T_2$ is a loose morphism in $\mathmybb{T}$. The $2$-component of the loose colax natural transformation $S(f)$ at $t$ is a $2$-morphism
				\[
				\begin{tikzcd}[ampersand replacement=\&]
					{(S(\dot{A}))(T_1)} \& {(S(\dot{B}))(T_1)} \\
					{(S(\dot{A}))(T_2)} \& {(S(\dot{B}))(T_2)}
					\arrow["{S(f)_{T_1}}", loose, from=1-1, to=1-2]
					\arrow["{(S(\dot{A}))(t)}"', loose, from=1-1, to=2-1]
					\arrow["{(S(\dot{B}))(t)}", loose, from=1-2, to=2-2]
					\arrow["{S(f)_t}"', between={0.3}{0.7}, Rightarrow, from=2-1, to=1-2]
					\arrow["{S(f)_{T_2}}"', loose, from=2-1, to=2-2]
				\end{tikzcd}
				\]
				in $\bbK$. So, together, we have a cone $\theta \colon \triangle(L_{T_1}) \to {\mathrm{ev}_{T_2} \cdot S}$, whose $2$-component at $f$ is given by
				\[
				\begin{tikzcd}[ampersand replacement=\&]
					\& {L_{T_1}} \& \\
					{(S(\dot{A}))(T_1)} \&\& {(S(\dot{B}))(T_1)} \\
					{(S(\dot{A}))(T_2)} \&\& {(S(\dot{B}))(T_2)}
					\arrow["{\eta^{T_1}_{\dot{A}}}"', from=1-2, to=2-1]
					\arrow[""{name=0, anchor=center, inner sep=0}, "{\eta^{T_1}_{\dot{B}}}", from=1-2, to=2-3]
					\arrow["{S(f)_{T_1}}"{description}, loose, from=2-1, to=2-3]
					\arrow["{(S(\dot{A}))(t)}"', loose, from=2-1, to=3-1]
					\arrow["{(S(\dot{B}))(t)}", loose, from=2-3, to=3-3]
					\arrow["{S(f)_t}"', between={0.4}{0.6}, Rightarrow, from=3-1, to=2-3]
					\arrow["{S(f)_{T_2}}"', loose, from=3-1, to=3-3]
					\arrow["{\eta^{T_1}_f}", between={0.4}{0.6}, Rightarrow, from=2-1, to=0]
				\end{tikzcd},
				\]
				which is clearly a marked-lax natural transformation. Note that
				\[
				\begin{tikzcd}[ampersand replacement=\&]
					\& {L_{T_2}} \& \\
					{(S(\dot{A}))(T_2)} \&\& {(S(\dot{B}))(T_2)}
					\arrow["{\eta^{T_2}_{\dot{A}}}"', from=1-2, to=2-1]
					\arrow[""{name=0, anchor=center, inner sep=0}, "{\eta^{T_2}_{\dot{B}}}", from=1-2, to=2-3]
					\arrow["{S(f)_{T_2}}"', loose, from=2-1, to=2-3]
					\arrow["{\eta^{T_2}_f}", between={0.4}{0.6}, Rightarrow, from=2-1, to=0]
				\end{tikzcd}
				\]
				is the dotted-lax limit cone for $\dotlim{l}{\mathrm{ev}_{T_2} \cdot S}$. By the universal property, there exists a unique morphism $L_t \colon L_{T_1} \leadsto L_{T_2}$ such that
				\begin{equation}
					\label{eqt:uni_L_t}
					\begin{tikzcd}[ampersand replacement=\&]
						\& {L_{T_1}} \& \\
						\& {L_{T_2}} \\
						{(S(\dot{A}))(T_2)} \&\& {(S(\dot{B}))(T_2)}
						\arrow["{\exists ! L_t}", dashed, loose, from=1-2, to=2-2]
						\arrow["{\eta^{T_2}_{\dot{A}}}"', from=2-2, to=3-1]
						\arrow[""{name=0, anchor=center, inner sep=0}, "{\eta^{T_2}_{\dot{B}}}", from=2-2, to=3-3]
						\arrow["{S(f)_{T_2}}"', loose, from=3-1, to=3-3]
						\arrow["{\eta^{T_2}_f}", between={0.4}{0.6}, Rightarrow, from=3-1, to=0]
					\end{tikzcd}
					=
					\begin{tikzcd}[ampersand replacement=\&]
						\& {L_{T_1}} \& \\
						{(S(\dot{A}))(T_1)} \&\& {(S(\dot{B}))(T_1)} \\
						{(S(\dot{A}))(T_2)} \&\& {(S(\dot{B}))(T_2)}
						\arrow["{\eta^{T_1}_{\dot{A}}}"', from=1-2, to=2-1]
						\arrow[""{name=0, anchor=center, inner sep=0}, "{\eta^{T_1}_{\dot{B}}}", from=1-2, to=2-3]
						\arrow["{S(f)_{T_1}}"{description}, loose, from=2-1, to=2-3]
						\arrow["{(S(\dot{A}))(t)}"', loose, from=2-1, to=3-1]
						\arrow["{(S(\dot{B}))(t)}", loose, from=2-3, to=3-3]
						\arrow["{S(f)_t}"', between={0.4}{0.6}, Rightarrow, from=3-1, to=2-3]
						\arrow["{S(f)_{T_2}}"', loose, from=3-1, to=3-3]
						\arrow["{\eta^{T_1}_f}", between={0.4}{0.6}, Rightarrow, from=2-1, to=0]
					\end{tikzcd}.
				\end{equation}
				If $t \colon T_1 \to T_2$ is tight, then $\theta$ becomes a dotted-lax natural transformation, and hence $L_t$ is also tight. Also, if $t$ is an identity, then $L_t$ is also an identity. Besides, if we are given a composable pair $T_1 \xleadsto{t} T_2 \xleadsto{t'} T_3$ in $\mathmybb{T}$, then we have
				\begin{align*}
					\begin{tikzcd}[ampersand replacement=\&, column sep =0.75]
						\& {L_{T_1}} \& \\
						\\
						\& {L_{T_3}} \\
						{(S(\dot{A}))(T_3)} \&\& {(S(\dot{B}))(T_3)}
						\arrow["{L_{t't}}", loose, from=1-2, to=3-2]
						\arrow["{\eta^{T_3}_{\dot{A}}}"', from=3-2, to=4-1]
						\arrow[""{name=0, anchor=center, inner sep=0}, "{\eta^{T_3}_{\dot{B}}}", from=3-2, to=4-3]
						\arrow["{S(f)_{T_3}}"', loose, from=4-1, to=4-3]
						\arrow["{\eta^{T_3}_f}", between={0.4}{0.6}, Rightarrow, from=4-1, to=0]
					\end{tikzcd}
					=
					\begin{tikzcd}[ampersand replacement=\&, column sep=0.9]
						\& {L_{T_1}} \& \\
						{(S(\dot{A}))(T_1)} \&\& {(S(\dot{B}))(T_1)} \\
						{(S(\dot{A}))(T_2)} \&\& {(S(\dot{B}))(T_2)} \\
						{(S(\dot{A}))(T_3)} \&\& {(S(\dot{B}))(T_3)}
						\arrow["{\eta^{T_1}_{\dot{A}}}"', from=1-2, to=2-1]
						\arrow[""{name=0, anchor=center, inner sep=0}, "{\eta^{T_1}_{\dot{B}}}", from=1-2, to=2-3]
						\arrow["{S(f)_{T_1}}"', loose, from=2-1, to=2-3]
						\arrow["{(S(\dot{A}))(t)}"', loose, from=2-1, to=3-1]
						\arrow["{(S(\dot{B}))(t)}", loose, from=2-3, to=3-3]
						\arrow["{S(f)_t}"', between={0.4}{0.6}, Rightarrow, from=3-1, to=2-3]
						\arrow["{S(f)_{T_2}}"', loose, from=3-1, to=3-3]
						\arrow["{(S(\dot{A}))(t')}"', loose, from=3-1, to=4-1]
						\arrow["{(S(\dot{B}))(t')}", loose, from=3-3, to=4-3]
						\arrow["{S(f)_{t'}}"', between={0.4}{0.6}, Rightarrow, from=4-1, to=3-3]
						\arrow["{S(f)_{T_3}}"', loose, from=4-1, to=4-3]
						\arrow["{\eta^{T_1}_f}", between={0.4}{0.6}, Rightarrow, from=2-1, to=0]
					\end{tikzcd}
					=
					\begin{tikzcd}[ampersand replacement=\&, column sep =0.75]
						\& {L_{T_1}} \& \\
						\& {L_{T_2}} \\
						\& {L_{T_3}} \\
						{(S(\dot{A}))(T_3)} \&\& {(S(\dot{B}))(T_3)}
						\arrow["{L_t}", loose, from=1-2, to=2-2]
						\arrow["{L_{t'}}", loose, from=2-2, to=3-2]
						\arrow["{\eta^{T_3}_{\dot{A}}}"', from=3-2, to=4-1]
						\arrow[""{name=0, anchor=center, inner sep=0}, "{\eta^{T_3}_{\dot{B}}}", from=3-2, to=4-3]
						\arrow["{S(f)_{T_3}}"', loose, from=4-1, to=4-3]
						\arrow["{\eta^{T_3}_f}", between={0.4}{0.6}, Rightarrow, from=4-1, to=0]
					\end{tikzcd}
				\end{align*}
				which gives $L_{t't} = L_{t'} \cdot L_t$. These give the functoriality of $L$ on $1$-morphisms. Next, suppose
				\[
				\begin{tikzcd}[ampersand replacement=\&]
					{T_1} \&\& {T_2}
					\arrow[""{name=0, anchor=center, inner sep=0}, "t"', curve={height=12pt}, loose, from=1-1, to=1-3]
					\arrow[""{name=1, anchor=center, inner sep=0}, "s", curve={height=-12pt}, loose, from=1-1, to=1-3]
					\arrow["\gamma", between={0.3}{0.7}, Rightarrow, from=1, to=0]
				\end{tikzcd}
				\]
				be a $2$-morphism in $\mathmybb{T}$. From above, $s$ and $t$ induce marked-lax natural transformations $\theta_s$ and $\theta_t$, respectively. We are going to construct a modification $G \colon \theta_s \to \theta_t$. Set $G_{\dot{A}} \colon {\theta_s}_{\dot{A}} \to {\theta_t}_{\dot{A}}$ to be the composite
				\[
				\begin{tikzcd}[ampersand replacement=\&]
					{L_{T_1}} \& {(S(\dot{A}))(T_1)} \&\&\& {(S(\dot{A}))(T_2)}
					\arrow["{\eta^{T_1}_{\dot{A}}}", from=1-1, to=1-2]
					\arrow[""{name=0, anchor=center, inner sep=0}, "{(S(\dot{A}))(t)}"', curve={height=12pt}, loose, from=1-2, to=1-5]
					\arrow[""{name=1, anchor=center, inner sep=0}, "{(S(\dot{A}))(s)}", curve={height=-12pt}, loose, from=1-2, to=1-5]
					\arrow["\scriptstyle{(S(\dot{A}))(\gamma)}", between={0.3}{0.7}, Rightarrow, from=1, to=0]
				\end{tikzcd}.
				\]
				By the naturality of $S(f)$, we have
				\begin{align*}
					\begin{tikzcd}[ampersand replacement=\&]
						\& {L_{T_1}} \& \\
						{(S(\dot{A}))(T_1)} \&\& {(S(\dot{B}))(T_1)} \\
						{(S(\dot{A}))(T_2)} \&\& {(S(\dot{B}))(T_2)}
						\arrow["{\eta^{T_1}_{\dot{A}}}"', from=1-2, to=2-1]
						\arrow[""{name=0, anchor=center, inner sep=0}, "{\eta^{T_1}_{\dot{B}}}", from=1-2, to=2-3]
						\arrow["{S(f)_{T_1}}"{description}, loose, from=2-1, to=2-3]
						\arrow[""{name=1, anchor=center, inner sep=0}, "{(S(\dot{A}))(t)}", curve={height=-12pt}, loose, from=2-1, to=3-1]
						\arrow[""{name=2, anchor=center, inner sep=0}, "{(S(\dot{A}))(s)}"', curve={height=12pt}, loose, from=2-1, to=3-1]
						\arrow["{(S(\dot{B}))(s)}", loose, from=2-3, to=3-3]
						\arrow["{S(f)_s}"', between={0.4}{0.6}, Rightarrow, from=3-1, to=2-3]
						\arrow["{S(f)_{T_2}}"', loose, from=3-1, to=3-3]
						\arrow["{\eta^{T_1}_f}", between={0.4}{0.6}, Rightarrow, from=2-1, to=0]
						\arrow["{{\scalebox{0.5}{$(S(\dot{A}))(\gamma)$}}}"', between={0.2}{0.8}, Rightarrow, from=2, to=1]
					\end{tikzcd}
					=
					\begin{tikzcd}[ampersand replacement=\&]
						\& {L_{T_1}} \& \\
						{(S(\dot{A}))(T_1)} \&\& {(S(\dot{B}))(T_1)} \\
						{(S(\dot{A}))(T_2)} \&\& {(S(\dot{B}))(T_2)}
						\arrow["{\eta^{T_1}_{\dot{A}}}"', from=1-2, to=2-1]
						\arrow[""{name=0, anchor=center, inner sep=0}, "{\eta^{T_1}_{\dot{B}}}", from=1-2, to=2-3]
						\arrow["{S(f)_{T_1}}"{description}, loose, from=2-1, to=2-3]
						\arrow["{(S(\dot{A}))(t)}"', loose, from=2-1, to=3-1]
						\arrow[""{name=1, anchor=center, inner sep=0}, "{(S(\dot{B}))(t)}", curve={height=-12pt}, loose, from=2-3, to=3-3]
						\arrow[""{name=2, anchor=center, inner sep=0}, "{(S(\dot{B}))(s)}"', curve={height=12pt}, loose, from=2-3, to=3-3]
						\arrow["{S(f)_t}"', between={0.4}{0.6}, Rightarrow, from=3-1, to=2-3]
						\arrow["{S(f)_{T_2}}"', loose, from=3-1, to=3-3]
						\arrow["{\eta^{T_1}_f}", between={0.4}{0.6}, Rightarrow, from=2-1, to=0]
						\arrow["{{\scalebox{0.5}{$(S(\dot{B}))(\gamma)$}}}"', between={0.2}{0.8}, Rightarrow, from=2, to=1]
					\end{tikzcd},
				\end{align*}
				so the $1$-components of $G$ satisfy the modification axiom. Now, by the $2$-dimensional universal property, there exists a unique $2$-morphism
				\[
				\begin{tikzcd}[ampersand replacement=\&]
					{L_{T_1}} \&\& {L_{T_2}}
					\arrow[""{name=0, anchor=center, inner sep=0}, "{L_t}"', curve={height=12pt}, loose, from=1-1, to=1-3]
					\arrow[""{name=1, anchor=center, inner sep=0}, "{L_s}", curve={height=-12pt}, loose, from=1-1, to=1-3]
					\arrow["{L_\gamma}", between={0.3}{0.7}, Rightarrow, from=1, to=0]
				\end{tikzcd}
				\]
				in $\bbK$ such that
				\begin{equation}
					\label{eqt:L_gamma}
					\begin{tikzcd}[ampersand replacement=\&]
						\& {L_{T_1}} \& \\
						\& {L_{T_2}} \\
						{(S(\dot{A}))(T_2)} \&\& {(S(\dot{B}))(T_2)}
						\arrow[""{name=0, anchor=center, inner sep=0}, "{L_t}", curve={height=-12pt}, loose, from=1-2, to=2-2]
						\arrow[""{name=1, anchor=center, inner sep=0}, "{L_s}"', curve={height=12pt}, loose, from=1-2, to=2-2]
						\arrow["{\eta^{T_2}_{\dot{A}}}"', from=2-2, to=3-1]
						\arrow[""{name=2, anchor=center, inner sep=0}, "{\eta^{T_2}_{\dot{B}}}", from=2-2, to=3-3]
						\arrow["{S(f)_{T_2}}"', loose, from=3-1, to=3-3]
						\arrow["{\exists!L_\gamma}"', between={0.3}{0.7}, Rightarrow, dashed, from=1, to=0]
						\arrow["{\eta^{T_2}_f}", between={0.4}{0.6}, Rightarrow, from=3-1, to=2]
					\end{tikzcd}
					=
					\begin{tikzcd}[ampersand replacement=\&]
						\& {L_{T_1}} \& \\
						{(S(\dot{A}))(T_1)} \&\& {(S(\dot{B}))(T_1)} \\
						{(S(\dot{A}))(T_2)} \&\& {(S(\dot{B}))(T_2)}
						\arrow["{\eta^{T_1}_{\dot{A}}}"', from=1-2, to=2-1]
						\arrow[""{name=0, anchor=center, inner sep=0}, "{\eta^{T_1}_{\dot{B}}}", from=1-2, to=2-3]
						\arrow["{S(f)_{T_1}}"{description}, loose, from=2-1, to=2-3]
						\arrow[""{name=1, anchor=center, inner sep=0}, "{(S(\dot{A}))(t)}", curve={height=-12pt}, loose, from=2-1, to=3-1]
						\arrow[""{name=2, anchor=center, inner sep=0}, "{(S(\dot{A}))(s)}"', curve={height=12pt}, loose, from=2-1, to=3-1]
						\arrow["{(S(\dot{B}))(s)}", loose, from=2-3, to=3-3]
						\arrow["{S(f)_s}"', between={0.4}{0.6}, Rightarrow, from=3-1, to=2-3]
						\arrow["{S(f)_{T_2}}"', loose, from=3-1, to=3-3]
						\arrow["{\eta^{T_1}_f}", between={0.4}{0.6}, Rightarrow, from=2-1, to=0]
						\arrow["{{\scalebox{0.5}{$(S(\dot{A}))(\gamma)$}}}"', between={0.2}{0.8}, Rightarrow, from=2, to=1]
					\end{tikzcd}.
				\end{equation}
				The functoriality on $2$-morphisms follows similarly. Altogether, we obtain an $\F$-functor $L \colon \mathmybb{T} \to \bbK$.
				
				We then check that $L \cong \dotlim{l}{S}$ in $\mathmybb{Fun}_{s, c}(\mathmybb{T}, \mathmybb{K})$.
				
				\noindent\textbf{$\boldsymbol{1}$-dimensional universal property of $\boldsymbol{L}$:}
				
				First, note that $\{\eta^T\}_{T \in \ob\mathmybb{T}}$ assembles to a dotted-lax natural transformation $\eta \colon \triangle{L} \to S$, which is a tight morphism in $[\bbD, \mathmybb{Fun}_{s, c}(\mathmybb{T}, \mathmybb{K})]_{s, \Sigma, \Gamma}$. In fact, for any (dotted) object $\dot{A}$ of $\bbD$, we have an $\F$-natural transformation $\eta_{\dot{A}} \colon L \to S(\dot{A})$, whose $1$-component is given by
				$$\eta^T_{\dot{A}} \colon L_T \to (S(\dot{A}))(T),$$
				and $2$-component $\eta^t_{A}$ for a morphism $t \colon T_1 \leadsto T_2$ given by
				\begin{equation}
					\label{diag:eta_X}
					\begin{tikzcd}[ampersand replacement=\&]
						{L_{T_1}} \& {(S(\dot{A}))(T_1)} \\
						{L_{T_2}} \& {(S(\dot{A}))(T_2)}
						\arrow["{\eta^{T_1}_{\dot{A}}}", from=1-1, to=1-2]
						\arrow["{L_t}"', loose, from=1-1, to=2-1]
						\arrow["{\eta^t_{\dot{A}}}"{description}, draw=none, from=1-2, to=2-1]
						\arrow["{(S(\dot{A}))(t)}", loose, from=1-2, to=2-2]
						\arrow["{\eta^{T_2}_{\dot{A}}}"', from=2-1, to=2-2]
					\end{tikzcd}.
				\end{equation}
				Note that the commutativity of this square is guaranteed by the universality of $L_t$, as in \longref{Equation}{eqt:uni_L_t}. For the unique non-identity morphism $f \colon \dot{A} \leadsto \dot{B}$ in $\bbD$, we have a modification, i.e., a $2$-morphism
				\[
				\begin{tikzcd}[ampersand replacement=\&]
					\& L \& \\
					{S(\dot{A})} \&\& {S(\dot{B})}
					\arrow["{\eta_{\dot{A}}}"', from=1-2, to=2-1]
					\arrow[""{name=0, anchor=center, inner sep=0}, "{\eta_{\dot{B}}}", from=1-2, to=2-3]
					\arrow["{S(f)}"', loose, from=2-1, to=2-3]
					\arrow["{\eta_f}"', between={0.3}{0.6}, Rightarrow, from=2-1, to=0]
				\end{tikzcd}
				\]
				in $\mathmybb{Fun}_{s, c}(\mathmybb{T}, \mathmybb{K})$, whose $1$-components, given by $\eta^T_f$ for $T \in \ob\mathmybb{T}$, clearly satisfy the modification axiom
				\begin{equation*}
					\begin{tikzcd}[ampersand replacement=\&, column sep=0.7]
						\& {L_{T_2}} \&\& {L_{T_1}} \\
						{(S(\dot{A}))(T_2)} \&\& {(S(\dot{A}))(T_1)} \\
						\& {(S(\dot{B}))(T_2)} \&\& {(S(\dot{B}))(T_1)}
						\arrow["{\eta^{T_2}_{\dot{A}}}"', from=1-2, to=2-1]
						\arrow["{\eta^t_{\dot{A}}}"{description}, draw=none, from=1-2, to=2-3]
						\arrow["{L_t}"', loose, from=1-4, to=1-2]
						\arrow["{\eta^{T_1}_{\dot{A}}}"', from=1-4, to=2-3]
						\arrow[""{name=0, anchor=center, inner sep=0}, "{\eta^{T_1}_{\dot{B}}}", from=1-4, to=3-4]
						\arrow["{S(f)_{T_2}}"', loose, from=2-1, to=3-2]
						\arrow["{S(f)_t}"', between={0.4}{0.6}, Rightarrow, from=2-1, to=3-4]
						\arrow["{(S(\dot{A}))(t)}"', loose, from=2-3, to=2-1]
						\arrow["{S(f)_{T_1}}"{description}, loose, from=2-3, to=3-4]
						\arrow["{(S(\dot{B}))(t)}", loose, from=3-4, to=3-2]
						\arrow["{\eta^{T_1}_f}", between={0.3}{0.7}, Rightarrow, from=2-3, to=0]
					\end{tikzcd}
					=
					\begin{tikzcd}[ampersand replacement=\&, column sep = 15]
						\& {L_{T_2}} \&\& {L_{T_1}} \\
						{(S(\dot{A}))(T_2)} \\
						\& {(S(\dot{B}))(T_2)} \&\& {(S(\dot{B}))(T_1)}
						\arrow["{\eta^{T_2}_{\dot{A}}}"', from=1-2, to=2-1]
						\arrow[""{name=0, anchor=center, inner sep=0}, "{\eta^{T_2}_{\dot{B}}}", from=1-2, to=3-2]
						\arrow["{\eta^t_{\dot{B}}}"{description}, draw=none, from=1-2, to=3-4]
						\arrow["{L_t}"', loose, from=1-4, to=1-2]
						\arrow["{\eta^{T_1}_{\dot{B}}}", from=1-4, to=3-4]
						\arrow["{S(f)_{T_2}}"', loose, from=2-1, to=3-2]
						\arrow["{(S(\dot{B}))(t)}", loose, from=3-4, to=3-2]
						\arrow["{\eta^{T_2}_f}", between={0.4}{0.7}, Rightarrow, from=2-1, to=0]
					\end{tikzcd}
				\end{equation*}
				by \longref{Equation}{eqt:uni_L_t}. The naturality of $\eta \colon \triangle(L) \to S$ is inherited from that of $\eta^T \colon \triangle(L_T) \to {\mathrm{ev}_T \cdot S}$. We then proceed to show that $\eta$ exhibits the universal property of a dotted-lax limit cone.
				
				Suppose there is a marked-lax natural transformation $\zeta \colon \triangle{X} \to S$, where $X$ is an $\F$-functor $\mathmybb{T} \to \bbK$. For the morphism $f \colon \dot{A} \leadsto \dot{B}$ in $\bbD$, $\zeta_f$ is given by the $2$-morphism
				\[
				\begin{tikzcd}[ampersand replacement=\&]
					\& X \& \\
					{S(\dot{A})} \&\& {S(\dot{B})}
					\arrow["{\zeta_{\dot{A}}}"', loose, from=1-2, to=2-1]
					\arrow[""{name=0, anchor=center, inner sep=0}, "{\zeta_{\dot{B}}}", loose, from=1-2, to=2-3]
					\arrow["{S(f)}"', loose, from=2-1, to=2-3]
					\arrow["{\zeta_f}"', between={0.3}{0.6}, Rightarrow, from=2-1, to=0]
				\end{tikzcd}
				\]
				in $\mathmybb{Fun}_{s, c}(\mathmybb{T}, \mathmybb{K})$. Applying the evaluation $\mathrm{ev}_T$ at $T \in \ob\mathmybb{T}$, we obtain a $2$-morphism
				\[
				\begin{tikzcd}[ampersand replacement=\&]
					\& {X(T)} \& \\
					{(S(\dot{A}))(T)} \&\& {(S(\dot{B}))(T)}
					\arrow["{\zeta^T_{\dot{A}}}"', loose, from=1-2, to=2-1]
					\arrow[""{name=0, anchor=center, inner sep=0}, "{\zeta^T_{\dot{B}}}", loose, from=1-2, to=2-3]
					\arrow["{S(f)_T}"', loose, from=2-1, to=2-3]
					\arrow["{\zeta^T_f}"', between={0.3}{0.6}, Rightarrow, from=2-1, to=0]
				\end{tikzcd}
				\]
				in $\bbK$. Now, by the universal property of $L_T = \dotlim{l}{\mathrm{ev}_T \cdot S}$, there exists a unique $1$-morphism $\chi_T \colon X(T) \leadsto L_T$ such that
				\begin{equation*}
					\begin{tikzcd}[ampersand replacement=\&]
						\& {X(T)} \& \\
						\& {L_T} \\
						{(S(\dot{A}))(T)} \&\& {(S(\dot{B}))(T)}
						\arrow["{\exists!\chi_T}", dashed, loose, from=1-2, to=2-2]
						\arrow["{\eta^T_{\dot{A}}}"', from=2-2, to=3-1]
						\arrow[""{name=0, anchor=center, inner sep=0}, "{\eta^T_{\dot{B}}}", from=2-2, to=3-3]
						\arrow["{S(f)_T}"', loose, from=3-1, to=3-3]
						\arrow["{\eta^T_f}", between={0.4}{0.6}, Rightarrow, from=3-1, to=0]
					\end{tikzcd}
					=
					\begin{tikzcd}[ampersand replacement=\&]
						\& {X(T)} \& \\
						\\
						{(S(\dot{A}))(T)} \&\& {(S(\dot{B}))(T)}
						\arrow["{\zeta^T_{\dot{A}}}"', loose, from=1-2, to=3-1]
						\arrow[""{name=0, anchor=center, inner sep=0}, "{\zeta^T_{\dot{B}}}", loose, from=1-2, to=3-3]
						\arrow["{S(f)_T}"', loose, from=3-1, to=3-3]
						\arrow["{\zeta^T_f}"', between={0.3}{0.6}, Rightarrow, from=3-1, to=0]
					\end{tikzcd}.
				\end{equation*}
				Let $t \colon T_1 \leadsto T_2$ be a morphism in $\mathmybb{T}$. Consider the composite, denoted as $\mu_f$,
				\[
				\begin{tikzcd}[ampersand replacement=\&]
					\& {X(T_1)} \& \\
					\& {X(T_2)} \\
					{(S(\dot{A}))(T_2)} \&\& {(S(\dot{B}))(T_2)}
					\arrow["{X(t)}", loose, from=1-2, to=2-2]
					\arrow["{\zeta^{T_2}_{\dot{A}}}"', loose, from=2-2, to=3-1]
					\arrow[""{name=0, anchor=center, inner sep=0}, "{\zeta^{T_2}_{\dot{B}}}", loose, from=2-2, to=3-3]
					\arrow["{S(f)_{T_2}}"', loose, from=3-1, to=3-3]
					\arrow["{\zeta^{T_2}_f}"', between={0.3}{0.6}, Rightarrow, from=3-1, to=0]
				\end{tikzcd},
				\]
				which clearly induces a marked-lax natural transformation $\mu \colon \triangle(X(T_1)) \to \mathrm{ev}_{T_2} \cdot S$, whose $1$-component at $\dot{A}$ is given by $\mu_A := \zeta^{T_2}_{\dot{A}} \cdot X(t)$, and $2$-component is given by $\mu_f$. We then have 
				\begin{equation*}
					\begin{tikzcd}[ampersand replacement=\&]
						\& {X(T_1)} \& \\
						\& {X(T_2)} \\
						\\
						{(S(\dot{A}))(T_2)} \&\& {(S(\dot{B}))(T_2)}
						\arrow["{X(t)}", loose, from=1-2, to=2-2]
						\arrow["{\zeta^{T_2}_{\dot{A}}}"', loose, from=2-2, to=4-1]
						\arrow[""{name=0, anchor=center, inner sep=0}, "{\zeta^{T_2}_{\dot{B}}}", loose, from=2-2, to=4-3]
						\arrow["{S(f)_{T_2}}"', loose, from=4-1, to=4-3]
						\arrow["{\zeta^{T_2}_f}"', between={0.3}{0.6}, Rightarrow, from=4-1, to=0]
					\end{tikzcd}
					=
					\begin{tikzcd}[ampersand replacement=\&]
						\& {X(T_1)} \& \\
						\& {X(T_2)} \\
						\& {L_{T_2}} \\
						{(S(\dot{A}))(T_2)} \&\& {(S(\dot{B}))(T_2)}
						\arrow["{X(t)}", loose, from=1-2, to=2-2]
						\arrow["{\chi_{T_2}}", loose, from=2-2, to=3-2]
						\arrow["{\eta^{T_2}_{\dot{A}}}"', from=3-2, to=4-1]
						\arrow[""{name=0, anchor=center, inner sep=0}, "{\eta^{T_2}_{\dot{B}}}", from=3-2, to=4-3]
						\arrow["{S(f)_{T_2}}"', loose, from=4-1, to=4-3]
						\arrow["{\eta^{T_2}_f}", between={0.4}{0.6}, Rightarrow, from=4-1, to=0]
					\end{tikzcd}.
				\end{equation*}
				Next, consider the composite, denoted as $\lambda_f$,
				\[
				\begin{tikzcd}[ampersand replacement=\&]
					\& {X(T_1)} \& \\
					{(S(\dot{A}))(T_1)} \&\& {(S(\dot{B}))(T_1)} \\
					{(S(\dot{A}))(T_2)} \&\& {(S(\dot{B}))(T_2)}
					\arrow["{\zeta^{T_1}_{\dot{A}}}"', loose, from=1-2, to=2-1]
					\arrow[""{name=0, anchor=center, inner sep=0}, "{\zeta^{T_1}_{\dot{B}}}", loose, from=1-2, to=2-3]
					\arrow["{S(f)_{T_1}}"{description}, loose, from=2-1, to=2-3]
					\arrow["{(S(\dot{A}))(t)}"', loose, from=2-1, to=3-1]
					\arrow["{(S(\dot{B}))(t)}", loose, from=2-3, to=3-3]
					\arrow["{S(f)_t}"', between={0.4}{0.6}, Rightarrow, from=3-1, to=2-3]
					\arrow["{S(f)_{T_2}}"', loose, from=3-1, to=3-3]
					\arrow["{\zeta^{T_1}_f}", between={0.4}{0.6}, Rightarrow, from=2-1, to=0]
				\end{tikzcd}
				\]
				for $f \in \bbD$, which clearly induces a marked-lax natural transformation $\lambda \colon \triangle(X(T_1)) \to \mathrm{ev}_{T_2} \cdot S$, whose $1$-component at $\dot{A}$ is given by $\lambda_A \colon (S(\dot{A}))(t) \cdot \zeta^{T_1}_{\dot{A}}$, and $2$-component is given by $\lambda_f$. We have
				\begin{equation}
					\label{eqt:lambda}
					\begin{aligned}
						\begin{tikzcd}[ampersand replacement=\&, column sep=10]
							\& {X(T_1)} \& \\
							\\
							{(S(\dot{A}))(T_1)} \&\& {(S(\dot{B}))(T_1)} \\
							{(S(\dot{A}))(T_2)} \&\& {(S(\dot{B}))(T_2)}
							\arrow["{\zeta^{T_1}_{\dot{A}}}"', loose, from=1-2, to=3-1]
							\arrow[""{name=0, anchor=center, inner sep=0}, "{\zeta^{T_1}_{\dot{B}}}", loose, from=1-2, to=3-3]
							\arrow["{S(f)_{T_1}}"{description}, loose, from=3-1, to=3-3]
							\arrow["{(S(\dot{A}))(t)}"', loose, from=3-1, to=4-1]
							\arrow["{(S(\dot{B}))(t)}", loose, from=3-3, to=4-3]
							\arrow["{S(f)_t}"', between={0.4}{0.6}, Rightarrow, from=4-1, to=3-3]
							\arrow["{S(f)_{T_2}}"', loose, from=4-1, to=4-3]
							\arrow["{\zeta^{T_1}_f}", between={0.4}{0.6}, Rightarrow, from=3-1, to=0]
						\end{tikzcd}
						&=
						\begin{tikzcd}[ampersand replacement=\&, column sep=10]
							\& {X(T_1)} \& \\
							\& {L_{T_1}} \\
							{(S(\dot{A}))(T_1)} \&\& {(S(\dot{B}))(T_1)} \\
							{(S(\dot{A}))(T_2)} \&\& {(S(\dot{B}))(T_2)}
							\arrow["{\chi_{T_1}}", loose, from=1-2, to=2-2]
							\arrow["{\eta^{T_1}_{\dot{A}}}"', from=2-2, to=3-1]
							\arrow[""{name=0, anchor=center, inner sep=0}, "{\eta^{T_1}_{\dot{B}}}", from=2-2, to=3-3]
							\arrow["{S(f)_{T_1}}"{description}, loose, from=3-1, to=3-3]
							\arrow["{(S(\dot{A}))(t)}"', loose, from=3-1, to=4-1]
							\arrow["{(S(\dot{B}))(t)}", loose, from=3-3, to=4-3]
							\arrow["{S(f)_t}"', between={0.4}{0.6}, Rightarrow, from=4-1, to=3-3]
							\arrow["{S(f)_{T_2}}"', loose, from=4-1, to=4-3]
							\arrow["{\eta^{T_1}_f}", between={0.4}{0.6}, Rightarrow, from=3-1, to=0]
						\end{tikzcd}
						\\
						&=
						\begin{tikzcd}[ampersand replacement=\&]
							\& {X(T_1)} \& \\
							\& {L_{T_1}} \\
							\& {L_{T_2}} \\
							{(S(\dot{A}))(T_2)} \&\& {(S(\dot{B}))(T_2)}
							\arrow["{\chi_{T_1}}", loose, from=1-2, to=2-2]
							\arrow["{L_t}", loose, from=2-2, to=3-2]
							\arrow["{\eta^{T_2}_{\dot{A}}}"', from=3-2, to=4-1]
							\arrow[""{name=0, anchor=center, inner sep=0}, "{\eta^{T_2}_{\dot{B}}}", from=3-2, to=4-3]
							\arrow["{S(f)_{T_2}}"', loose, from=4-1, to=4-3]
							\arrow["{\eta^{T_2}_f}", between={0.4}{0.6}, Rightarrow, from=4-1, to=0]
						\end{tikzcd}.
					\end{aligned}
				\end{equation}
				We are going to construct a modification $\lambda \to \mu$. Indeed, as $\zeta_{\dot{A}}$ is a loose colax natural transformation, we have the $2$-component
				\[
				\begin{tikzcd}[ampersand replacement=\&]
					{X(T_1)} \& {(S(\dot{A}))(T_1)} \\
					{X(T_2)} \& {(S(\dot{A}))(T_2)}
					\arrow["{\zeta^{T_1}_{\dot{A}}}", loose, from=1-1, to=1-2]
					\arrow["{X(t)}"', loose, from=1-1, to=2-1]
					\arrow["{(S(\dot{A}))(t)}", loose, from=1-2, to=2-2]
					\arrow["{\zeta^{t}_{\dot{A}}}"', between={0.3}{0.7}, Rightarrow, from=2-1, to=1-2]
					\arrow["{\zeta^{T_2}_{\dot{A}}}"', loose, from=2-1, to=2-2]
				\end{tikzcd}
				\]
				of $\zeta_{\dot{A}}$ at $t$ as a $2$-morphism in $\bbK$, and since $\zeta_f$ is a modification by definition, 
				\begin{equation}
					\label{eqt:zeta_f}
					\begin{tikzcd}[ampersand replacement=\&, column sep=small]
						\& {X(T_2)} \& {X(T_1)} \\
						{(S(\dot{A}))(T_2)} \\
						\& {(S(\dot{B}))(T_2)} \& {(S(\dot{B}))(T_1)}
						\arrow["{\zeta^{T_2}_{\dot{A}}}"', loose, from=1-2, to=2-1]
						\arrow[""{name=0, anchor=center, inner sep=0}, "{\zeta^{T_2}_{\dot{B}}}"{description}, loose, from=1-2, to=3-2]
						\arrow["{\zeta^t_{\dot{B}}}", between={0.4}{0.6}, Rightarrow, from=1-2, to=3-3]
						\arrow["{X(t)}"', loose, from=1-3, to=1-2]
						\arrow["{\zeta^{T_1}_{\dot{B}}}", loose, from=1-3, to=3-3]
						\arrow["{S(f)_{T_2}}"', loose, from=2-1, to=3-2]
						\arrow["{(S(\dot{B}))(t)}", loose, from=3-3, to=3-2]
						\arrow["{\zeta^{T_2}_f}"', between={0.3}{0.7}, Rightarrow, from=2-1, to=0]
					\end{tikzcd}
					=
					\begin{tikzcd}[ampersand replacement=\&]
						\& {X(T_2)} \& {X(T_1)} \\
						{(S(\dot{A}))(T_2)} \& {(S(\dot{A}))(T_1)} \\
						\& {(S(\dot{B}))(T_2)} \& {(S(\dot{B}))(T_1)}
						\arrow["{\zeta^{T_2}_{\dot{A}}}"', loose, from=1-2, to=2-1]
						\arrow["{\zeta^{t}_{\dot{A}}}"', between={0.2}{0.8}, Rightarrow, from=1-2, to=2-2]
						\arrow["{X(t)}"', loose, from=1-3, to=1-2]
						\arrow["{\zeta^{T_1}_{\dot{A}}}"{description}, loose, from=1-3, to=2-2]
						\arrow[""{name=0, anchor=center, inner sep=0}, "{\zeta^{T_1}_{\dot{B}}}", loose, from=1-3, to=3-3]
						\arrow["{S(f)_{T_2}}"', loose, from=2-1, to=3-2]
						\arrow["{S(f)_t}"', between={0.4}{0.6}, Rightarrow, from=2-1, to=3-3]
						\arrow["{(S(\dot{A}))(t)}", loose, from=2-2, to=2-1]
						\arrow["{S(f)_{T_1}}"{description}, loose, from=2-2, to=3-3]
						\arrow["{(S(\dot{B}))(t)}", loose, from=3-3, to=3-2]
						\arrow["{\zeta^{T_1}_f}"', between={0.3}{0.7}, Rightarrow, from=2-2, to=0]
					\end{tikzcd},
				\end{equation}
				which also gives the modification axiom for $\{\zeta^t_{\dot{A}}\}_{{\dot{A}} \in \ob\bbD}$, and hence, $\{\zeta^t_{\dot{A}}\}_{{\dot{A}} \in \ob\bbD}$ assembles to a modification $\zeta^t \colon \lambda \to \mu$. Now, by the $2$-dimensional universal property of $L_{T_2}$, there is a unique $2$-morphism
				\[
				\begin{tikzcd}[ampersand replacement=\&]
					{X(T_1)} \& {L_{T_1}} \\
					{X(T_2)} \& {L_{T_2}}
					\arrow["{\chi_{T_1}}", loose, from=1-1, to=1-2]
					\arrow["{X(t)}"', loose, from=1-1, to=2-1]
					\arrow["{L_t}", loose, from=1-2, to=2-2]
					\arrow["{\chi_t}"', between={0.3}{0.7}, Rightarrow, from=2-1, to=1-2]
					\arrow["{\chi_{T_2}}"', loose, from=2-1, to=2-2]
				\end{tikzcd}
				\]
				in $\bbK$ satisfying
				\begin{equation*}
					\begin{aligned}
						\begin{tikzcd}[ampersand replacement=\&]
							\& {X(T_1)} \& \\
							{X(T_2)} \&\& {L_{T_1}} \\
							\& {L_{T_2}} \\
							{(S(\dot{A}))(T_2)} \&\& {(S(\dot{B}))(T_2)}
							\arrow["{X(t)}"', loose, from=1-2, to=2-1]
							\arrow["{\chi_{T_1}}", loose, from=1-2, to=2-3]
							\arrow["{\exists!\chi_t}"', between={0.4}{0.6}, Rightarrow, from=2-1, to=2-3]
							\arrow["{\chi_{T_2}}"', loose, from=2-1, to=3-2]
							\arrow["{L_t}", loose, from=2-3, to=3-2]
							\arrow["{\eta^{T_2}_{\dot{A}}}"', from=3-2, to=4-1]
							\arrow[""{name=0, anchor=center, inner sep=0}, "{\eta^{T_2}_{\dot{B}}}", from=3-2, to=4-3]
							\arrow["{S(f)_{T_2}}"', loose, from=4-1, to=4-3]
							\arrow["{\eta^{T_2}_f}"', between={0.3}{0.7}, Rightarrow, from=4-1, to=0]
						\end{tikzcd}
						&=
						\begin{tikzcd}[ampersand replacement=\&, column sep=small]
							\& {X(T_2)} \& {X(T_1)} \\
							{(S(\dot{A}))(T_2)} \\
							\& {(S(\dot{B}))(T_2)} \& {(S(\dot{B}))(T_1)}
							\arrow["{\zeta^{T_2}_{\dot{A}}}"', loose, from=1-2, to=2-1]
							\arrow[""{name=0, anchor=center, inner sep=0}, "{\zeta^{T_2}_{\dot{B}}}"{description}, loose, from=1-2, to=3-2]
							\arrow["{\zeta^t_{\dot{B}}}", between={0.4}{0.6}, Rightarrow, from=1-2, to=3-3]
							\arrow["{X(t)}"', loose, from=1-3, to=1-2]
							\arrow["{\zeta^{T_1}_{\dot{B}}}", loose, from=1-3, to=3-3]
							\arrow["{S(f)_{T_2}}"', loose, from=2-1, to=3-2]
							\arrow["{(S(\dot{B}))(t)}", loose, from=3-3, to=3-2]
							\arrow["{\zeta^{T_2}_f}"', between={0.3}{0.7}, Rightarrow, from=2-1, to=0]
						\end{tikzcd}.
					\end{aligned}
				\end{equation*}
				In particular, we have
				\begin{equation}
					\label{eqt:chi_t_leg}
					\eta^{T_2}_{\dot{A}} \cdot \chi_t = \zeta^{t}_{\dot{A}},
				\end{equation}
				\begin{equation}
					\label{eqt:chi_t_leg_B}
					\eta^{T_2}_{\dot{B}} \cdot \chi_t = \zeta^t_{\dot{B}}.
				\end{equation}
				Now, we shall verify the naturality of $\chi_t$, i.e., for any $2$-morphism
				\[
				\begin{tikzcd}[ampersand replacement=\&]
					{L_{T_1}} \&\& {L_{T_2}}
					\arrow[""{name=0, anchor=center, inner sep=0}, "{L_t}"', curve={height=12pt}, loose, from=1-1, to=1-3]
					\arrow[""{name=1, anchor=center, inner sep=0}, "{L_s}", curve={height=-12pt}, loose, from=1-1, to=1-3]
					\arrow["{L_\gamma}", between={0.3}{0.7}, Rightarrow, from=1, to=0]
				\end{tikzcd}
				\]
				in $\mathmybb{T}$,
				\begin{equation}
					\label{eqt:chi_t_natural}
					\begin{tikzcd}[ampersand replacement=\&]
						{X(T_1)} \&\& {L_{T_1}} \\
						{X(T_2)} \&\& {L_{T_2}}
						\arrow["{\chi_{T_1}}", loose, from=1-1, to=1-3]
						\arrow["{X(s)}"', loose, from=1-1, to=2-1]
						\arrow[""{name=0, anchor=center, inner sep=0}, "{L_t}", curve={height=-12pt}, loose, from=1-3, to=2-3]
						\arrow[""{name=1, anchor=center, inner sep=0}, "{L_s}"', curve={height=12pt}, loose, from=1-3, to=2-3]
						\arrow["{\chi_s}"', between={0.4}{0.6}, Rightarrow, from=2-1, to=1-3]
						\arrow["{\chi_{T_2}}"', loose, from=2-1, to=2-3]
						\arrow["{L_\gamma}"', between={0.2}{0.8}, Rightarrow, from=1, to=0]
					\end{tikzcd}
					=
					\begin{tikzcd}[ampersand replacement=\&]
						{X(T_1)} \&\& {L_{T_1}} \\
						{X(T_2)} \&\& {L_{T_2}}
						\arrow["{\chi_{T_1}}", loose, from=1-1, to=1-3]
						\arrow[""{name=0, anchor=center, inner sep=0}, "{X(s)}"', curve={height=12pt}, loose, from=1-1, to=2-1]
						\arrow[""{name=1, anchor=center, inner sep=0}, "{X(t)}", curve={height=-12pt}, loose, from=1-1, to=2-1]
						\arrow["{L_t}", loose, from=1-3, to=2-3]
						\arrow["{\chi_t}"', between={0.4}{0.6}, Rightarrow, from=2-1, to=1-3]
						\arrow["{\chi_{T_2}}"', loose, from=2-1, to=2-3]
						\arrow["{X(\gamma)}"', between={0.2}{0.8}, Rightarrow, from=0, to=1]
					\end{tikzcd}.
				\end{equation}
				In fact, we have
				{
					\allowdisplaybreaks
					\begin{align*}
						\begin{tikzcd}[ampersand replacement=\&, sep=large]
							\& {X(T_1)} \& \\
							{X(T_2)} \&\& {L_{T_1}} \\
							\& {L_{T_2}} \\
							{(S(\dot{A}))(T_2)} \&\& {(S(\dot{B}))(T_2)}
							\arrow[""{name=0, anchor=center, inner sep=0}, "{X(s)}"', curve={height=12pt}, loose, from=1-2, to=2-1]
							\arrow[""{name=1, anchor=center, inner sep=0}, "{X(t)}", curve={height=-12pt}, loose, from=1-2, to=2-1]
							\arrow["{\chi_{T_1}}", loose, from=1-2, to=2-3]
							\arrow["{\chi_t}"', between={0.4}{0.6}, Rightarrow, from=2-1, to=2-3]
							\arrow["{\chi_{T_2}}"', loose, from=2-1, to=3-2]
							\arrow["{L_t}", loose, from=2-3, to=3-2]
							\arrow["{\eta^{T_2}_{\dot{A}}}"', from=3-2, to=4-1]
							\arrow[""{name=2, anchor=center, inner sep=0}, "{\eta^{T_2}_{\dot{B}}}", from=3-2, to=4-3]
							\arrow["{S(f)_{T_2}}"', loose, from=4-1, to=4-3]
							\arrow["{X(\gamma)}"', between={0.2}{0.8}, Rightarrow, from=0, to=1]
							\arrow["{\eta^{T_2}_f}"', between={0.3}{0.7}, Rightarrow, from=4-1, to=2]
						\end{tikzcd}
						=
						\begin{tikzcd}[ampersand replacement=\&,sep=large]
							\& {X(T_2)} \& {X(T_1)} \\
							{(S(\dot{A}))(T_2)} \\
							\& {(S(\dot{B}))(T_2)} \& {(S(\dot{B}))(T_1)}
							\arrow["{\zeta^{T_2}_{\dot{A}}}"', loose, from=1-2, to=2-1]
							\arrow[""{name=0, anchor=center, inner sep=0}, "{\zeta^{T_2}_{\dot{B}}}"{description}, loose, from=1-2, to=3-2]
							\arrow["{\zeta^t_{\dot{B}}}"', between={0.4}{0.6}, Rightarrow, from=1-2, to=3-3]
							\arrow[""{name=1, anchor=center, inner sep=0}, "{X(s)}"', curve={height=12pt}, loose, from=1-3, to=1-2]
							\arrow[""{name=2, anchor=center, inner sep=0}, "{X(t)}", curve={height=-12pt}, loose, from=1-3, to=1-2]
							\arrow["{\zeta^{T_1}_{\dot{B}}}", loose, from=1-3, to=3-3]
							\arrow["{S(f)_{T_2}}"', loose, from=2-1, to=3-2]
							\arrow["{(S(\dot{B}))(t)}", loose, from=3-3, to=3-2]
							\arrow["{X(\gamma)}"', between={0.2}{0.8}, Rightarrow, from=1, to=2]
							\arrow["{\zeta^{T_2}_f}"', between={0.3}{0.7}, Rightarrow, from=2-1, to=0]
						\end{tikzcd}
						&
						\\
						\stackunder{\text{naturality of $\zeta_{\dot{B}}$}}{$=$}
						\begin{tikzcd}[ampersand replacement=\&,sep=large]
							\& {X(T_2)} \& {X(T_1)} \\
							{(S(\dot{A}))(T_2)} \\
							\& {(S(\dot{B}))(T_2)} \& {(S(\dot{B}))(T_1)}
							\arrow["{\zeta^{T_2}_{\dot{A}}}"', loose, from=1-2, to=2-1]
							\arrow[""{name=0, anchor=center, inner sep=0}, "{\zeta^{T_2}_{\dot{B}}}"{description}, loose, from=1-2, to=3-2]
							\arrow["{\zeta^s_{\dot{B}}}", between={0.4}{0.6}, Rightarrow, from=1-2, to=3-3]
							\arrow["{X(s)}"', loose, from=1-3, to=1-2]
							\arrow["{\zeta^{T_1}_{\dot{B}}}", loose, from=1-3, to=3-3]
							\arrow["{S(f)_{T_2}}"', loose, from=2-1, to=3-2]
							\arrow[""{name=1, anchor=center, inner sep=0}, "{(S(\dot{B}))(t)}", curve={height=-12pt}, loose, from=3-3, to=3-2]
							\arrow[""{name=2, anchor=center, inner sep=0}, "{(S(\dot{B}))(s)}"', curve={height=12pt}, loose, from=3-3, to=3-2]
							\arrow["{\zeta^{T_2}_f}"', between={0.3}{0.7}, Rightarrow, from=2-1, to=0]
							\arrow["{{\scalebox{0.8}{$(S(\dot{B}))(\gamma)$}}}"{description}, between={0.2}{0.8}, Rightarrow, from=2, to=1]
						\end{tikzcd}
						&
						\\
						\stackunder{\text{$\zeta_f$ is a  modification}}{$=$}
						\begin{tikzcd}[ampersand replacement=\&,sep=large]
							\& {X(T_2)} \& {X(T_1)} \\
							{(S(\dot{A}))(T_2)} \& {(S(\dot{A}))(T_1)} \\
							\& {(S(\dot{B}))(T_2)} \& {(S(\dot{B}))(T_1)}
							\arrow["{\zeta^{T_2}_{\dot{A}}}"', loose, from=1-2, to=2-1]
							\arrow["{\zeta^{s}_{\dot{A}}}"', between={0.2}{0.8}, Rightarrow, from=1-2, to=2-2]
							\arrow["{X(s)}"', loose, from=1-3, to=1-2]
							\arrow["{\zeta^{T_1}_{\dot{A}}}"{description}, loose, from=1-3, to=2-2]
							\arrow[""{name=0, anchor=center, inner sep=0}, "{\zeta^{T_1}_{\dot{B}}}", loose, from=1-3, to=3-3]
							\arrow["{S(f)_{T_2}}"', loose, from=2-1, to=3-2]
							\arrow["{(S(\dot{A}))(s)}", loose, from=2-2, to=2-1]
							\arrow["{\scalebox{0.8}{$S(f)_{T_1}$}}", loose, from=2-2, to=3-3]
							\arrow[""{name=1, anchor=center, inner sep=0}, "{\scalebox{0.7}{$(S(\dot{B}))(s)$}}"{description}, curve={height=12pt}, loose, from=3-3, to=3-2]
							\arrow[""{name=2, anchor=center, inner sep=0}, "{(S(\dot{B}))(t)}", curve={height=-12pt}, loose, from=3-3, to=3-2]
							\arrow["{S(f)_s}", between={0.4}{0.6}, Rightarrow, from=2-1, to=1]
							\arrow["{\zeta^{T_1}_f}", between={0.3}{0.7}, Rightarrow, from=2-2, to=0]
							\arrow["{{\scalebox{0.6}{$(S(\dot{B}))(\gamma)$}}}"{description}, between={0.2}{0.8}, Rightarrow, from=1, to=2]
						\end{tikzcd}
						&
						\\
						\stackunder{\text{naturality of $S(f)$}}{$=$}
						\begin{tikzcd}[ampersand replacement=\&,sep=large]
							\& {X(T_2)} \& {X(T_1)} \\
							{(S(\dot{A}))(T_2)} \& {(S(\dot{A}))(T_1)} \\
							\& {(S(\dot{B}))(T_2)} \& {(S(\dot{B}))(T_1)}
							\arrow["{\zeta^{T_2}_{\dot{A}}}"', loose, from=1-2, to=2-1]
							\arrow["{\zeta^{s}_{\dot{A}}}"', between={0.2}{0.8}, Rightarrow, from=1-2, to=2-2]
							\arrow["{X(s)}"', loose, from=1-3, to=1-2]
							\arrow["{\zeta^{T_1}_{\dot{A}}}"{description}, loose, from=1-3, to=2-2]
							\arrow[""{name=0, anchor=center, inner sep=0}, "{\zeta^{T_1}_{\dot{B}}}", loose, from=1-3, to=3-3]
							\arrow["{S(f)_{T_2}}"', loose, from=2-1, to=3-2]
							\arrow["{S(f)_t}"', between={0.4}{0.6}, Rightarrow, from=2-1, to=3-3]
							\arrow[""{name=1, anchor=center, inner sep=0}, "{\scalebox{0.8}{$(S(\dot{A}))(s)$}}"{description}, curve={height=12pt}, loose, from=2-2, to=2-1]
							\arrow[""{name=2, anchor=center, inner sep=0}, "{\scalebox{0.8}{$(S(\dot{A}))(t)$}}"{description}, curve={height=-12pt}, loose, from=2-2, to=2-1]
							\arrow["{\scalebox{0.8}{$S(f)_{T_1}$}}", loose, from=2-2, to=3-3]
							\arrow["{(S(\dot{B}))(t)}", loose, from=3-3, to=3-2]
							\arrow["{\scalebox{0.7}{$(S(\dot{A}))(\gamma)$}}"{description}, between={0.2}{0.8}, Rightarrow, from=1, to=2]
							\arrow["{\zeta^{T_1}_f}", between={0.3}{0.7}, Rightarrow, from=2-2, to=0]
						\end{tikzcd}
						&
						\\
						\stackunder{\longref{}{eqt:lambda}}{$=$}
						\begin{tikzcd}[ampersand replacement=\&,sep=large]
							\& {X(T_2)} \& {X(T_1)} \\
							{(S(\dot{A}))(T_2)} \& {(S(\dot{A}))(T_1)} \& {L_{T_1}} \\
							\& {(S(\dot{B}))(T_2)} \& {(S(\dot{B}))(T_1)}
							\arrow["{\zeta^{T_2}_{\dot{A}}}"', loose, from=1-2, to=2-1]
							\arrow["{\zeta^{s}_{\dot{A}}}"', between={0.2}{0.8}, Rightarrow, from=1-2, to=2-2]
							\arrow["{X(s)}"', loose, from=1-3, to=1-2]
							\arrow["{\chi_{T_1}}", loose, from=1-3, to=2-3]
							\arrow["{S(f)_{T_2}}"', loose, from=2-1, to=3-2]
							\arrow["{S(f)_t}"', between={0.4}{0.6}, Rightarrow, from=2-1, to=3-3]
							\arrow[""{name=0, anchor=center, inner sep=0}, "{\scalebox{0.8}{$(S(\dot{A}))(s)$}}"{description}, curve={height=12pt}, loose, from=2-2, to=2-1]
							\arrow[""{name=1, anchor=center, inner sep=0}, "{\scalebox{0.8}{$(S(\dot{A}))(t)$}}"{description}, curve={height=-12pt}, loose, from=2-2, to=2-1]
							\arrow["{\scalebox{0.8}{$S(f)_{T_1}$}}"{description}, loose, from=2-2, to=3-3]
							\arrow["{\eta^{T_1}_{\dot{A}}}"', from=2-3, to=2-2]
							\arrow[""{name=2, anchor=center, inner sep=0}, "{\eta^{T_1}_{\dot{B}}}", from=2-3, to=3-3]
							\arrow["{(S(\dot{B}))(t)}", loose, from=3-3, to=3-2]
							\arrow["{\scalebox{0.7}{$(S(\dot{A}))(\gamma)$}}"{description}, between={0.2}{0.8}, Rightarrow, from=0, to=1]
							\arrow["{\eta^{T_1}_f}", between={0.3}{0.7}, Rightarrow, from=2-2, to=2]
						\end{tikzcd}
						&
						\\
						\stackunder{\longref{}{eqt:L_gamma}\&\longref{}{diag:eta_X}}{$=$}
						\begin{tikzcd}[ampersand replacement=\&,sep=large]
							\& {X(T_2)} \& {X(T_1)} \\
							{(S(\dot{A}))(T_2)} \& {L_{T_2}} \& {L_{T_1}} \\
							\& {(S(\dot{B}))(T_2)} \& {(S(\dot{B}))(T_1)}
							\arrow["{\zeta^{T_2}_{\dot{A}}}"', loose, from=1-2, to=2-1]
							\arrow["{\zeta^{s}_{\dot{A}}}"', between={0.2}{0.8}, Rightarrow, from=1-2, to=2-2]
							\arrow["{X(s)}"', loose, from=1-3, to=1-2]
							\arrow["{\chi_{T_1}}", loose, from=1-3, to=2-3]
							\arrow["{S(f)_{T_2}}"', loose, from=2-1, to=3-2]
							\arrow["{\eta^{T_2}_{\dot{A}}}"', from=2-2, to=2-1]
							\arrow[""{name=0, anchor=center, inner sep=0}, "{\eta^{T_2}_{\dot{B}}}", from=2-2, to=3-2]
							\arrow["{\eta^t_{\dot{A}}}"', draw=none, from=2-2, to=3-3]
							\arrow[""{name=1, anchor=center, inner sep=0}, "{L_s}"', curve={height=12pt}, loose, from=2-3, to=2-2]
							\arrow[""{name=2, anchor=center, inner sep=0}, "{L_t}"{description}, curve={height=-12pt}, loose, from=2-3, to=2-2]
							\arrow["{\eta^{T_1}_{\dot{B}}}", from=2-3, to=3-3]
							\arrow["{(S(\dot{B}))(t)}", loose, from=3-3, to=3-2]
							\arrow["{\eta^{T_2}_f}", between={0.3}{0.7}, Rightarrow, from=2-1, to=0]
							\arrow["{L_\gamma}", between={0.2}{0.8}, Rightarrow, from=1, to=2]
						\end{tikzcd}
						&
						\\
						\stackunder{\longref{}{eqt:chi_t_leg}}{$=$}
						\begin{tikzcd}[ampersand replacement=\&,sep=large]
							\& {X(T_1)} \& \\
							{X(T_2)} \&\& {L_{T_1}} \\
							\& {L_{T_2}} \\
							{(S(\dot{A}))(T_2)} \&\& {(S(\dot{B}))(T_2)}
							\arrow["{X(s)}"', loose, from=1-2, to=2-1]
							\arrow["{\chi_{T_1}}", loose, from=1-2, to=2-3]
							\arrow["{\chi_s}"', between={0.4}{0.6}, Rightarrow, from=2-1, to=2-3]
							\arrow["{\chi_{T_2}}"', loose, from=2-1, to=3-2]
							\arrow[""{name=0, anchor=center, inner sep=0}, "{L_s}"{description}, curve={height=12pt}, loose, from=2-3, to=3-2]
							\arrow[""{name=1, anchor=center, inner sep=0}, "{L_t}", curve={height=-12pt}, loose, from=2-3, to=3-2]
							\arrow["{\eta^{T_2}_{\dot{A}}}"', from=3-2, to=4-1]
							\arrow[""{name=2, anchor=center, inner sep=0}, "{\eta^{T_2}_{\dot{B}}}", from=3-2, to=4-3]
							\arrow["{S(f)_{T_2}}"', loose, from=4-1, to=4-3]
							\arrow["{L_\gamma}"', between={0.2}{0.8}, Rightarrow, from=0, to=1]
							\arrow["{\eta^{T_2}_f}"', between={0.3}{0.7}, Rightarrow, from=4-1, to=2]
						\end{tikzcd}
						&
					\end{align*}
				}
				Now, by the universality, we conclude that \longref{}{eqt:chi_t_natural} holds. Therefore, $\chi$ is a loose colax natural transformation. If $\zeta$ is actually a dotted-lax natural transformation, then $\chi$ would be an $\F$-natural transformation, which is tight in $\mathmybb{Fun}_{s, c}(\mathmybb{T}, \mathmybb{K})$.
				
				Altogether, we have just shown that for any marked (resp. dotted)-lax natural transformation $\zeta \colon \triangle(X) \to S$, there is a unique $\F$- (resp. loose colax) natural transformation $\chi \colon X \to L$ in $\mathmybb{Fun}_{s, c}(\mathmybb{T}, \mathmybb{K})$ such that
				\begin{equation*}
					\begin{tikzcd}[ampersand replacement=\&]
						\& X \& \\
						\& L \\
						{S(\dot{A})} \&\& {S(\dot{B})}
						\arrow["{\exists! \chi}", dashed, loose, from=1-2, to=2-2]
						\arrow["{\eta_{\dot{A}}}"', from=2-2, to=3-1]
						\arrow[""{name=0, anchor=center, inner sep=0}, "{\eta_{\dot{B}}}", from=2-2, to=3-3]
						\arrow["{S(f)}"', loose, from=3-1, to=3-3]
						\arrow["{\eta_f}"', between={0.3}{0.6}, Rightarrow, from=3-1, to=0]
					\end{tikzcd}
					=
					\begin{tikzcd}[ampersand replacement=\&]
						\& X \& \\
						\\
						{S(\dot{A})} \&\& {S(\dot{B})}
						\arrow["{\zeta_{\dot{A}}}"', loose, from=1-2, to=3-1]
						\arrow[""{name=0, anchor=center, inner sep=0}, "{\zeta_{\dot{B}}}", loose, from=1-2, to=3-3]
						\arrow["{S(f)}"', loose, from=3-1, to=3-3]
						\arrow["{\zeta_f}"', between={0.3}{0.6}, Rightarrow, from=3-1, to=0]
					\end{tikzcd};
				\end{equation*}
				here, $\chi$ is tight if and only if both $\zeta_{\dot{A}}$ and $\zeta_{\dot{B}}$ are tight, which is precisely when $\zeta$ is a tight morphism in $[\bbD, \mathmybb{Fun}_{s, c}(\mathmybb{T}, \mathmybb{K})]_{s, \Sigma, \Gamma}$. This is exactly the $1$-dimensional universal property of $L$ in $\mathmybb{Fun}_{s, c}(\mathmybb{T}, \mathmybb{K})$.
				
				\noindent\textbf{$\boldsymbol{2}$-dimensional universal property of $\boldsymbol{L}$:}
				
				Next, suppose there is a modification $Z \colon \zeta_1 \to \zeta_2$, where $\zeta_1, \zeta_2 \colon \triangle(X) \rightrightarrows S$ are both marked-lax natural transformations. Then, the component $Z_{\dot{A}}$ of $Z$ at $\dot{A} \in \ob\bbD$ is a $2$-morphism in  $\mathmybb{Fun}_{s, c}(\mathmybb{T}, \mathmybb{K})$, which is a modification $Z_A \colon {\zeta_1}_{\dot{A}} \to {\zeta_2}_{\dot{A}}$ between the loose colax natural transformations ${\zeta_1}_{\dot{A}}$ and ${\zeta_2}_{\dot{A}}$. The component $Z^T_{\dot{A}}$ of $Z_{\dot{A}}$ at $T \in \ob\mathmybb{T}$ is then a $2$-morphism
				\[
				\begin{tikzcd}[ampersand replacement=\&]
					{X(T)} \&\& {(S(\dot{A}))(T)}
					\arrow[""{name=0, anchor=center, inner sep=0}, "{{\zeta^T_2}_{\dot{A}}}"', curve={height=12pt}, loose, from=1-1, to=1-3, in = 210, out = 330]
					\arrow[""{name=1, anchor=center, inner sep=0}, "{{\zeta^T_1}_{\dot{A}}}", curve={height=-12pt}, loose, from=1-1, to=1-3, in = 150, out = 30]
					\arrow["{Z^T_{\dot{A}}}", between={0.3}{0.7}, Rightarrow, from=1, to=0]
				\end{tikzcd}
				\]
				in $\bbK$. We claim that $\{Z^T_{\dot{A}}\}_{\dot{A} \in \ob\bbD}$ assembles to a modification $Z^T \colon \zeta^T_1 \to \zeta^T_2$, where $\zeta^T_1, \zeta^T_2 \colon \triangle(X(T)) \to \mathrm{ev}_T \cdot S$ are both marked-lax natural transformations.
				
				First, it is clear that for a marked-lax natural transformation $\zeta \colon \triangle(X) \to S$, its evaluation $\zeta^T\colon \triangle(X(T)) \to \mathrm{ev}_T \cdot S$ at $T$ is a marked-lax natural transformation. It remains to show that the components of $Z^T$ satisfy the modification axiom, but that just follows directly from the fact that $Z$ is a modification. By the $2$-dimensional universal property of $L_T$, there is a unique $2$-morphism $x_T \colon {\chi_1}_T \to {\chi_2}_T$ in $\bbK$ such that
				\begin{equation}
					\label{eqt:x_T}
					\begin{tikzcd}[ampersand replacement=\&, column sep =0.5, row sep = large]
						\& {X(T)} \& \\
						\& {L_T} \\
						{(S(\dot{A}))(T)} \&\& {(S(\dot{B}))(T)}
						\arrow[""{name=0, anchor=center, inner sep=0}, "{{\chi_1}_T}"', curve={height=12pt}, loose, from=1-2, to=2-2]
						\arrow[""{name=1, anchor=center, inner sep=0}, "{{\chi_2}_T}", curve={height=-12pt}, loose, from=1-2, to=2-2]
						\arrow["{\eta^T_{\dot{A}}}"', from=2-2, to=3-1]
						\arrow[""{name=2, anchor=center, inner sep=0}, "{\eta^T_{\dot{B}}}", from=2-2, to=3-3]
						\arrow["{S(f)_T}"', loose, from=3-1, to=3-3]
						\arrow["{\exists! x_T}"', between={0.2}{0.8}, Rightarrow, dashed, from=0, to=1]
						\arrow["{\eta^T_f}", between={0.4}{0.6}, Rightarrow, from=3-1, to=2]
					\end{tikzcd}
					\hspace{-0.5em}
					=
					\hspace{-0.5em}
					\begin{tikzcd}[ampersand replacement=\&,column sep=0.5,row sep=large]
						\& {X(T)} \& \\
						\\
						{(S(\dot{A}))(T)} \&\& {(S(\dot{B}))(T)}
						\arrow[""{name=0, anchor=center, inner sep=0}, "{{\zeta^T_2}_{\dot{A}}}"{description}, curve={height=-12pt}, loose, 	from=1-2, to=3-1]
						\arrow[""{name=1, anchor=center, inner sep=0}, "{{\zeta^T_1}_{\dot{A}}}"', curve={height=12pt}, loose, from=1-2, 	to=3-1]
						\arrow[""{name=2, anchor=center, inner sep=0}, "{{\zeta^T_2}_{\dot{B}}}", loose, from=1-2, to=3-3]
						\arrow["{S(f)_T}"', loose, from=3-1, to=3-3]
						\arrow["{Z^T_{\dot{A}}}"', between={0.2}{0.8}, Rightarrow, from=1, to=0]
						\arrow["{{\zeta^T_2}_f}"', between={0.3}{0.6}, Rightarrow, from=3-1, to=2]
					\end{tikzcd}
					\hspace{-0.5em}
					=
					\hspace{-0.5em}
					\begin{tikzcd}[ampersand replacement=\&,column sep=0.5,row sep=large]
						\& {X(T)} \& \\
						\\
						{(S(\dot{A}))(T)} \&\& {(S(\dot{B}))(T)}
						\arrow["{{\zeta^T_1}_{\dot{A}}}"', loose, from=1-2, to=3-1]
						\arrow[""{name=0, anchor=center, inner sep=0}, "{{\zeta^T_2}_{\dot{B}}}", curve={height=-12pt}, loose, from=1-2, 	to=3-3]
						\arrow[""{name=1, anchor=center, inner sep=0}, "{{\zeta^T_1}_{\dot{B}}}"{description}, curve={height=12pt}, loose, 	from=1-2, to=3-3]
						\arrow["{S(f)_T}"', loose, from=3-1, to=3-3]
						\arrow["{Z^T_{\dot{B}}}"', between={0.2}{0.8}, Rightarrow, from=1, to=0]
						\arrow["{{\zeta^T_1}_{f}}"', between={0.3}{0.7}, Rightarrow, from=3-1, to=1]
					\end{tikzcd}.
				\end{equation}
				It then remains to show that $\{x_T\}_{T \in \ob\mathmybb{T}}$ assembles to a modification $x \colon \chi_1 \to \chi_2$ in $\mathmybb{Fun}_{s, c}(\mathmybb{T}, \mathmybb{K})$, which means we have to verify
				\begin{equation}
					\label{eqt:x_modification}
					\begin{tikzcd}[ampersand replacement=\&]
						{X(T_1)} \&\& {X(T_2)} \\
						{L_{T_1}} \&\& {L_{T_2}}
						\arrow["{X(t)}", loose, from=1-1, to=1-3]
						\arrow[""{name=0, anchor=center, inner sep=0}, "{{\chi_1}_{T_1}}", curve={height=-12pt}, loose, from=1-1, to=2-1]
						\arrow[""{name=1, anchor=center, inner sep=0}, "{{\chi_2}_{T_1}}"', curve={height=12pt}, loose, from=1-1, to=2-1]
						\arrow["{{\chi_1}_t}", between={0.4}{0.6}, Rightarrow, from=1-3, to=2-1]
						\arrow["{{\chi_1}_{T_2}}", loose, from=1-3, to=2-3]
						\arrow["{L_t}"', loose, from=2-1, to=2-3]
						\arrow["{x_{T_1}}", between={0.2}{0.8}, Rightarrow, from=0, to=1]
					\end{tikzcd}
					=
					\begin{tikzcd}[ampersand replacement=\&]
						{X(T_1)} \&\& {X(T_2)} \\
						{L_{T_1}} \&\& {L_{T_2}}
						\arrow["{X(t)}", loose, from=1-1, to=1-3]
						\arrow["{{\chi_2}_{T_1}}"', loose, from=1-1, to=2-1]
						\arrow["{{\chi_2}_t}"', between={0.4}{0.6}, Rightarrow, from=1-3, to=2-1]
						\arrow[""{name=0, anchor=center, inner sep=0}, "{{\chi_2}_{T_2}}"', curve={height=12pt}, loose, from=1-3, to=2-3]
						\arrow[""{name=1, anchor=center, inner sep=0}, "{{\chi_1}_{T_2}}", curve={height=-12pt}, loose, from=1-3, to=2-3]
						\arrow["{L_t}"', loose, from=2-1, to=2-3]
						\arrow["{x_{T_2}}", between={0.2}{0.8}, Rightarrow, from=1, to=0]
					\end{tikzcd}.
				\end{equation}
				In fact, we have
				{
					\allowdisplaybreaks
					\begin{align*}
						\begin{tikzcd}[ampersand replacement=\&, column sep=small]
							\& {X(T_1)} \& \\
							{X(T_2)} \&\& {L_{T_1}} \\
							\& {L_{T_2}} \\
							{(S(\dot{A}))({T_2})} \&\& {(S(\dot{B}))({T_2})}
							\arrow["{X(t)}"', loose, from=1-2, to=2-1]
							\arrow["{{\chi_2}_{T_1}}", loose, from=1-2, to=2-3]
							\arrow["{{\chi_2}_{t}}", between={0.4}{0.6}, Rightarrow, from=2-1, to=2-3]
							\arrow[""{name=0, anchor=center, inner sep=0}, "{{\chi_2}_{T_2}}"{description}, curve={height=-12pt}, loose, from=2-1, to=3-2]
							\arrow[""{name=1, anchor=center, inner sep=0}, "{{\chi_1}_{T_2}}"', curve={height=12pt}, loose, from=2-1, to=3-2]
							\arrow["{L_t}", loose, from=2-3, to=3-2]
							\arrow["{\eta^{T_2}_{\dot{A}}}"', from=3-2, to=4-1]
							\arrow[""{name=2, anchor=center, inner sep=0}, "{\eta^{T_2}_{\dot{B}}}", from=3-2, to=4-3]
							\arrow["{S(f)_{T_2}}"', loose, from=4-1, to=4-3]
							\arrow["{x_{T_2}}"', between={0.2}{0.8}, Rightarrow, from=1, to=0]
							\arrow["{\eta^{T_2}_f}", between={0.4}{0.6}, Rightarrow, from=4-1, to=2]
						\end{tikzcd}
						\stackunder{\longref{}{eqt:chi_t_leg_B}\&\longref{}{eqt:x_T}}{$=$}
						\begin{tikzcd}[ampersand replacement=\&]
							\& {X({T_2})} \&\& {X(T_1)} \\
							{(S(\dot{A}))({T_2})} \\
							\& {(S(\dot{B}))({T_2})} \&\& {(S(\dot{B}))(T_1)}
							\arrow[""{name=0, anchor=center, inner sep=0}, "{{\zeta^{T_2}_2}_{\dot{A}}}"{description}, curve={height=-12pt}, 	loose, from=1-2, to=2-1]
							\arrow[""{name=1, anchor=center, inner sep=0}, "{{\zeta^{T_2}_1}_{\dot{A}}}"', curve={height=12pt}, loose, from=1-2, to=2-1]
							\arrow[""{name=2, anchor=center, inner sep=0}, "{{\zeta^{T_2}_2}_{\dot{B}}}", loose, from=1-2, to=3-2]
							\arrow["{{\zeta^{t}_2}_{\dot{B}}}", between={0.4}{0.6}, Rightarrow, from=1-2, to=3-4]
							\arrow["{X(t)}"', loose, from=1-4, to=1-2]
							\arrow["{{\zeta^{T_1}_2}_{\dot{B}}}", loose, from=1-4, to=3-4]
							\arrow["{S(f)_{T_2}}"', loose, from=2-1, to=3-2]
							\arrow["{(S(\dot{B}))(t)}", loose, from=3-4, to=3-2]
							\arrow["{Z^{T_2}_{\dot{A}}}"', between={0.2}{0.8}, Rightarrow, from=1, to=0]
							\arrow["{{\zeta^{T_2}_2}_f}"', between={0.3}{0.6}, Rightarrow, from=2-1, to=2]
						\end{tikzcd}
						&
						\\
						\stackunder{\longref{}{eqt:zeta_f}}{$=$}
						\begin{tikzcd}[ampersand replacement=\&]
							\& {X(T_2)} \& {X(T_1)} \\
							{(S(\dot{A}))(T_2)} \& {(S(\dot{A}))(T_1)} \\
							\& {(S(\dot{B}))(T_2)} \& {(S(\dot{B}))(T_1)}
							\arrow[""{name=0, anchor=center, inner sep=0}, "{{\zeta_2}^{T_2}_{\dot{A}}}"{description}, curve={height=-12pt}, 	loose, from=1-2, to=2-1]
							\arrow[""{name=1, anchor=center, inner sep=0}, "{{\zeta_1}^{T_2}_{\dot{A}}}"', curve={height=12pt}, loose, from=1-2, to=2-1]
							\arrow["{{\zeta_2}^{t}_{\dot{A}}}", between={0.2}{0.8}, Rightarrow, from=1-2, to=2-2]
							\arrow["{X(t)}"', loose, from=1-3, to=1-2]
							\arrow["{{\zeta_2}^{T_1}_{\dot{A}}}"{description}, loose, from=1-3, to=2-2]
							\arrow[""{name=2, anchor=center, inner sep=0}, "{{\zeta_2}^{T_1}_{\dot{B}}}", loose, from=1-3, to=3-3]
							\arrow["{S(f)_{T_2}}"', loose, from=2-1, to=3-2]
							\arrow["{S(f)_t}"', between={0.4}{0.6}, Rightarrow, from=2-1, to=3-3]
							\arrow["{(S(\dot{A}))(t)}", loose, from=2-2, to=2-1]
							\arrow["{S(f)_{T_1}}"{description}, loose, from=2-2, to=3-3]
							\arrow["{(S(\dot{B}))(t)}", loose, from=3-3, to=3-2]
							\arrow["{Z^{T_2}_{\dot{A}}}", between={0.2}{0.8}, Rightarrow, from=1, to=0]
							\arrow["{{\zeta_2}^{T_1}_f}"', between={0.3}{0.7}, Rightarrow, from=2-2, to=2]
						\end{tikzcd}
						&
						\\
						\stackunder{$Z_A$ is a modification}{$=$}
						\begin{tikzcd}[ampersand replacement=\&]
							\& {X(T_2)} \& {X(T_1)} \\
							{(S(\dot{A}))(T_2)} \& {(S(\dot{A}))(T_1)} \\
							\& {(S(\dot{B}))(T_2)} \& {(S(\dot{B}))(T_1)}
							\arrow["{{\zeta_1}^{T_2}_{\dot{A}}}"', loose, from=1-2, to=2-1]
							\arrow["{{\zeta_1}^{t}_{\dot{A}}}"', between={0.2}{0.8}, Rightarrow, from=1-2, to=2-2]
							\arrow["{X(t)}"', loose, from=1-3, to=1-2]
							\arrow[""{name=0, anchor=center, inner sep=0}, "{{\zeta_2}^{T_1}_{\dot{A}}}"{description}, curve={height=-12pt}, 	loose, from=1-3, to=2-2]
							\arrow[""{name=1, anchor=center, inner sep=0}, "{{\zeta_1}^{T_1}_{\dot{A}}}"{description}, curve={height=12pt}, 	loose, from=1-3, to=2-2]
							\arrow[""{name=2, anchor=center, inner sep=0}, "{{\zeta_2}^{T_1}_{\dot{B}}}", loose, from=1-3, to=3-3]
							\arrow["{S(f)_{T_2}}"', loose, from=2-1, to=3-2]
							\arrow["{S(f)_t}"', between={0.4}{0.6}, Rightarrow, from=2-1, to=3-3]
							\arrow["{(S(\dot{A}))(t)}", loose, from=2-2, to=2-1]
							\arrow["{S(f)_{T_1}}"{description}, loose, from=2-2, to=3-3]
							\arrow["{(S(\dot{B}))(t)}", loose, from=3-3, to=3-2]
							\arrow["{Z^{T_1}_{\dot{A}}}", between={0.2}{0.8}, Rightarrow, from=1, to=0]
							\arrow["{{\zeta_2}^{T_1}_f}"', between={0.3}{0.7}, Rightarrow, from=2-2, to=2]
						\end{tikzcd}
						&
						\\
						\stackunder{\longref{}{eqt:x_T}}{$=$}
						\begin{tikzcd}[ampersand replacement=\&]
							\& {X(T_2)} \& {X(T_1)} \\
							{(S(\dot{A}))(T_2)} \& {(S(\dot{A}))(T_1)} \& {L_{T_1}} \\
							\& {(S(\dot{B}))(T_2)} \& {(S(\dot{B}))(T_1)}
							\arrow["{{\zeta_1}^{T_2}_{\dot{A}}}"', loose, from=1-2, to=2-1]
							\arrow["{{\zeta_1}^{t}_{\dot{A}}}"', between={0.2}{0.8}, Rightarrow, from=1-2, to=2-2]
							\arrow["{X(t)}"', loose, from=1-3, to=1-2]
							\arrow[""{name=0, anchor=center, inner sep=0}, "{{\chi_2}_{T_1}}", curve={height=-12pt}, loose, from=1-3, to=2-3]
							\arrow[""{name=1, anchor=center, inner sep=0}, "{{\chi_1}_{T_1}}"', curve={height=12pt}, loose, from=1-3, to=2-3]
							\arrow["{S(f)_{T_2}}"', loose, from=2-1, to=3-2]
							\arrow["{S(f)_t}"', between={0.4}{0.6}, Rightarrow, from=2-1, to=3-3]
							\arrow["{(S(\dot{A}))(t)}", loose, from=2-2, to=2-1]
							\arrow["{S(f)_{T_1}}"{description}, loose, from=2-2, to=3-3]
							\arrow["{\eta^{T_1}_{\dot{A}}}"{description}, from=2-3, to=2-2]
							\arrow[""{name=2, anchor=center, inner sep=0}, "{\eta^{T_1}_{\dot{B}}}", from=2-3, to=3-3]
							\arrow["{(S(\dot{B}))(t)}", loose, from=3-3, to=3-2]
							\arrow["{x_{T_1}}", between={0.2}{0.8}, Rightarrow, from=1, to=0]
							\arrow["{\scriptscriptstyle{\eta^{T_1}_f}}"{description}, between={0.3}{0.7}, Rightarrow, from=2-2, to=2]
						\end{tikzcd}
						&
						\\
						\stackunder{\longref{}{eqt:uni_L_t}}{$=$}
						\begin{tikzcd}[ampersand replacement=\&]
							\& {X(T_2)} \& {X(T_1)} \\
							{(S(\dot{A}))(T_2)} \&\& {L_{T_1}} \\
							{(S(\dot{B}))(T_2)} \&\& {L_{T_2}}
							\arrow["{{\zeta_1}^{T_2}_{\dot{A}}}"', loose, from=1-2, to=2-1]
							\arrow["{X(t)}"', loose, from=1-3, to=1-2]
							\arrow[""{name=0, anchor=center, inner sep=0}, "{{\chi_2}_{T_1}}", curve={height=-12pt}, loose, from=1-3, to=2-3]
							\arrow[""{name=1, anchor=center, inner sep=0}, "{{\chi_1}_{T_1}}"', curve={height=12pt}, loose, from=1-3, to=2-3]
							\arrow["{{\zeta_1}^{t}_{\dot{A}}}", between={0.5}{0.7}, Rightarrow, from=2-1, to=2-3]
							\arrow["{S(f)_{T_2}}"', loose, from=2-1, to=3-1]
							\arrow["{L_t}", loose, from=2-3, to=3-3]
							\arrow["{\eta^{T_2}_{\dot{A}}}"', from=3-3, to=2-1]
							\arrow[""{name=2, anchor=center, inner sep=0}, "{\eta^{T_2}_{\dot{B}}}", from=3-3, to=3-1]
							\arrow["{x_{T_1}}", between={0.2}{0.8}, Rightarrow, from=1, to=0]
							\arrow["{\eta^{T_2}_f}"', between={0.3}{0.7}, Rightarrow, from=2-1, to=2]
						\end{tikzcd}
						&
						\\
						\stackunder{\longref{}{eqt:chi_t_leg}}{$=$}
						\begin{tikzcd}[ampersand replacement=\&,column sep=small]
							\& {X(T_1)} \& \\
							{X(T_2)} \&\& {L_{T_1}} \\
							\& {L_{T_2}} \\
							{(S(\dot{A}))({T_2})} \&\& {(S(\dot{B}))({T_2})}
							\arrow["{X(t)}"', loose, from=1-2, to=2-1]
							\arrow[""{name=0, anchor=center, inner sep=0}, "{{\chi_2}_{T_1}}", curve={height=-12pt}, loose, from=1-2, to=2-3]
							\arrow[""{name=1, anchor=center, inner sep=0}, "{{\chi_1}_{T_1}}"{description}, curve={height=12pt}, loose, from=1-2, to=2-3]
							\arrow["{{\chi_1}_{t}}", between={0.4}{0.6}, Rightarrow, from=2-1, to=2-3]
							\arrow["{{\chi_1}_{T_2}}"', loose, from=2-1, to=3-2]
							\arrow["{L_t}", loose, from=2-3, to=3-2]
							\arrow["{\eta^{T_2}_{\dot{A}}}"', from=3-2, to=4-1]
							\arrow[""{name=2, anchor=center, inner sep=0}, "{\eta^{T_2}_{\dot{B}}}", from=3-2, to=4-3]
							\arrow["{S(f)_{T_2}}"', loose, from=4-1, to=4-3]
							\arrow["{x_{T_1}}"', between={0.2}{0.8}, Rightarrow, from=1, to=0]
							\arrow["{\eta^{T_2}_f}", between={0.4}{0.6}, Rightarrow, from=4-1, to=2]
						\end{tikzcd}
						&
					\end{align*}
				}
				As a consequence, \longref{}{eqt:x_modification} holds, and $L$ exhibits also the $2$-dimensional universal property of the colax limit of $S(f)$.
				
				All-in-all, we have proved that $L$ is the dotted-lax limit $\dotlim{l}{S}$ of $S$ in $\mathmybb{Fun}_{s, c}(\mathmybb{T}, \mathmybb{K})$.
				
				Next, suppose further that both $S(\dot{A})$ and $S(\dot{B})$ are $\F$-models. This means that for any $\F$-weighted cone
				$$(\Phi\colon \mathccal{I} \to \bbF, \; F \colon \mathccal{I} \to \mathmybb{T}, \; X \in \mathmybb{T}, \; \phi \colon \Phi \to \mathmybb{T}(X, F-))$$
				where $\mathccal{I}$ is a $2$-category, the composites
				\begin{align*}
					&\Phi \xrightarrow{\phi} \mathmybb{T}(X, F-) \xrightarrow{S(\dot{A})} \bbK((S(\dot{A}))(X), S(\dot{A}) \cdot F-),
					\\
					&\Phi \xrightarrow{\phi} \mathmybb{T}(X, F-) \xrightarrow{S(\dot{B})} \bbK((S(\dot{B}))(X), S(\dot{B}) \cdot F-),
				\end{align*}
				are both limit cones in $\bbK$. By \cite[Proposition 5.3.0.8]{Ko:2023}, these two $\F$-natural transformations are equivalent to the dotted-lax natural transformations
				\begin{align*}
					&\triangle((S(\dot{A}))(X)) \to S(\dot{A}) \cdot F \cdot P,
					\\
					&\triangle((S(\dot{B}))(X)) \to S(\dot{B}) \cdot F \cdot P,
				\end{align*}
				where $P \colon \bbEl{\Phi} \to \mathccal{I}$ is the canonical projection from the dotted $\F$-category of elements of $\Phi$ constructed in \cite[\S5.3]{Ko:2023}. Therefore, to say that the two $\F$-natural transformations are limit cones in $\bbK$, is equivalent to saying that the above dotted-lax natural transformations correspond to dotted-lax limits in \longref{Definition}{def:dotted_lim}. Since every morphism in $\bbEl{\Phi}$ is of the form $(d \in \mor\mathccal{I}, \omega)$, it is clear that every morphism  in $\bbEl{\Phi}$ is tight; also, all objects in  in $\bbEl{\Phi}$ are dotted. This allows us to prove our desired statement in a slightly more general setting, as follows.
				
				Let $(\calJ, \Sigma_\calJ, \Gamma_\calJ)$ be a dotted $2$-category. Let $D \colon \calJ \to \mathmybb{T}$ be an $\F$-functor, and $X$ be an object of $\mathmybb{T}$. Consider a dotted-lax natural transformation
				$$\triangle(X) \xrightarrow{\delta} D$$
				in $[\calJ, \mathmybb{T}]_{s, \Sigma_\calJ, \Gamma_\calJ}$. We prove that if both the images		
				\begin{align*}
					&\triangle((S(\dot{A}))(X)) \to S(\dot{A}) \cdot D,
					\\
					&\triangle((S(\dot{B}))(X)) \to S(\dot{B}) \cdot D,
				\end{align*}
				of $\delta$ under the post-compositions $(S(\dot{A}))_*$ and $(S(\dot{B}))_*$ are dotted-lax limit cones in $\bbK$, then the image
				$$\triangle(L(X)) \to L \cdot D$$
				of $\delta$ under the post-composition $L_*$ by $L$ is also a dotted-lax limit cone in $\bbK$.
				
				We first view $L \cdot D$ as an $\F$-functor $$\calJ \times \bbD \xrightarrow{\pi_1} \calJ \xrightarrow{D} \mathmybb{T} \xrightarrow{L} \bbK,$$
				where $\pi_1$ is the evident projection. We can then view $L \cdot \delta \colon \triangle(L(X)) \to L \cdot D$ as a dotted-lax natural transformation in $[\calJ \times \bbD, \bbK]_{s, \Sigma_\calJ \times \Sigma_\bbD, \Gamma_\calJ \times \Gamma_\bbD}$; here, a morphism is marked precisely when both the coordinates are marked, and an object is dotted precisely when both the coordinates are dotted. After that, we consider the $\F$-functor
				\begin{align*}
					\mathrm{ev}_{D(-)} \cdot S \colon \calJ \times \bbD &\to \bbK
					\\
					(J, \dot{A}) &\mapsto (S(\dot{A}))(D(J))
					\\
					(j \colon J_1 \to J_2, f \colon \dot{A} \leadsto \dot{B}) &\mapsto S(f)_{D(J_2)} \cdot (S(\dot{A}))(D(j)),
				\end{align*}
				where the functoriality is guaranteed by the fact that for any composable pair $J_1 \xrightarrow{j} J_2 \xrightarrow{k} J_3$ of morphisms in $\calJ$, we must have
				\begin{align*}
					S(f)_{D(J_3)} \cdot (S(\dot{A}))(D(k \cdot j)) &= S(f)_{D(J_3)} \cdot (S(\dot{A}))(D(k)) \cdot (S(\dot{A}))(D(j))
					\\
					&= (S(\dot{B}))(D(k)) \cdot S(f)_{D(J_2)} \cdot (S(\dot{A}))(D(j)),
				\end{align*}
				as $S(f)_{D(j)} = 1$, because $\calJ$ is chordate.
				
				We proceed to construct a dotted-lax natural transformation
				$$\eta^{D(-)}_{(-)} \colon L\cdot D \to \mathrm{ev}_{D(-)} \cdot S$$
				in $[\calJ \times \bbD, \bbK]_{s, \Sigma_\calJ \times \Sigma_\bbD, \Gamma_\calJ \times \Gamma_\bbD}$ as follows. Following \longref{}{diag:eta_X}, we set its $1$-component at $(J, \dot{A}) \in \ob(\calJ \times \bbD)$ to be $\eta^{D(J)}_{\dot{A}} \colon L(D(J)) \to (S(\dot{A}))(D(J))$, and its $2$-component at $(j \colon J_1 \to J_2, f \colon \dot{A} \leadsto \dot{B})$ to be the composite
				\[
				\begin{tikzcd}[ampersand replacement=\&,column sep=small]
					\& {L(D(J_1))} \&\& {L(D(J_2))} \& \\
					{(S(\dot{A}))(D(J_1))} \&\& {(S(\dot{B}))(D(J_1))} \&\& {(S(\dot{B}))(D(J_2))}
					\arrow["{ L(D(j)) = L_{D(j)}}", from=1-2, to=1-4]
					\arrow["{\eta^{D(J_1)}_{\dot{A}}}"', from=1-2, to=2-1]
					\arrow[""{name=0, anchor=center, inner sep=0}, "{\eta^{D(J_1)}_{\dot{B}}}", from=1-2, to=2-3]
					\arrow["{\eta^{D(j)}_{\dot{B}}}"{description}, draw=none, from=1-4, to=2-3]
					\arrow["{\eta^{D(J_2)}_{\dot{B}}}", from=1-4, to=2-5]
					\arrow["{S(f)_{D(J_1)}}"', loose, from=2-1, to=2-3]
					\arrow["{(S(\dot{B}))(D(j))}"', from=2-3, to=2-5]
					\arrow["{\eta^{D(J_1)}_f}", between={0.4}{0.6}, Rightarrow, from=2-1, to=0]
				\end{tikzcd}
				\]
				which we denote by $\eta^{D(j)}_f$. The data clearly assemble to a dotted-lax natural transformation, as $\eta^{D(J)}$ is one and $\eta_{\dot{B}}$ is an $\F$-natural transformation. Consequently, we obtain a composite of dotted-lax natural transformations
				$$\triangle(L(X)) \xrightarrow{L\cdot \delta} L\cdot D \xrightarrow{\eta^{D(-)}_{(-)}} \mathrm{ev}_{D(-)} \cdot S,$$
				as a tight morphism in $[\calJ \times \bbD, \bbK]_{s, \Sigma_\calJ \times \Sigma_\bbD, \Gamma_\calJ \times \Gamma_\bbD}$, whose $1$-component at $(J, \dot{A}) \in \ob(\calJ \times \bbD)$ is given by
				$$L(X) \xrightarrow{L(\delta_J)} L(D(J)) \xrightarrow{\eta^{D(J)}_{\dot{A}}} (S(\dot{A})(D(J))),$$
				and $2$-component at $(j \colon J_1 \to J_2, f)$ is given by
				\[
				\begin{tikzcd}[ampersand replacement=\&,column sep=small]
					\&\& {L(X)} \&\& \\
					\& {L(D(J_1))} \&\& {L(D(J_2))} \\
					{(S(\dot{A}))(D(J_1))} \&\& {(S(\dot{B}))(D(J_1))} \&\& {(S(\dot{B}))(D(J_2))}
					\arrow["{L(\delta_{J_1})}"', from=1-3, to=2-2]
					\arrow[""{name=0, anchor=center, inner sep=0}, "{L(\delta_{J_1})}", from=1-3, to=2-4]
					\arrow["{ L(D(j))}"{description}, from=2-2, to=2-4]
					\arrow["{\eta^{D(J_1)}_{\dot{A}}}"', from=2-2, to=3-1]
					\arrow[""{name=1, anchor=center, inner sep=0}, "{\eta^{D(J_1)}_{\dot{B}}}", from=2-2, to=3-3]
					\arrow["{\eta^{D(j)}_{\dot{B}}}"{description}, draw=none, from=2-4, to=3-3]
					\arrow["{\eta^{D(J_2)}_{\dot{B}}}", from=2-4, to=3-5]
					\arrow["{S(f)_{D(J_1)}}"', loose, from=3-1, to=3-3]
					\arrow["{(S(\dot{B}))(D(j))}"', from=3-3, to=3-5]
					\arrow["{L(\delta_j)}", between={0.4}{0.6}, Rightarrow, from=2-2, to=0]
					\arrow["{\eta^{D(J_1)}_f}", between={0.4}{0.6}, Rightarrow, from=3-1, to=1]
				\end{tikzcd}
				\]
				in $\bbK$.
				
				For an object $X$ of $\mathmybb{T}$, we can view $\mathrm{ev}_X \cdot S$ as an $\F$-functor
				$$\calJ \times \bbD \xrightarrow{\pi_1} \bbD \xrightarrow{S} \mathmybb{Fun}_{s, c}(\mathmybb{T}, \mathmybb{K}) \xrightarrow{\mathrm{ev}_X} \bbK.$$
				We construct another dotted-lax natural transformation
				$$(S(-))(\delta(-)) \colon \mathrm{ev}_X \cdot S \to \mathrm{ev}_{D(-)} \cdot S$$
				in $[\calJ \times \bbD, \bbK]_{s, \Sigma_\calJ \times \Sigma_\bbD, \Gamma_\calJ \times \Gamma_\bbD}$ as follows. We set its $1$-component at $(J, \dot{A})$ to be $(S(\dot{A}))(\delta_J) \colon (S(\dot{A}))(X) \to (S(\dot{A}))(D(J))$, and its $2$-component at $(j, f)$ to be the composite
				\[
				\begin{tikzcd}[ampersand replacement=\&,column sep=small]
					\& {(S(\dot{A}))(X)} \&\& {(S(\dot{B}))(X)} \& \\
					{(S(\dot{A}))(D(J_1))} \&\& {(S(\dot{B}))(D(J_1))} \&\& {(S(\dot{B}))(D(J_2))}
					\arrow["{S(f)_X}", loose, from=1-2, to=1-4]
					\arrow["{(S(\dot{A}))(\delta_{J_1})}"', from=1-2, to=2-1]
					\arrow["{S(f)_{\delta_{J_1}}}"{description}, draw=none, from=1-4, to=2-1]
					\arrow["{(S(\dot{B}))(\delta_{J_1})}"{description}, from=1-4, to=2-3]
					\arrow[""{name=0, anchor=center, inner sep=0}, "{(S(\dot{B}))(\delta_{J_2})}", from=1-4, to=2-5]
					\arrow["{S(f)_{D(J_1)}}"', loose, from=2-1, to=2-3]
					\arrow["{(S(\dot{B}))(D(j))}"', from=2-3, to=2-5]
					\arrow["{(S(\dot{B}))(\delta_{j})}"', between={0.4}{0.6}, Rightarrow, from=2-3, to=0]
				\end{tikzcd}
				\]
				which we denote by $S(f)(\delta_j)$. Note that here $S(f)_{\delta_{J_1}}$ is the identity because $\delta_{J_1}$ is always tight for any $J_1 \in \ob\calJ$. The data clearly assemble to a dotted-lax natural transformation, as $S(f)$ is a loose-lax natural transformation and $(S(\dot{B}))(\delta)$ is a dotted-lax natural transformation. As a consequence, we obtain a composite
				$$\triangle(L(X)) \xrightarrow{\eta^X} \mathrm{ev}_X \cdot S \xrightarrow{(S(-))(\delta(-))} \mathrm{ev}_{D(-)} \cdot S$$
				of dotted-lax natural transformations, which is a tight morphism in $[\calJ \times \bbD, \bbK]_{s, \Sigma_\calJ \times \Sigma_\bbD, \Gamma_\calJ \times \Gamma_\bbD}$, whose $1$-component at $(J, \dot{A})$ is given by the composite
				$$L(X) \xrightarrow{\eta^X_{\dot{A}}} (S(\dot{A}))(X) \xrightarrow{(S(\dot{A}))(\delta_J)} (S(\dot{A}))(D(J)),$$
				and $2$-component at $(j, f)$ is given by
				\[
				\begin{tikzcd}[ampersand replacement=\&,column sep=small]
					\&\& {L(X)} \&\& \\
					\& {(S(\dot{A}))(X)} \&\& {(S(\dot{B}))(X)} \\
					{(S(\dot{A}))(D(J_1))} \&\& {(S(\dot{B}))(D(J_1))} \&\& {(S(\dot{B}))(D(J_2))}
					\arrow["{\eta^X_{\dot{A}}}"', from=1-3, to=2-2]
					\arrow[""{name=0, anchor=center, inner sep=0}, "{\eta^X_{\dot{B}}}", from=1-3, to=2-4]
					\arrow["{S(f)_X}"{description}, loose, from=2-2, to=2-4]
					\arrow["{(S(\dot{A}))(\delta_{J_1})}"', from=2-2, to=3-1]
					\arrow["{S(f)_{\delta_{J_1}}}"{description}, draw=none, from=2-4, to=3-1]
					\arrow["{(S(\dot{B}))(\delta_{J_1})}"{description}, from=2-4, to=3-3]
					\arrow[""{name=1, anchor=center, inner sep=0}, "{(S(\dot{B}))(\delta_{J_2})}", from=2-4, to=3-5]
					\arrow["{S(f)_{D(J_1)}}"', loose, from=3-1, to=3-3]
					\arrow["{(S(\dot{B}))(D(j))}"', from=3-3, to=3-5]
					\arrow["{\eta^X_f}", between={0.4}{0.6}, Rightarrow, from=2-2, to=0]
					\arrow["{(S(\dot{B}))(\delta_{j})}"', between={0.4}{0.6}, Rightarrow, from=3-3, to=1]
				\end{tikzcd}
				\]
				in $\bbK$.
				
				We claim that the two composites of dotted-lax natural transformations are equal, i.e.,
				\begin{equation}
					\label{eqt:equal_trans}
					\eta^{D(-)}_{(-)} \cdot (L \cdot \delta) = S(-)(\delta(-)) \cdot \eta^X.
				\end{equation}
				Indeed, by the universal property of $L(D(J_1))$, the $2$-component of $S(-)(\delta(-)) \cdot \eta^X$ at $(j, f)$ is
				{
					\allowdisplaybreaks
					\begin{align*}
						(S(-)(\delta(-)) \cdot \eta^X)_{(j, f)}
						=
						\hspace{-2em}
						\begin{tikzcd}[ampersand replacement=\&,column sep=tiny]
							\&\& {L(X)} \&\& \\
							\& {(S(\dot{A}))(X)} \&\& {(S(\dot{B}))(X)} \\
							{(S(\dot{A}))(D(J_1))} \&\& {(S(\dot{B}))(D(J_1))} \&\& {(S(\dot{B}))(D(J_2))}
							\arrow["{\eta^X_{\dot{A}}}"', from=1-3, to=2-2]
							\arrow[""{name=0, anchor=center, inner sep=0}, "{\eta^X_{\dot{B}}}", from=1-3, to=2-4]
							\arrow["{S(f)_X}"{description}, loose, from=2-2, to=2-4]
							\arrow["{(S(\dot{A}))(\delta_{J_1})}"', from=2-2, to=3-1]
							\arrow["{S(f)_{\delta_{J_1}}}"{description}, draw=none, from=2-4, to=3-1]
							\arrow["{(S(\dot{B}))(\delta_{J_1})}"{description}, from=2-4, to=3-3]
							\arrow[""{name=1, anchor=center, inner sep=0}, "{(S(\dot{B}))(\delta_{J_2})}", from=2-4, to=3-5]
							\arrow["{S(f)_{D(J_1)}}"', loose, from=3-1, to=3-3]
							\arrow["{(S(\dot{B}))(D(j))}"', from=3-3, to=3-5]
							\arrow["{\eta^X_f}", between={0.4}{0.6}, Rightarrow, from=2-2, to=0]
							\arrow["{(S(\dot{B}))(\delta_{j})}"', between={0.4}{0.6}, Rightarrow, from=3-3, to=1]
						\end{tikzcd}
						&
						\\
						\stackunder{\longref{}{eqt:uni_L_t}}{$=$}
						\hspace{-2em}
						\begin{tikzcd}[ampersand replacement=\&,column sep=tiny]
							\&\& {L(X)} \&\& \\
							\& {L(D(J_1))} \&\& {(S(\dot{B}))(X)} \\
							{(S(\dot{A}))(D(J_1))} \&\& {(S(\dot{B}))(D(J_1))} \&\& {(S(\dot{B}))(D(J_2))}
							\arrow["{L(\delta_{J_1})}"', from=1-3, to=2-2]
							\arrow["{\eta^X_{\dot{B}}}", from=1-3, to=2-4]
							\arrow["{\eta^{\delta_{J_1}}_{\dot{B}}}"{description}, draw=none, from=1-3, to=3-3]
							\arrow["{\eta^{D(J_1)}_{\dot{A}}}"', from=2-2, to=3-1]
							\arrow[""{name=0, anchor=center, inner sep=0}, "{\eta^{D(J_1)}_{\dot{B}}}", from=2-2, to=3-3]
							\arrow["{(S(\dot{B}))(\delta_{J_1})}"{description}, from=2-4, to=3-3]
							\arrow[""{name=1, anchor=center, inner sep=0}, "{(S(\dot{B}))(\delta_{J_2})}", from=2-4, to=3-5]
							\arrow["{S(f)_{D(J_1)}}"', loose, from=3-1, to=3-3]
							\arrow["{(S(\dot{B}))(D(j))}"', from=3-3, to=3-5]
							\arrow["{\eta^{D(J_1)}_{f}}", between={0.4}{0.6}, Rightarrow, from=3-1, to=0]
							\arrow["{(S(\dot{B}))(\delta_{j})}"', between={0.4}{0.6}, Rightarrow, from=3-3, to=1]
						\end{tikzcd}
						&
						\\
						\stackunder{$\eta_{\dot{B}}$ is an $\F$-natural transformation}{$=$}
						\hspace{-6em}
						\begin{tikzcd}[ampersand replacement=\&,column sep=small]
							\&\& {L(X)} \&\& \\
							\& {L(D(J_1))} \&\& {L(D(J_2))} \\
							{(S(\dot{A}))(D(J_1))} \&\& {(S(\dot{B}))(D(J_1))} \&\& {(S(\dot{B}))(D(J_2))}
							\arrow["{L(\delta_{J_1})}"', from=1-3, to=2-2]
							\arrow[""{name=0, anchor=center, inner sep=0}, "{L(\delta_{J_1})}", from=1-3, to=2-4]
							\arrow["{ L(D(j))}"{description}, from=2-2, to=2-4]
							\arrow["{\eta^{D(J_1)}_{\dot{A}}}"', from=2-2, to=3-1]
							\arrow[""{name=1, anchor=center, inner sep=0}, "{\eta^{D(J_1)}_{\dot{B}}}", from=2-2, to=3-3]
							\arrow["{\eta^{D(j)}_{\dot{B}}}"{description}, draw=none, from=2-4, to=3-3]
							\arrow["{\eta^{D(J_2)}_{\dot{B}}}", from=2-4, to=3-5]
							\arrow["{S(f)_{D(J_1)}}"', loose, from=3-1, to=3-3]
							\arrow["{(S(\dot{B}))(D(j))}"', from=3-3, to=3-5]
							\arrow["{L(\delta_j)}", between={0.4}{0.6}, Rightarrow, from=2-2, to=0]
							\arrow["{\eta^{D(J_1)}_f}", between={0.4}{0.6}, Rightarrow, from=3-1, to=1]
						\end{tikzcd},
						&
					\end{align*}
				}
				which is exactly the $2$-component of $\eta^{D(-)}_{(-)} \cdot (L \cdot \delta)$ at $(j, f)$. Hence, \longref{Equation}{eqt:equal_trans} holds.
				
				\noindent\textbf{$\boldsymbol{1}$-dimensional universal property of $\boldsymbol{L \cdot \delta}$:}
				
				With all these necessary constructions, we proceed to establish the $1$-dimensional universal property of $L \cdot \delta$. Let $E$ be an object of $\bbK$, and $\epsilon \colon \triangle(E) \to L\cdot D$ be any arbitrary marked-lax natural transformation in $[\calJ, \bbK]_{s, \Sigma_\calJ, \Gamma_\calJ}$. Our goal is to show that there exists a unique morphism $w \colon E \leadsto L(X)$ in $\bbK$ such that $(L\cdot \delta) \cdot \triangle(w) = \epsilon$; furthermore, if $\epsilon$ is actually a dotted-lax natural transformation, then $w$ is tight.
				
				To begin with, we view $\epsilon \colon \triangle(E) \to L\cdot D$ as a marked-lax natural transformation in $[\calJ \times \bbD, \bbK]_{s, \Sigma_\calJ \times \Sigma_\bbD, \Gamma_\calJ \times \Gamma_\bbD}$, by composing with the canonical projection $\calJ \times \bbD \to \calJ$. Then, consider the composite
				$$\triangle(E) \xrightarrow{\epsilon} L\cdot D \xrightarrow{\eta^{D(-)}_{(-)}} \mathrm{ev}_{D(-)} \cdot S$$
				of marked-lax natural transformations. Since
				\begin{align*}
					&\triangle((S(\dot{A}))(X)) \to S(\dot{A}) \cdot D,
					\\
					&\triangle((S(\dot{B}))(X)) \to S(\dot{B}) \cdot D,
				\end{align*}
				are limit cones by assumption, so there exists a unique morphism $\alpha_{\dot{A}} \colon E \leadsto (S(\dot{A}))(X)$ and a unique morphism  $\alpha_{\dot{B}} \colon E \leadsto (S(\dot{B}))(X)$ such that
				\begin{equation}
					\label{eqt:epsilon_A}
					\begin{tikzcd}[ampersand replacement=\&,column sep=small]
						\& E \& \\
						\& {(S(\dot{A}))(X)} \\
						{(S(\dot{A}))(D(J_1))} \&\& {(S(\dot{A}))(D(J_2))}
						\arrow["{\exists!\alpha_{\dot{A}}}", dashed, loose, from=1-2, to=2-2]
						\arrow["{(S(\dot{A}))(\delta_{J_1})}"', from=2-2, to=3-1]
						\arrow[""{name=0, anchor=center, inner sep=0}, "{(S(\dot{A}))(\delta_{J_2})}", from=2-2, to=3-3]
						\arrow["{(S(\dot{A}))(D(j))}"', from=3-1, to=3-3]
						\arrow["{(S(\dot{A}))(\delta_{j})}"', between={0.4}{0.6}, Rightarrow, from=3-1, to=0]
					\end{tikzcd}
					=
					\begin{tikzcd}[ampersand replacement=\&,column sep=small]
						\& E \& \\
						{L(D(J_1))} \&\& {L(D(J_2))} \\
						{(S(\dot{A}))(D(J_1))} \&\& {(S(\dot{A}))(D(J_2))}
						\arrow["{\epsilon_{J_1}}"', loose, from=1-2, to=2-1]
						\arrow[""{name=0, anchor=center, inner sep=0}, "{\epsilon_{J_2}}", loose, from=1-2, to=2-3]
						\arrow["{L(D(j))}"{description}, from=2-1, to=2-3]
						\arrow["{\eta^{D(J_1)}_{\dot{A}}}"', from=2-1, to=3-1]
						\arrow["{\eta^{D(j)}_{\dot{A}}}"{description}, draw=none, from=2-1, to=3-3]
						\arrow["{\eta^{D(J_2)}_{\dot{A}}}", from=2-3, to=3-3]
						\arrow["{(S(\dot{A}))(D(j))}"', from=3-1, to=3-3]
						\arrow["{\epsilon_j}", between={0.4}{0.6}, Rightarrow, from=2-1, to=0]
					\end{tikzcd}
				\end{equation}
				and
				\begin{equation*}
					\begin{tikzcd}[ampersand replacement=\&,column sep=small]
						\& E \& \\
						\& {(S(\dot{B}))(X)} \\
						{(S(\dot{B}))(D(J_1))} \&\& {(S(\dot{B}))(D(J_2))}
						\arrow["{\exists!\alpha_{\dot{B}}}", dashed, loose, from=1-2, to=2-2]
						\arrow["{(S(\dot{B}))(\delta_{J_1})}"', from=2-2, to=3-1]
						\arrow[""{name=0, anchor=center, inner sep=0}, "{(S(\dot{B}))(\delta_{J_2})}", from=2-2, to=3-3]
						\arrow["{(S(\dot{B}))(D(j))}"', from=3-1, to=3-3]
						\arrow["{(S(\dot{B}))(\delta_{j})}"', between={0.4}{0.6}, Rightarrow, from=3-1, to=0]
					\end{tikzcd}
					=
					\begin{tikzcd}[ampersand replacement=\&,column sep=small]
						\& E \& \\
						{L(D(J_1))} \&\& {L(D(J_2))} \\
						{(S(\dot{B}))(D(J_1))} \&\& {(S(\dot{B}))(D(J_2))}
						\arrow["{\epsilon_{J_1}}"', loose, from=1-2, to=2-1]
						\arrow[""{name=0, anchor=center, inner sep=0}, "{\epsilon_{J_2}}", loose, from=1-2, to=2-3]
						\arrow["{L(D(j))}"{description}, from=2-1, to=2-3]
						\arrow["{\eta^{D(J_1)}_{\dot{B}}}"', from=2-1, to=3-1]
						\arrow["{\eta^{D(j)}_{\dot{B}}}"{description}, draw=none, from=2-1, to=3-3]
						\arrow["{\eta^{D(J_2)}_{\dot{B}}}", from=2-3, to=3-3]
						\arrow["{(S(\dot{B}))(D(j))}"', from=3-1, to=3-3]
						\arrow["{\epsilon_j}", between={0.4}{0.6}, Rightarrow, from=2-1, to=0]
					\end{tikzcd},
				\end{equation*}
				for any $j \colon J_1 \to J_2$ in $\calJ$. Clearly, if $\epsilon$ is a dotted-lax natural transformation, then both $\alpha_{\dot{A}}$ and $\alpha_{\dot{B}}$ would be tight. Let $\varphi_j$ denote the composite
				\[
				\begin{tikzcd}[ampersand replacement=\&]
					\& E \& \\
					\& {(S(\dot{A}))(X)} \\
					{(S(\dot{A}))(D(J_1))} \&\& {(S(\dot{A}))(D(J_1))} \\
					{(S(\dot{B}))(D(J_1))} \&\& {(S(\dot{B}))(D(J_2))}
					\arrow["{\alpha_{\dot{A}}}", loose, from=1-2, to=2-2]
					\arrow["{(S(\dot{A}))(\delta_{J_1})}"', from=2-2, to=3-1]
					\arrow[""{name=0, anchor=center, inner sep=0}, "{(S(\dot{A}))(\delta_{J_2})}", from=2-2, to=3-3]
					\arrow["{(S(\dot{A}))(D(j))}"{description}, from=3-1, to=3-3]
					\arrow["{S(f)_{D(J_1)}}"', loose, from=3-1, to=4-1]
					\arrow["{S(f)_{D(j)}}"{description}, draw=none, from=3-3, to=4-1]
					\arrow["{S(f)_{D(J_2)}}", loose, from=3-3, to=4-3]
					\arrow["{(S(\dot{B}))(D(j))}"', from=4-1, to=4-3]
					\arrow["{(S(\dot{A}))(\delta_{j})}", between={0.4}{0.6}, Rightarrow, from=3-1, to=0]
				\end{tikzcd};
				\]
				here, $S(f)_{D(j)}$ is the identity, as $j$ is always tight for any $j \in \calJ$. Obviously, $\{\varphi_j\}_{j \in \ob\calJ}$ induces a marked-lax natural transformation
				$$\varphi \colon \triangle(E) \to (S(\dot{B}))(D(-))$$
				in $[\calJ, \bbK]_{s, \Sigma_\calJ, \Gamma_\calJ}$. Then, we construct a modification $\Xi \colon \varphi \colon \eta^{D(-)}_{\dot{B}} \cdot \epsilon$ as follows. For an object $J$ of $\calJ$, let $\Xi_J$ denote the composite
				\[
				\begin{tikzcd}[ampersand replacement=\&]
					E \&\& {L(D(J))} \\
					{(S(\dot{A}))(X)} \& {(S(\dot{A}))(D(J))} \\
					{(S(\dot{B}))(X)} \&\& {(S(\dot{B}))(X)}
					\arrow["{\epsilon_J}", loose, from=1-1, to=1-3]
					\arrow["{\alpha_{\dot{A}}}"', loose, from=1-1, to=2-1]
					\arrow["{\eta^{D(J)}_{\dot{A}}}"', from=1-3, to=2-2]
					\arrow[""{name=0, anchor=center, inner sep=0}, "{\eta^{D(J)}_{\dot{B}}}", from=1-3, to=3-3]
					\arrow["{(S(\dot{A}))(\delta_J)}", from=2-1, to=2-2]
					\arrow["{S(f)_X}"', loose, from=2-1, to=3-1]
					\arrow["{S(f)_{\delta_J}}"{description}, draw=none, from=2-2, to=3-1]
					\arrow["{S(f)_{D(J)}}"', loose, from=2-2, to=3-3]
					\arrow["{(S(\dot{B}))(\delta_J)}"', from=3-1, to=3-3]
					\arrow["{\eta^{D(J)}_f}"', between={0.4}{0.7}, Rightarrow, from=2-2, to=0]
				\end{tikzcd}
				\]
				in $\bbK$. We have
				{
					\allowdisplaybreaks
					\begin{align*}
						\begin{tikzcd}[ampersand replacement=\&,column sep=tiny]
							E \&\& {L(D(J_2))} \\
							{(S(\dot{A}))(X)} \\
							{(S(\dot{A}))(D(J_1))} \& {(S(\dot{A}))(D(J_2))} \\
							{(S(\dot{B}))(D(J_1))} \&\& {(S(\dot{B}))(D(J_2))}
							\arrow["{\epsilon_{J_2}}", loose, from=1-1, to=1-3]
							\arrow["{\alpha_{\dot{A}}}"', loose, from=1-1, to=2-1]
							\arrow["{\eta^{D(J_2)}_{\dot{A}}}"', from=1-3, to=3-2]
							\arrow[""{name=0, anchor=center, inner sep=0}, "{\eta^{D(J_2)}_{\dot{B}}}", from=1-3, to=4-3]
							\arrow["{(S(\dot{A}))(\delta_{J_1})}"', from=2-1, to=3-1]
							\arrow[""{name=1, anchor=center, inner sep=0}, "{(S(\dot{A}))(\delta_{J_2})}", from=2-1, to=3-2]
							\arrow["{(S(\dot{A}))(D(j))}"', from=3-1, to=3-2]
							\arrow["{S(f)_{D(J_1)}}"', loose, from=3-1, to=4-1]
							\arrow["{S(f)_{D(J_2)}}", loose, from=3-2, to=4-3]
							\arrow[""{name=2, anchor=center, inner sep=0}, "{(S(\dot{B}))(D(j))}"', from=4-1, to=4-3]
							\arrow["{\scalebox{0.5}{$(S(\dot{A}))(\delta_j)$}}", between={0.2}{1}, Rightarrow, from=3-1, to=1]
							\arrow["{S(f)_{D(j)}}"{description}, draw=none, from=3-2, to=2]
							\arrow["{\eta^{D(J_2)}_f}", between={0.3}{0.7}, Rightarrow, from=3-2, to=0]
						\end{tikzcd}
						\hspace{-1.5em}
						\stackunder{\longref{}{eqt:epsilon_A}}{$=$}
						\hspace{-1.5em}
						\begin{tikzcd}[ampersand replacement=\&,column sep=tiny]
							E \&\& {L(D(J_2))} \\
							{L(D(J_1))} \\
							{(S(\dot{A}))(D(J_1))} \& {(S(\dot{A}))(D(J_2))} \\
							{(S(\dot{B}))(D(J_1))} \&\& {(S(\dot{B}))(D(J_2))}
							\arrow[""{name=0, anchor=center, inner sep=0}, "{\epsilon_{J_2}}", loose, from=1-1, to=1-3]
							\arrow["{\epsilon_{J_1}}"', loose, from=1-1, to=2-1]
							\arrow[""{name=1, anchor=center, inner sep=0}, "{\eta^{D(J_2)}_{\dot{A}}}"', from=1-3, to=3-2]
							\arrow[""{name=2, anchor=center, inner sep=0}, "{\eta^{D(J_2)}_{\dot{B}}}", from=1-3, to=4-3]
							\arrow["{L(D(j))}"{description}, from=2-1, to=1-3]
							\arrow[""{name=3, anchor=center, inner sep=0}, "{\eta^{D(J_1)}_{\dot{A}}}"', from=2-1, to=3-1]
							\arrow["{(S(\dot{A}))(D(j))}"', from=3-1, to=3-2]
							\arrow["{S(f)_{D(J_1)}}"', loose, from=3-1, to=4-1]
							\arrow["{S(f)_{D(J_2)}}", loose, from=3-2, to=4-3]
							\arrow[""{name=4, anchor=center, inner sep=0}, "{(S(\dot{B}))(D(j))}"', from=4-1, to=4-3]
							\arrow["{\eta^{D(j)}_{\dot{A}}}"{description}, draw=none, from=1, to=3]
							\arrow["{\epsilon_j}", between={0.4}{0.7}, Rightarrow, from=2-1, to=0]
							\arrow["{S(f)_{D(j)}}"{description}, draw=none, from=3-2, to=4]
							\arrow["{\eta^{D(J_2)}_f}", between={0.3}{0.7}, Rightarrow, from=3-2, to=2]
						\end{tikzcd}
						&
						\\
						\stackunder{$\eta_f$ is a modification}{$=$}
						\begin{tikzcd}[ampersand replacement=\&,column sep=tiny]
							E \&\&\& {L(D(J_2))} \\
							{L(D(J_1))} \\
							\\
							{(S(\dot{A}))(D(J_1))} \&\& {(S(\dot{B}))(D(J_1))} \& {(S(\dot{B}))(D(J_2))}
							\arrow[""{name=0, anchor=center, inner sep=0}, "{\epsilon_{J_2}}", loose, from=1-1, to=1-4]
							\arrow["{\epsilon_{J_1}}"', loose, from=1-1, to=2-1]
							\arrow["{\eta^{D(J_2)}_{\dot{B}}}", from=1-4, to=4-4]
							\arrow[""{name=1, anchor=center, inner sep=0}, "{L(D(j))}"', from=2-1, to=1-4]
							\arrow["{\eta^{D(J_1)}_{\dot{A}}}"', from=2-1, to=4-1]
							\arrow[""{name=2, anchor=center, inner sep=0}, "{\eta^{D(J_1)}_{\dot{B}}}", from=2-1, to=4-3]
							\arrow["{S(f)_{D(J_1)}}"', loose, from=4-1, to=4-3]
							\arrow["{(S(\dot{B}))(D(j))}"', from=4-3, to=4-4]
							\arrow["{\epsilon_j}", between={0.4}{0.7}, Rightarrow, from=2-1, to=0]
							\arrow["{\eta^{D(j)}_{\dot{B}}}"{description}, draw=none, from=1, to=4-4]
							\arrow["{\eta^{D(J_1)}_f}", between={0.3}{0.7}, Rightarrow, from=4-1, to=2]
						\end{tikzcd}.
					\end{align*}
				}
				In other words, $\{\Xi_J\}_{J\in\ob\calJ}$ satisfy the modification axiom. Now, by the $2$-dimensional universal property of $(S(\dot{B}))(X)$, there is a unique $2$-morphism $\alpha_f \colon \alpha_{\dot{B}} \to S(f)_X \cdot \alpha_{\dot{A}}$ such that
				\begin{equation}
					\label{eqt:alpha_f}
					\begin{tikzcd}[ampersand replacement=\&,column sep=small]
						\& E \& \\
						{(S(\dot{A}))(X)} \\
						\& {(S(\dot{B}))(X)} \\
						{(S(\dot{B}))(D(J_1))} \&\& {(S(\dot{B}))(D(J_2))}
						\arrow["{\alpha_{\dot{A}}}"', loose, from=1-2, to=2-1]
						\arrow[""{name=0, anchor=center, inner sep=0}, "{\alpha_{\dot{B}}}", loose, from=1-2, to=3-2]
						\arrow["{S(f)_X}"', loose, from=2-1, to=3-2]
						\arrow["{(S(\dot{B}))(\delta_{J_1})}"', from=3-2, to=4-1]
						\arrow[""{name=1, anchor=center, inner sep=0}, "{(S(\dot{B}))(\delta_{J_2})}", from=3-2, to=4-3]
						\arrow["{(S(\dot{B}))(D(j))}"', from=4-1, to=4-3]
						\arrow["{\exists!\alpha_f}", between={0.3}{0.7}, Rightarrow, dashed, from=2-1, to=0]
						\arrow["{(S(\dot{B}))(\delta_{j})}"', between={0.4}{0.6}, Rightarrow, from=4-1, to=1]
					\end{tikzcd}
					\hspace{-1em}
					=
					\begin{tikzcd}[ampersand replacement=\&,column sep=tiny]
						E \&\&\& {L(D(J_2))} \\
						{L(D(J_1))} \\
						\\
						{(S(\dot{A}))(D(J_1))} \&\& {(S(\dot{B}))(D(J_1))} \& {(S(\dot{B}))(D(J_2))}
						\arrow[""{name=0, anchor=center, inner sep=0}, "{\epsilon_{J_2}}", loose, from=1-1, to=1-4]
						\arrow["{\epsilon_{J_1}}"', loose, from=1-1, to=2-1]
						\arrow["{\eta^{D(J_2)}_{\dot{B}}}", from=1-4, to=4-4]
						\arrow[""{name=1, anchor=center, inner sep=0}, "{L(D(j))}"', from=2-1, to=1-4]
						\arrow["{\eta^{D(J_1)}_{\dot{A}}}"', from=2-1, to=4-1]
						\arrow[""{name=2, anchor=center, inner sep=0}, "{\eta^{D(J_1)}_{\dot{B}}}", from=2-1, to=4-3]
						\arrow["{S(f)_{D(J_1)}}"', loose, from=4-1, to=4-3]
						\arrow["{(S(\dot{B}))(D(j))}"', from=4-3, to=4-4]
						\arrow["{\epsilon_j}", between={0.4}{0.7}, Rightarrow, from=2-1, to=0]
						\arrow["{\eta^{D(j)}_{\dot{B}}}"{description}, draw=none, from=1, to=4-4]
						\arrow["{\eta^{D(J_1)}_f}", between={0.3}{0.7}, Rightarrow, from=4-1, to=2]
					\end{tikzcd},
				\end{equation}
				which is equal to
				\[
				\begin{tikzcd}[ampersand replacement=\&,column sep=small]
					\&\& E \&\& \\
					\& {(S(\dot{A}))(X)} \&\& {(S(\dot{B}))(X)} \\
					{(S(\dot{A}))(D(J_1))} \&\& {(S(\dot{B}))(D(J_1))} \&\& {(S(\dot{B}))(D(J_2))}
					\arrow["{\alpha_{\dot{A}}}"', loose, from=1-3, to=2-2]
					\arrow[""{name=0, anchor=center, inner sep=0}, "{\alpha_{\dot{B}}}", loose, from=1-3, to=2-4]
					\arrow[""{name=1, anchor=center, inner sep=0}, "{S(f)_X}"{description}, loose, from=2-2, to=2-4]
					\arrow["{(S(\dot{A}))(\delta_{J_1})}"', from=2-2, to=3-1]
					\arrow["{(S(\dot{B}))(\delta_{J_1})}"{description}, from=2-4, to=3-3]
					\arrow[""{name=2, anchor=center, inner sep=0}, "{(S(\dot{B}))(\delta_{J_2})}", from=2-4, to=3-5]
					\arrow[""{name=3, anchor=center, inner sep=0}, "{S(f)_{D(J_1)}}"', loose, from=3-1, to=3-3]
					\arrow["{(S(\dot{B}))(D(j))}"', from=3-3, to=3-5]
					\arrow["{\alpha_f}", between={0.3}{0.5}, Rightarrow, from=2-2, to=0]
					\arrow["{S(f)_{\delta_{J_1}}}"{description}, draw=none, from=1, to=3]
					\arrow["{(S(\dot{B}))(\delta_{j})}"', between={0.4}{0.6}, Rightarrow, from=3-3, to=2]
				\end{tikzcd}.
				\]
				We claim that $\{\alpha_{\dot{A}}\}_{\dot{A} \in \ob\bbD}$ gives rise to a marked-lax natural transformation $\alpha \colon \triangle(E) \to \mathrm{ev}_X \cdot S$ in $[\calJ \times \bbD, \bbK]_{s, \Sigma_\calJ \times \Sigma_\bbD, \Gamma_\calJ \times \Gamma_\bbD}$, whose $2$-component at $(j, f)$ is given by $\alpha_f$; if $\epsilon$ is in fact a dotted-lax natural transformation, then so is $\alpha$. Since there is only one non-identity morphism in $\bbD$, it is trivial that $\alpha_g \cdot \alpha_f = \alpha_{gf}$ for composable $g, f \in \mor\bbD$. The naturality of $\alpha$ follows from the fact that both $S(f)$ and $\eta^{D(J_1)}$ fulfil naturality. Finally, since only the identities in $\bbD$ are marked, so $\alpha_f = 1$ for marked $f$ in $\bbD$. It is also clear that if $\epsilon$ is a dotted-lax natural transformation, then $\alpha$ becomes a dotted-lax natural transformation.
				
				Altogether, we obtain $\alpha \colon \triangle(E) \to \mathrm{ev}_X \cdot S$ such that the composite
				$$\triangle(E) \xrightarrow{\alpha} \mathrm{ev}_X \cdot S \xrightarrow{(S(-))(\delta(-))} \mathrm{ev}_{D(-)} \cdot S$$
				of marked-lax natural transformations is equal to the composite
				$$\triangle(E) \xrightarrow{\epsilon} L\cdot D \xrightarrow{\eta^{D(-)}_{(-)}} \mathrm{ev}_{D(-)} \cdot S$$
				of marked-lax natural transformations. Now, recall that $\eta^X \colon \triangle(L(X)) \to \mathrm{ev}_X \cdot S$ is the limit cone for $\dotlim{l}{\mathrm{ev}_X \cdot S}$. Therefore, by the universal property of $\dotlim{l}{\mathrm{ev}_X \cdot S}$, there is a unique morphism $w \colon E \leadsto L(X)$ in $\bbK$ such that $\eta^X \cdot \triangle(w) = \alpha$, where $\triangle(w) \colon \triangle(E) \to \triangle(L(X))$ is the canonically induced marked-lax natural transformation. Note that if $\epsilon$ is indeed a dotted-lax natural transformation, then $w$ would be tight, and that $\triangle(w)$ would then be a dotted-lax natural transformation. We then attain
				\begin{align*}
					\eta^{D(-)} \cdot \epsilon &= (S(-))(\delta(-))\cdot \alpha &
					\\
					&= (S(-))(\delta(-))\cdot \eta^X \cdot \triangle(w) &
					\\
					&= \eta^{D(-)}_{(-)} \cdot (L \cdot \delta) \cdot \triangle(w) & \text{(by \longref{Equation}{eqt:equal_trans})}
				\end{align*}
				in  $[\calJ \times \bbD, \bbK]_{s, \Sigma_\calJ \times \Sigma_\bbD, \Gamma_\calJ \times \Gamma_\bbD}$. Then,  for any object $J$ of $\calJ$, and a morphism $f \colon \dot{A} \leadsto \dot{B}$ in $\bbD$, we have
				$$\eta^{D(J)}_f \cdot \epsilon_J = \eta^{D(J)}_f \cdot L(\delta_J) \cdot w;$$
				simiarly, for a morphism $j \colon J_1 \to J_2$ in $\calJ$, we have
				$$\eta^{D(J_2)}_f \cdot \epsilon_j = \eta^{D(J_2)}_f \cdot L(\delta_j) \cdot w.$$
				Now, by the universal property of the limit cone $\eta^{D(J)}$ for $\dotlim{l}{\mathrm{ev}_{D(J)} \cdot S}$, we conclude that
				$$\epsilon = (L\cdot\delta)\cdot \triangle(w)$$
				in $[\calJ, \bbK]_{s, \Sigma_\calJ, \Gamma_\calJ}$.
				
				
				In summary, we have just shown that given any arbitrary marked (resp. dotted)-lax natural transformation $\epsilon \colon \triangle(E) \to L\cdot D$ in $[\calJ, \bbK]_{s, \Sigma_\calJ, \Gamma_J}$, there exists a unique loose (resp. tight) morphism $w \colon E \leadsto L(X)$ in $\bbK$ such that $\epsilon = (L\cdot\delta)\cdot \triangle(w)$. This gives the $1$-dimensional unversal property for $L \cdot \delta$.
				
				\noindent\textbf{$\boldsymbol{2}$-dimensional universal property of $\boldsymbol{L \cdot \delta}$:}
				
				We proceed to establish the $2$-dimensional universal property of $L \cdot \delta$. 
				
				Suppose $\epsilon_1, \epsilon_2 \colon \triangle(E) \rightrightarrows L\cdot D$ are marked-lax natural transformations, which, according to the last paragraph, induce $\alpha_1, \alpha_2 \colon \triangle(E) \rightrightarrows \mathrm{ev}_X \cdot S$ and hence loose morphisms $w_1, w_2$ in $\bbK$, respectively, satisfying the desired equations. Let $\xi \colon \epsilon_1 \to \epsilon_2$ be a modification. Then, by the universal property of $(S(\dot{A}))(X)$, there is a unique $2$-morphism
				\[
				\begin{tikzcd}[ampersand replacement=\&]
					E \&\& {(S(\dot{A}))(X)}
					\arrow[""{name=0, anchor=center, inner sep=0}, "{{\alpha_2}_{\dot{A}}}"', curve={height=12pt}, loose, from=1-1, to=1-3, in = 210, out = 315]
					\arrow[""{name=1, anchor=center, inner sep=0}, "{{\alpha_1}_{\dot{A}}}", curve={height=-12pt}, loose, from=1-1, to=1-3, in = 150, out = 45]
					\arrow["{{\mathscr{A}}_{\dot{A}}}", between={0.3}{0.7}, Rightarrow, from=1, to=0]
				\end{tikzcd}
				\]
				in $\bbK$ such that
				{
					\allowdisplaybreaks
					\begin{align*}
						\begin{tikzcd}[ampersand replacement=\&,row sep=large]
							\& E \& \\
							\& {(S(\dot{A}))(X)} \\
							{(S(\dot{A}))(D(J_1))} \&\& {(S(\dot{A}))(D(J_2))}
							\arrow[""{name=0, anchor=center, inner sep=0}, "{{\alpha_2}_{\dot{A}}}", curve={height=-12pt}, loose, from=1-2, to=2-2]
							\arrow[""{name=1, anchor=center, inner sep=0}, "{{\alpha_1}_{\dot{A}}}"', curve={height=12pt}, loose, from=1-2, to=2-2]
							\arrow["{(S(\dot{A}))(\delta_{J_1})}"', from=2-2, to=3-1]
							\arrow[""{name=2, anchor=center, inner sep=0}, "{(S(\dot{A}))(\delta_{J_2})}", from=2-2, to=3-3]
							\arrow["{(S(\dot{A}))(D(j))}"', from=3-1, to=3-3]
							\arrow["{\exists!{\mathscr{A}_{\dot{A}}}}"', between={0.2}{0.8}, Rightarrow, dashed, from=1, to=0]
							\arrow["{(S(\dot{A}))(\delta_{j})}"', between={0.4}{0.6}, Rightarrow, from=3-1, to=2]
						\end{tikzcd}
						=
						\begin{tikzcd}[ampersand replacement=\&]
							\& E \& \\
							{L(D(J_1))} \&\& {L(D(J_2))} \\
							{(S(\dot{A}))(D(J_1))} \&\& {(S(\dot{A}))(D(J_2))}
							\arrow[""{name=0, anchor=center, inner sep=0}, "{{\epsilon_2}_{J_1}}"{description}, curve={height=-12pt}, loose, from=1-2, to=2-1]
							\arrow[""{name=1, anchor=center, inner sep=0}, "{{\epsilon_1}_{J_1}}"', curve={height=12pt}, loose, from=1-2, to=2-1]
							\arrow[""{name=2, anchor=center, inner sep=0}, "{{\epsilon_2}_{J_2}}", loose, from=1-2, to=2-3]
							\arrow[""{name=3, anchor=center, inner sep=0}, "{L(D(j))}"{description}, from=2-1, to=2-3]
							\arrow["{\eta^{D(J_1)}_{\dot{A}}}"', from=2-1, to=3-1]
							\arrow["{\eta^{D(j)}_{\dot{A}}}"{description}, draw=none, from=2-1, to=3-3]
							\arrow["{\eta^{D(J_2)}_{\dot{A}}}", from=2-3, to=3-3]
							\arrow["{(S(\dot{A}))(D(j))}"', from=3-1, to=3-3]
							\arrow["{\xi_{J_1}}"', between={0.2}{0.8}, Rightarrow, from=1, to=0]
							\arrow["{{\epsilon_2}_j}", between={0.4}{0.6}, Rightarrow, from=3, to=2]
						\end{tikzcd}
						&
						\\
						\stackunder{$\xi$ is a modification}{$=$}
						\begin{tikzcd}[ampersand replacement=\&]
							\& E \& \\
							{L(D(J_1))} \&\& {L(D(J_2))} \\
							{(S(\dot{A}))(D(J_1))} \&\& {(S(\dot{A}))(D(J_2))}
							\arrow["{{\epsilon_1}_{J_1}}"', loose, from=1-2, to=2-1]
							\arrow[""{name=0, anchor=center, inner sep=0}, "{{\epsilon_1}_{J_2}}"{description}, curve={height=12pt}, loose, from=1-2, to=2-3]
							\arrow[""{name=1, anchor=center, inner sep=0}, "{{\epsilon_2}_{J_2}}", curve={height=-12pt}, loose, from=1-2, to=2-3]
							\arrow["{L(D(j))}"{description}, from=2-1, to=2-3]
							\arrow["{\eta^{D(J_1)}_{\dot{A}}}"', from=2-1, to=3-1]
							\arrow["{\eta^{D(j)}_{\dot{A}}}"{description}, draw=none, from=2-1, to=3-3]
							\arrow["{\eta^{D(J_2)}_{\dot{A}}}", from=2-3, to=3-3]
							\arrow["{(S(\dot{A}))(D(j))}"', from=3-1, to=3-3]
							\arrow["{\xi_{J_2}}"', between={0.2}{0.8}, Rightarrow, from=0, to=1]
							\arrow["{{\epsilon_1}_j}", between={0.4}{0.6}, Rightarrow, from=2-1, to=0]
						\end{tikzcd}
						&
					\end{align*}
				}
				In particular,
				\begin{equation}
					\label{eqt:scrA_leg_1}
					(S(\dot{A}))(\delta_{J_1}) \cdot {\alpha_{2}}_{\dot{A}} = \eta^{D(J_1)}_{\dot{A}} \cdot \xi_{J_1},
				\end{equation}
				\begin{equation}
					\label{eqt:scrA_leg_2}
					(S(\dot{A}))(\delta_{J_2}) \cdot {\alpha_{2}}_{\dot{A}} = \eta^{D(J_2)}_{\dot{A}} \cdot \xi_{J_2}.
				\end{equation}
				We claim that $\{\mathscr{A}_{\dot{A}}\}_{\dot{A} \in \ob\bbD}$ assembles to a modification $\scrA \colon \alpha_1 \to \alpha_2$. In fact, for any $j \colon J_1 \to J_2$ in $\calJ$, we have
				{
					\allowdisplaybreaks
					\begin{align*}
						&
						\quad\quad\quad\quad\quad\quad\quad\quad\quad\quad
						\begin{tikzcd}[ampersand replacement=\&]
							\& E \& \\
							{(S(\dot{A}))(X)} \\
							\& {(S(\dot{B}))(X)} \\
							{(S(\dot{B}))(D(J_1))} \&\& {(S(\dot{B}))(D(J_2))}
							\arrow[""{name=0, anchor=center, inner sep=0}, "{{\alpha_2}_{\dot{A}}}"{description}, curve={height=-12pt}, loose, from=1-2, to=2-1]
							\arrow[""{name=1, anchor=center, inner sep=0}, "{{\alpha_1}_{\dot{A}}}"', curve={height=12pt}, loose, from=1-2, to=2-1]
							\arrow[""{name=2, anchor=center, inner sep=0}, "{{\alpha_2}_{\dot{B}}}", loose, from=1-2, to=3-2]
							\arrow["{S(f)_X}"', loose, from=2-1, to=3-2]
							\arrow["{(S(\dot{B}))(\delta_{J_1})}"', from=3-2, to=4-1]
							\arrow[""{name=3, anchor=center, inner sep=0}, "{(S(\dot{B}))(\delta_{J_2})}", from=3-2, to=4-3]
							\arrow["{(S(\dot{B}))(D(j))}"', from=4-1, to=4-3]
							\arrow["{\mathscr{A}_{\dot{A}}}"', between={0.2}{0.8}, Rightarrow, from=1, to=0]
							\arrow["{{\alpha_2}_f}"', between={0.3}{0.7}, Rightarrow, from=2-1, to=2]
							\arrow["{(S(\dot{B}))(\delta_{j})}"', between={0.4}{0.6}, Rightarrow, from=4-1, to=3]
						\end{tikzcd}
						\\
						&=
						\begin{tikzcd}[ampersand replacement=\&]
							\&\& E \&\& \\
							\\
							\& {(S(\dot{A}))(X)} \&\& {(S(\dot{B}))(X)} \\
							{(S(\dot{A}))(D(J_1))} \&\& {(S(\dot{B}))(D(J_1))} \&\& {(S(\dot{B}))(D(J_2))}
							\arrow[""{name=0, anchor=center, inner sep=0}, "{{\alpha_2}_{\dot{A}}}"{description}, curve={height=-12pt}, loose, from=1-3, to=3-2]
							\arrow[""{name=1, anchor=center, inner sep=0}, "{{\alpha_1}_{\dot{A}}}"', curve={height=12pt}, loose, from=1-3, to=3-2]
							\arrow[""{name=2, anchor=center, inner sep=0}, "{{\alpha_2}_{\dot{B}}}", loose, from=1-3, to=3-4]
							\arrow[""{name=3, anchor=center, inner sep=0}, "{S(f)_X}"{description}, loose, from=3-2, to=3-4]
							\arrow["{(S(\dot{A}))(\delta_{J_1})}"', from=3-2, to=4-1]
							\arrow["{S(f)_{\delta_{J_1}}}"{description}, draw=none, from=3-2, to=4-3]
							\arrow["{(S(\dot{B}))(\delta_{J_1})}"{description}, from=3-4, to=4-3]
							\arrow[""{name=4, anchor=center, inner sep=0}, "{(S(\dot{B}))(\delta_{J_2})}", from=3-4, to=4-5]
							\arrow["{S(f)_{D(J_1)}}"', loose, from=4-1, to=4-3]
							\arrow["{(S(\dot{B}))(D(j))}"', from=4-3, to=4-5]
							\arrow["{\mathscr{A}_{\dot{A}}}"', between={0.2}{0.8}, Rightarrow, from=1, to=0]
							\arrow["{{\alpha_2}_f}"', between={0.3}{0.7}, Rightarrow, from=3, to=2]
							\arrow["{(S(\dot{B}))(\delta_{j})}"', between={0.4}{0.6}, Rightarrow, from=4-3, to=4]
						\end{tikzcd}
						\\
						&
						\stackunder{\longref{}{eqt:alpha_f}\&\longref{}{eqt:scrA_leg_1}}{$=$}
						\quad\quad\quad
						\begin{tikzcd}[ampersand replacement=\&]
							E \&\&\& {L(D(J_2))} \\
							{L(D(J_1))} \\
							\\
							{(S(\dot{A}))(D(J_1))} \&\& {(S(\dot{B}))(D(J_1))} \& {(S(\dot{B}))(D(J_2))}
							\arrow[""{name=0, anchor=center, inner sep=0}, "{{\epsilon_2}_{J_2}}", loose, from=1-1, to=1-4]
							\arrow[""{name=1, anchor=center, inner sep=0}, "{{\epsilon_2}_{J_1}}", curve={height=-12pt}, loose, from=1-1, to=2-1]
							\arrow[""{name=2, anchor=center, inner sep=0}, "{{\epsilon_1}_{J_1}}"', curve={height=12pt}, loose, from=1-1, to=2-1]
							\arrow["{\eta^{D(J_2)}_{\dot{B}}}", from=1-4, to=4-4]
							\arrow[""{name=3, anchor=center, inner sep=0}, "{L(D(j))}"', from=2-1, to=1-4]
							\arrow["{\eta^{D(J_1)}_{\dot{A}}}"', from=2-1, to=4-1]
							\arrow[""{name=4, anchor=center, inner sep=0}, "{\eta^{D(J_1)}_{\dot{B}}}", from=2-1, to=4-3]
							\arrow["{S(f)_{D(J_1)}}"', loose, from=4-1, to=4-3]
							\arrow["{(S(\dot{B}))(D(j))}"', from=4-3, to=4-4]
							\arrow["{\xi_{J_1}}"', between={0.2}{0.8}, Rightarrow, from=2, to=1]
							\arrow["{{\epsilon_2}_j}", between={0.4}{0.7}, Rightarrow, from=2-1, to=0]
							\arrow["{\eta^{D(j)}_{\dot{B}}}"{description}, draw=none, from=3, to=4-4]
							\arrow["{\eta^{D(J_1)}_f}", between={0.3}{0.7}, Rightarrow, from=4-1, to=4]
						\end{tikzcd}
						\\
						&
						\stackunder{$\xi$ is a modification}{$=$}
						\begin{tikzcd}[ampersand replacement=\&]
							E \&\&\& {L(D(J_2))} \\
							{L(D(J_1))} \\
							\\
							{(S(\dot{A}))(D(J_1))} \&\& {(S(\dot{B}))(D(J_1))} \& {(S(\dot{B}))(D(J_2))}
							\arrow[""{name=0, anchor=center, inner sep=0}, "{{\epsilon_2}_{J_2}}", shift left=2, curve={height=-12pt}, loose, from=1-1, to=1-4]
							\arrow[""{name=1, anchor=center, inner sep=0}, "{{\epsilon_1}_{J_2}}"{description}, shift left=2, curve={height=12pt}, loose, from=1-1, to=1-4]
							\arrow["{{\epsilon_1}_{J_1}}"', loose, from=1-1, to=2-1]
							\arrow["{\eta^{D(J_2)}_{\dot{B}}}", from=1-4, to=4-4]
							\arrow[""{name=2, anchor=center, inner sep=0}, "{L(D(j))}"', from=2-1, to=1-4]
							\arrow["{\eta^{D(J_1)}_{\dot{A}}}"', from=2-1, to=4-1]
							\arrow[""{name=3, anchor=center, inner sep=0}, "{\eta^{D(J_1)}_{\dot{B}}}", from=2-1, to=4-3]
							\arrow["{S(f)_{D(J_1)}}"', loose, from=4-1, to=4-3]
							\arrow["{(S(\dot{B}))(D(j))}"', from=4-3, to=4-4]
							\arrow["{\xi_{J_2}}"', between={0.2}{0.8}, Rightarrow, from=1, to=0]
							\arrow["{{\epsilon_1}_j}", between={0.4}{0.7}, Rightarrow, from=2-1, to=1]
							\arrow["{\eta^{D(j)}_{\dot{B}}}"{description}, draw=none, from=2, to=4-4]
							\arrow["{\eta^{D(J_1)}_f}", between={0.3}{0.7}, Rightarrow, from=4-1, to=3]
						\end{tikzcd}
						\\
						&
						\stackunder{\longref{}{eqt:scrA_leg_2}\&\longref{}{eqt:alpha_f}}{$=$}
						\quad\quad\quad\quad
						\begin{tikzcd}[ampersand replacement=\&]
							\& E \& \\
							{(S(\dot{A}))(X)} \\
							\& {(S(\dot{B}))(X)} \\
							{(S(\dot{B}))(D(J_1))} \&\& {(S(\dot{B}))(D(J_2))}
							\arrow["{{\alpha_1}_{\dot{A}}}"', loose, from=1-2, to=2-1]
							\arrow[""{name=0, anchor=center, inner sep=0}, "{{\alpha_1}_{\dot{B}}}"{description}, loose, from=1-2, to=3-2]
							\arrow[""{name=1, anchor=center, inner sep=0}, "{{\alpha_2}_{\dot{B}}}"{pos=0.52, name=Z}, shift left=3, curve={height=-24pt}, loose, from=1-2, to=3-2]
							\arrow[""{name=1p, anchor=center, inner sep=0}, phantom, from=1-2, to=3-2, start anchor=center, end anchor=center, shift left=3, curve={height=-24pt}]
							\arrow["{S(f)_X}"', loose, from=2-1, to=3-2]
							\arrow["{(S(\dot{B}))(\delta_{J_1})}"', from=3-2, to=4-1]
							\arrow[""{name=2, anchor=center, inner sep=0}, "{(S(\dot{B}))(\delta_{J_2})}", from=3-2, to=4-3]
							\arrow["{(S(\dot{B}))(D(j))}"', from=4-1, to=4-3]
							\arrow["{\mathscr{A}_{\dot{B}}}"'{pos=0.5}, between={0.2}{0.8}, Rightarrow, from=0, to=Z]
							\arrow["{{\alpha_1}_f}"', between={0.3}{0.7}, Rightarrow, from=2-1, to=0]
							\arrow["{(S(\dot{B}))(\delta_{j})}"', between={0.4}{0.6}, Rightarrow, from=4-1, to=2]
						\end{tikzcd}.
					\end{align*}
				}
				Therefore, by the $2$-dimensional universal property of $(S(\dot{B}))(X)$, we conclude that
				\begin{equation}
					\label{eqt:equal_modif}
					\begin{tikzcd}[ampersand replacement=\&,row sep=large]
						\& E \& \\
						\\
						{(S(\dot{A}))(X)} \&\& {(S(\dot{B}))(X)}
						\arrow[""{name=0, anchor=center, inner sep=0}, "{{\alpha_2}_{\dot{A}}}"{description}, curve={height=-12pt}, loose, from=1-2, to=3-1]
						\arrow[""{name=1, anchor=center, inner sep=0}, "{{\alpha_1}_{\dot{A}}}"', curve={height=12pt}, loose, from=1-2, to=3-1]
						\arrow[""{name=2, anchor=center, inner sep=0}, "{{\alpha_2}_{\dot{B}}}", loose, from=1-2, to=3-3]
						\arrow["{S(f)_X}"', loose, from=3-1, to=3-3]
						\arrow["{\mathscr{A}_{\dot{A}}}"', between={0.2}{0.8}, Rightarrow, from=1, to=0]
						\arrow["{{\alpha_2}_f}"', between={0.3}{0.6}, Rightarrow, from=3-1, to=2]
					\end{tikzcd}
					=
					\begin{tikzcd}[ampersand replacement=\&,row sep=large]
						\& E \& \\
						\\
						{(S(\dot{A}))(X)} \&\& {(S(\dot{B}))(X)}
						\arrow["{{\alpha_1}_{\dot{A}}}"', loose, from=1-2, to=3-1]
						\arrow[""{name=0, anchor=center, inner sep=0}, "{{\alpha_2}_{\dot{B}}}", curve={height=-12pt}, loose, from=1-2, to=3-3]
						\arrow[""{name=1, anchor=center, inner sep=0}, "{{\alpha_1}_{\dot{B}}}"{description}, curve={height=12pt}, loose, from=1-2, to=3-3]
						\arrow["{S(f)_X}"', loose, from=3-1, to=3-3]
						\arrow["{\mathscr{A}_{\dot{B}}}"', between={0.2}{0.8}, Rightarrow, from=1, to=0]
						\arrow["{{\alpha_1}_{f}}"', between={0.3}{0.7}, Rightarrow, from=3-1, to=1]
					\end{tikzcd},
				\end{equation}
				and so $\{\mathscr{A}_{\dot{A}}\}_{\dot{A} \in \ob\bbD}$ satisfies the modification axiom.
				
				In short, we obtain a modification $\mathscr{A} \colon \alpha_1 \to \alpha_2$ such that
				$$(S(-))(\delta(-)) \cdot \mathscr{A} = \eta^{D(-)}_{(-)} \cdot \xi.$$ Since $\eta^X \colon \triangle(L(X)) \to \mathrm{ev}_X \cdot S$ is the limit cone for $\dotlim{l}{\mathrm{ev}_X \cdot S}$. Thus, by the $2$-dimensional universal property of $\dotlim{l}{\mathrm{ev}_X \cdot S}$, there is a unique $2$-morphism
				\[
				\begin{tikzcd}[ampersand replacement=\&]
					E \&\& {L(X)}
					\arrow[""{name=0, anchor=center, inner sep=0}, "{w_2}"', curve={height=12pt}, loose, from=1-1, to=1-3, in = 210, out = 315]
					\arrow[""{name=1, anchor=center, inner sep=0}, "{w_1}", curve={height=-12pt}, loose, from=1-1, to=1-3, in = 150, out = 45]
					\arrow["\omega", between={0.3}{0.7}, Rightarrow, from=1, to=0]
				\end{tikzcd}
				\]
				in $\bbK$ such that $\eta^X \cdot \triangle(\omega) = \mathscr{A}$, where $\triangle(\omega) \colon \triangle(w_1) \to \triangle(w_2)$ is the canonically induced modification. Consequently, we get a series of equalities
				\begin{align*}
					\eta^{D(-)} \cdot \xi &= (S(-))(\delta(-))\cdot \mathscr{A} &
					\\
					&= (S(-))(\delta(-))\cdot \eta^X \cdot \triangle(\omega) & \text{(by \longref{Equation}{eqt:equal_modif})}
					\\
					&= \eta^{D(-)}_{(-)} \cdot (L \cdot \delta) \cdot \triangle(\omega) & \text{(by \longref{Equation}{eqt:equal_trans})}
				\end{align*}
				between different modifications in  $[\calJ \times \bbD, \bbK]_{s, \Sigma_\calJ \times \Sigma_\bbD, \Gamma_\calJ \times \Gamma_\bbD}$.
				
				As a result, we have just shown that given any arbitrary modification $\xi \colon \epsilon_1 \to \epsilon_2$, where $\epsilon_1, \epsilon_2 \colon \triangle(E) \to L\cdot D$ are marked-lax natural transformations in $[\calJ, \bbK]_{s, \Sigma_\calJ, \Gamma_\calJ}$, there exists a unique $2$-morphism $\omega \colon w_1 \to w_2$ in $\bbK$ such that $\xi = (L \cdot \delta) \cdot \triangle(\omega)$.
				
				Henceforth, $L \cdot \delta \colon \triangle(L(X)) \to L \cdot D$ serves as the limit cone, and so $L$ is in fact an $\F$-model.
			\end{proof}
			
			With our main theorem in the previous section, we are able to find out limits that lift to $\mathmybb{Mod}_{s, w}(\mathmybb{T}, \mathmybb{K})$. These limits are \emph{$w$-rigged limits}, first introduced in \cite[\S5]{LS:2012}.
			
			\begin{rk}
				As explained in \longref{Remark}{rk:tight_cones_1}, in many interesting and important examples, the weighted cones are tight. Besides, if we do not assume that the cones are tight, we may encounter technical difficulties. In the proof of \longref{Proposition}{pro:colax_lim_arrow}, if $\calJ$ is not a $2$-category, i.e., a chordate $\F$-category, then the $\F$-functor $\mathrm{ev}_{D(-)}\cdot S$ is not well-defined; also, if we do not assume that every object of $\calJ$ is dotted, which means the projection $\delta_J$ is not always tight, then the composite $\Xi_J$ cannot be well-defined because the loose colax natural transformation $S(f)$ has its $2$-component at $\delta_J$ for an arbitrary object $J$ of $\calJ$ given by
				\[
				\begin{tikzcd}[ampersand replacement=\&]
					{(S(\dot{A}))(X)} \& {(S(\dot{A}))(D(J))} \\
					{(S(\dot{B}))(X)} \& {(S(\dot{B}))(X)}
					\arrow["{(S(\dot{A}))(\delta_J)}", from=1-1, to=1-2]
					\arrow["{S(f)_X}"', loose, from=1-1, to=2-1]
					\arrow["{S(f)_{\delta_J}}", between={0.3}{0.7}, Rightarrow, from=1-2, to=2-1]
					\arrow["{S(f)_{D(J)}}", loose, from=1-2, to=2-2]
					\arrow["{(S(\dot{B}))(\delta_J)}"', from=2-1, to=2-2]
				\end{tikzcd},
				\]
				which is pointing to the opposite direction of $\eta^{D(J)}_f$.
			\end{rk}
			
			\begin{coroll}
				\label{cor:limits}
				Let $\mathmybb{T}$ be an enhanced limit $2$-sketch with tight weighted cones, and $\bbK$ be a locally presentable enhanced $2$-category. The restriction
				$$\mathmybb{Mod}_{s, w}(\mathmybb{T}, \mathmybb{K}) \to \mathmybb{Mod}_{s, w}(\mathccal{T}_\tau, \mathmybb{K})$$
				creates all $w$-rigged limits.
				%
			\end{coroll}
			
			\begin{proof}
				Since $\bbK$ is complete, and by \longref{Proposition}{pro:V_models_monadic}, the inclusion $\FMod{}(\mathccal{T}_\tau, \bbK) \hookrightarrow [\mathccal{T}_\tau, \bbK]$ is reflective, thus, $\FMod{s, w}(\mathccal{T}_\tau, \bbK) = \FMod{}(\mathccal{T}_\tau, \bbK)$ is also complete.
				
				By \cite[Theorem 5.10]{LS:2012}, the underlying $\F$-functor $U_{s, w}$ in \longref{Diagram}{diag:equiv} creates all $w$-rigged limits.
				
				Since $i^*_{s, w} \colon \mathmybb{Mod}_{s, w}(\mathmybb{T}, \mathmybb{K}) \to \mathmybb{Mod}_{s, w}(\mathccal{T}_\tau, \mathmybb{K})$ is the composite
				$$\FMod{s, w}(\mathmybb{T}, \mathmybb{K}) \xrightarrow{\simeq} \FTAlg_{s, w} \xrightarrow{U_{s, w}} \FMod{}(\mathccal{T}_\tau, \mathmybb{K})$$
				of an equivalence and $U_{s, w}$, we conclude that $i^*_{s, w}$ creates all $w$-rigged limits.
			\end{proof}
			
			\begin{example}[{\cite{ABK:2024}}]
				Monoidal categories with strict and $w$-monoidal functors, double categories with strict and $w$-double functors, double (op)fibrations with strict and $w$-morphisms of fibrations, all admit $w$-rigged limits, as they are models of enhanced $2$-sketches with tight weighted cones to a complete enhanced $2$-category.
			\end{example}

			\bibliographystyle{alpha}
			
		\end{document}